\documentclass[a4paper,twoside,11pt]{article}
\usepackage[T1]{fontenc}
\usepackage{amsfonts}
\usepackage{amsmath}
\usepackage{amssymb}
\usepackage{geometry}
\usepackage{sectsty}
\usepackage{lineno}
\usepackage{theorem}
\usepackage{fancyhdr}
\usepackage{enumitem}
\usepackage{graphicx,subfigure}
\usepackage{multirow,multicol}
\usepackage{url}
\usepackage{doi}
\usepackage{array}
\usepackage[table]{xcolor}
\usepackage{geometry}
\usepackage[normalem]{ulem}
\usepackage{hyperref}
\usepackage[numbers]{natbib}
\usepackage[normalem]{ulem}
\usepackage{bbm}
\usepackage{dsfont}
\usepackage{xspace}
\usepackage{fancyhdr}

\setlist{nosep}

\newcolumntype{L}{>{$}l<{$}}
\newcolumntype{C}{>{$}c<{$}}
\definecolor{lgray}{gray}{0.8}
\providecommand{\keywords}[1]{{\textbf{Key words and phrases:}} #1}
\providecommand{\subjclass}[1]{{\textbf{MSC 2010 subject classifications:}} #1}
\theoremstyle{plain}
\newtheorem{theorem}{Theorem}[section]

\newtheorem{definition}[theorem]{Definition}

\newtheorem{lemma}[theorem]{Lemma}

\newtheorem{proposition}[theorem]{Proposition}
\newtheorem{remark}[theorem]{Remark}

\newenvironment{proof}[1][Proof]{\noindent\textbf{#1.} }{\ \rule{0.5em}{0.5em}}
\hypersetup{
    colorlinks,    citecolor=blue,    filecolor=blue,    linkcolor=blue,    urlcolor=blue
}

\DeclareMathOperator{\pen}{pen}
\DeclareMathOperator{\crit}{crit}
\DeclareMathOperator{\leb}{Leb}
\DeclareMathOperator{\card}{Card}

\begin{document}
\title{\textbf{Finite sample improvement of Akaike's
Information Criterion}}
\author{\textbf{Adrien Saumard} \hfill \textbf{Fabien Navarro} \\
CREST, CNRS-ENSAI, Universit\'{e} Bretagne Loire}
\date{}
\maketitle

\begin{abstract}
We emphasize that it is possible to improve the principle of unbiased risk estimation for model selection by addressing excess risk deviations in the design of penalization procedures. Indeed, we propose a modification of Akaike's Information Criterion that
avoids overfitting, even when the sample size is small. We call this
correction an over-penalization procedure.  As proof of concept, we show the nonasymptotic optimality of our histogram selection procedure in density estimation
by establishing sharp oracle inequalities for the Kullback-Leibler divergence. One of the main features of our theoretical results is that they include the estimation of unbounded log-densities. 
To do so, we prove several analytical and probabilistic lemmas that are of independent interest. In an experimental study, we also demonstrate state-of-the-art performance of our over-penalization criterion for bin size selection, in particular outperforming AICc procedure.

\medskip

\noindent \subjclass 62G07, 62G09, 62G10.

\noindent \keywords model selection, testing, AIC, maximum
likelihood, density estimation, bin size, over-penalization, AIC corrected, small sample
size.
\end{abstract}

\section{Introduction}

Since its introduction by Akaike in the early seventies~\cite{Akaike:73}, the celebrated Akaike's Information Criterion (AIC) has been an essential tool for the statistician and its use is almost systematic in problems of model selection and estimator selection for prediction. By choosing among estimators or models constructed from finite degrees of freedom,  the AIC recommends more specifically to maximize the  log-likelihood of the estimators penalized by their corresponding degrees of freedom. This procedure has
found pathbreaking applications in density estimation, regression, time series or neural network analysis, to name a few (\cite{ClaeskensHjort:08}). Because of its simplicity and negligible computation cost---whenever the estimators are given---, it is also far from outdated and continues to serve as one of the most useful devices for model selection in high-dimensional statistics. For instance, it can be used to efficiently tune the Lasso (\cite{ZouHastieTib:07}).

Any substantial and principled improvement of AIC is likely to have a significant impact on the practice of model choices and we bring in this paper an efficient and theoretically grounded solution to the problem of overfitting that can occur when using AIC on small to medium sample sizes.

The fact that AIC tends to be unstable and therefore perfectible in the case of small sample sizes is well known to practitioners and and has long been noted. Suguira~\cite{Sugiura:78} and Hurvich and Tsai~\cite{HurTsai:89} have proposed the so-called AICc (for AIC corrected), which tends to penalize more than AIC. However, the derivation of AICc comes from an asymptotic analysis where the dimension of the models are considered fixed relative to the sample size. In fact, such an assumption does not fit the usual practice of model selection, where the largest models are of dimensions close to the sample size. Another drawback related to AICc is that it has been legitimated through a mathematical analysis only in the linear regression model and for autoregressive models (\cite{ClaeskensHjort:08}). However, to the best of our knowledge, outside these frameworks, there is no theoretical ground for the use of AICc.

Building on considerations from the general nonasymptotic theory of model selection developed during the nineties (see for instance \cite{BarBirMassart:99} and \cite{Massart:07}) and in particular on Castellan's analysis~\cite{Castellan:03}, Birg\'{e} and Rozenholc~\cite{BirRozen:06} have considered an AIC modification specifically designed for the selection of the bin size in histogram selection for density estimation. Indeed, results of \cite{Castellan:03}---and more generally results of \cite{BarBirMassart:99}---advocate to take into account in the design of penalty the number of models to be selected. The importance of the cardinality of the collection of models for model selection is in fact a very general phenomenon and one of the main outcomes of the nonasymptotic model selection theory. In the bin size selection problem, this corresponds to adding a small amount to AIC. Unfortunately, the theory does not specify uniquely the term to be added to AIC. On the contrary, infinitely many corrections are accepted by the theory and in order to choose a good one, intensive experiments were conducted in~\cite{BirRozen:06}. The resulting AIC correction therefore always has the disadvantage of being specifically designed for the task on which it has been tested.

We propose a general approach that goes beyond the limits of the unbiased risk estimation principle. The latter principle is indeed at the core of Akaike's model selection procedure and is more generally the main model selection principle,  which underlies procedures such as Stein's Unbiased Risk Estimator (SURE, \cite{MR630098}) or cross-validation (\cite{ArlotCelisse:10}). We point out that it is more efficient to estimate a quantile of the risk of the estimators---the level of the quantile depending on the size of the collection of models---than its mean. We also develop a (pseudo-)testing point of view, where we find that unbiased risk estimation does not in general allow to control the sizes of the considered tests. This is thus a new, very general model selection principle that we put forward and formalize. We call it an over-penalization procedure, because it systematically involves adding small terms to traditional penalties such as AIC.

Our procedure consists in constructing an estimator from multiple pseudo-tests, built on some random events. From this perspective, our work shares strong connections recent advances in robust estimation by designing estimators from tests (\cite{MR2449129,MR2219712,MR3186748,BarBirgSart:17,lecue2017learning}). However, our focus and perspectives are significantly different from this existing line of works. Indeed, estimators such as T-estimators (\cite{MR2219712}) or $\rho$-estimators (\cite{BarBirgSart:17}) are quasi-universal estimators in the sense that they have very strong statistical guarantees, but they have the drawback to be very difficult---if feasible---to compute. In particular, \cite{MR2449129} also builds estimators from frequency histograms, but to our knowledge no implementation of such estimators exists and it seems to be an open question whether a polynomial time algorithm can effectively compute them or not. Here, we rather keep the tractability of AIC procedure, but we don't look particularly at robustness properties (for instance against outliers). We focus on improving AIC in the nonasymptotic regime.

In addition, it is worth noting that several authors have examined some links between multiple testing and model selection, in particular by making some modifications to classical criteria (see for instance \cite[Chapter 7]{MR3184277}). But these lines of research differ significantly from our approach. Indeed, the first and main use in the literature of multiple testing point of view for model selection concerns variable selection, i.e. the identification of models, particularly in the context of linear regression. (\cite{MR921623, MR2280619,MR2183942,MR2823520,MR2281879,MR2668704}). It consists in considering simultaneously the testing of each variable being equal to zero or not. Instead, we consider model selection from a predictive perspective and do not focus on model identification. It should also be noted that multiple tests may be considered after model selection or at the same time as selective inference (\cite{MR2395714,MR3174645,MR2156820}), but these questions are not directly related to the scope of our paper. 

Lets us now detail our contributions.

\begin{itemize}
\item We propose a general formulation of the model selection task for prediction in terms of a (pseudo-)test procedure (understood in a non classical way which will be detailed in the Section~\ref{ssection_testing}), thus establishing a link between two major topics of contemporary research. In particular, we propose a generic property that the pseudo-tests collection should satisfy in order to ensure an oracle inequality for the selected model. We call this property the ``transitivity property'' and show that it generalizes penalization procedures together with T-estimators and $\rho $-estimators.

\item Considering the problem of density estimation by selecting a histogram, we prove a sharp and fully nonasymptotic oracle inequality for our procedure. Indeed, we describe a control of Kullback-Leibler (KL) divergence---also called excess risk---of the selected histogram as soon as we have one observation. We emphasize that this very strong feature may not be possible when considering AIC. We also stressed that up to our knowledge, our oracle inequality is the first nonasymptotic result comparing
the KL divergence of the selected model to the KL divergence of the oracle in an unbounded setting. Indeed, oracle inequalities in density estimation
are generally expressed in terms of Hellinger distance---which is much easier to handle than the KL divergence, because it is bounded---for the selected model.

\item In order to prove our oracle inequality, we improve upon the previously best known concentration inequality for the chi-square statistics (Castellan \cite{Castellan:03}, Massart \cite{Massart:07}) and this allows us to gain an order of magnitude in the control of the deviations of the excess risks of the estimators. Our result on the chi-square statistics is general and of independent interest.

\item We also prove new Bernstein-type concentration inequalities for log-densities that are unbounded. Again, these probabilistic results, which are naturally linked to information theory, are general and of independent interest.

\item We generalize previous results of Barron and Sheu \cite{BarSheu:91} regarding the existence of margin relations in maximum likelihood estimation (MLE). Indeed, related results of \cite{BarSheu:91} where established under boundedness of the log-densities and we extend them to unbounded log-densities with moment conditions.

\item Finally, from a practical point of view, we bring a nonasymptotic improvement of AIC that has, in its simplest form, the same computational cost as AIC. Furthermore, we show that our over-penalization procedure largely outperforms AIC on small and medium sample sizes, but also surpasses existing AIC corrections such as AICc or Birg\'{e}-Rozenholc's procedure.
\end{itemize}

Let us end this introduction by detailing the organization of the paper.

We present our over-penalization procedure in Section~\ref{section_framework_and_notations_MLE}. More precisely, we detail in Sections~\ref{section_max_vrais_hist_MLE} and \ref{section_histo} our model selection framework related to MLE via histograms. Then in Section~\ref{section_overpen} we define formally over-penalization procedures and highlight their generality. We explain the ideas underlying over-penalization from three different angles: estimation of the ideal penalty, pseudo-testing and a graphical point of view.

Section~\ref{section_results} is devoted to statistical guarantees related to over-penalization. In particular, as concentration properties of the excess risks are at the heart of the design of an over-penalization, we detail them in Section~\ref{section_risks_bounds_MLE}.  We then deduce a general and sharp oracle inequality in Section~\ref{section_results_bounded_setting} and highlight the theoretical advantages compared to an AIC analysis.

New mathematical tools of a probabilistic and analytical nature and of independent interest are presented in Section~\ref{section_prob_tools}. Section~\ref{section_expe} contains the experiments, with detailed practical procedures. We consider two different practical variations of over-penalization and compare them with existing penalization procedures. The superiority of our method is particularly transparent.

The proofs are gathered in Section~\ref{section_proofs_MLE}, as well as in the Appendix~\ref{sec:appendix} which provides further theoretical developments that extend the description of our over-penalization procedure. 

\section{Statistical Framework and Notations\label%
{section_framework_and_notations_MLE}}

The over-penalization procedure, described in Section~\ref{section_overpen}, is legitimated at a heuristic level  within a generic M-estimator selection framework. We put to emphasis in Section~\ref{section_max_vrais_hist_MLE} on maximum likelihood estimation (MLE) since, as proof of concept for our over-penalization procedure, our theoretical and experimental results will address the case of bin size selection for maximum likelihood histogram selection in density estimation. In order to be able to discuss in Section~\ref{section_overpen} the generality of our approach in an M-estimation setting, our presentation of MLE brings notations which extend directly to M-estimation with a general contrast.

\subsection{Maximum Likelihood Density Estimation}
\label{section_max_vrais_hist_MLE}

We are given $n$ independent observations $\left( \xi_{1},\ldots,\xi_{n}\right)$ with unknown common distribution $P$ on a
measurable space $\left( \mathcal{Z},\mathcal{T}\right) $. We assume that there exists a known probability measure $\mu$ on $\left(\mathcal{Z},\mathcal{T}\right) $ such that $P$ admits a density $f_{\ast }$ with respect
to $\mu$: $f_{\ast}=dP/d\mu $. Our goal is to estimate the density $f_{\ast
} $.

For an integrable function $f$ on $\mathcal{Z}$, we set $Pf=P\left( f\right)
=\int_{\mathcal{Z}}f\left( z\right) dP\left( z\right) $ and $\mu f=\mu
\left( f\right) =\int_{\mathcal{Z}}f\left( z\right) d\mu \left( z\right) $.
If $P_{n}=1/n\sum_{i=1}^{n}\delta _{\xi _{i}}$ denotes the empirical
distribution associated to the sample $\left( \xi _{1},\ldots ,\xi
_{n}\right) $, then we set $P_{n}f=P_{n}\left( f\right)
=1/n\sum_{i=1}^{n}f\left( \xi _{i}\right) $. Moreover, taking the
conventions $\ln 0=-\infty $, $0\ln 0=0$ and defining the positive part as $%
\left( x\right) _{+}=x\vee 0$, we set 
\begin{equation*}
\mathcal{S}=\left\{ f:\mathcal{Z}\longrightarrow \mathbb{R}_{+}\text{; }%
\int_{\mathcal{Z}}fd\mu =1\text{ and }P\left( \ln f\right) _{+}<\infty
\right\} \text{ .}
\end{equation*}%
We assume that the unknown density $f_{\ast }$ belongs to $\mathcal{S}$.

Note that since $P\left(\ln f_{\ast }\right) _{-}=\int f_{\ast }\ln f_{\ast}%
\mathbbm{1} _{f_{\ast }\leq 1}d\mu <\infty $, the fact that $f_{\ast }$
belongs to $\mathcal{S}$ is equivalent to $\ln (f_{\ast })\in
L_{1}\left(P\right)$, the space of integrable functions on $\mathcal{Z}$
with respect to $P$.

We consider the MLE of the density $f_{\ast }$. To do so, we define the
maximum likelihood contrast $\gamma $ to be the following functional,%
\begin{equation*}
\gamma :f\in \mathcal{S}\longmapsto \left( z\in \mathcal{Z\longmapsto }-\ln
\left( f\left( z\right) \right) \right) \text{ .}
\end{equation*}%
Then the risk $P\gamma \left( f\right) $ associated to the contrast $\gamma $
on a function $f\in \mathcal{S}$ is the following, 
\begin{equation*}
P\gamma \left( f\right) =P\left( \ln f\right) _{-}-P\left( \ln f\right)
_{+}\in \mathbb{R\cup \left\{ +\infty \right\} }\text{ .}
\end{equation*}

Also, the excess risk of a function $f$ with respect to the density $f_{\ast
}$ is classically given in this context by the KL divergence of $f$ with
respect to $f_{\ast }$. Recall that for two probability distributions $P_{f}$
and $P_{g}$ on $\left( \mathcal{Z},\mathcal{T}\right) $ of respective
densities $f$ and $g$ with respect to $\mu $, the KL divergence of $P_{g}$
with respect to $P_{f}$ is defined to be 
\begin{equation*}
\mathcal{K}\left( P_{f},P_{g}\right) =%
\begin{cases}
\int_{\mathcal{Z}}\ln \left( \frac{dP_{f}}{dP_{g}}\right) dP_{g}=\int_{%
\mathcal{Z}}f\ln \left( \frac{f}{g}\right) d\mu & \text{if}\quad P_{f}\ll
P_{g} \\ 
\infty & \text{otherwise.}%
\end{cases}%
\end{equation*}%
By a slight abuse of notation we denote $\mathcal{K}\left( f,t\right) $
rather than $\mathcal{K}\left( P_{f},P_{g}\right) $ and by the Jensen
inequality we notice that $\mathcal{K}\left( f,g\right) $ is a nonnegative
quantity, equal to zero if and only if $f=g$ $\ \mu $-$a.s.$ Hence, for any $%
f\in \mathcal{S}$, the excess risk of a function $f$ with respect to the
density $f_{\ast }$ satisfies 
\begin{equation}
P\left( \gamma (f))-P(\gamma \left( f_{\ast }\right) \right) =\int_{\mathcal{%
Z}}\ln \left( \frac{f_{\ast }}{f}\right) f_{\ast }d\mu =\mathcal{K}\left(
f_{\ast },f\right) \geq 0  \label{def_excess_risk_kullback}
\end{equation}%
and this nonnegative quantity is equal to zero if and only if $f_{\ast }=f$ $%
\ \mu -a.s.$ Consequently, the unknown density $f_{\ast }$ is uniquely
defined by 
\begin{align*}
f_{\ast }& =\arg \min_{f\in \mathcal{S}}\left\{ P\left( -\ln f\right)
\right\} \\
& =\arg \min_{f\in \mathcal{S}}\left\{ P\gamma \left( f\right) \right\} 
\text{ .}
\end{align*}%
For a model $m$, that is a subset $m\subset \mathcal{S}$, we define the
maximum likelihood estimator on $m$, whenever it exists, by 
\begin{align}
\hat{f}_{m}& \in \arg \min_{f\in m}\left\{ P_{n}\gamma \left( f\right)
\right\}  \label{def_estimator_MLE} \\
& =\arg \min_{f\in m}\left\{ \frac{1}{n}\sum_{i=1}^{n}-\ln \left( f\left(
\xi _{i}\right) \right) \right\} \text{ }.  \notag
\end{align}

\subsection{Histogram Models\label{section_histo}}

The models $m$ that we consider here to define the maximum likelihood
estimators as in (\ref{def_estimator_MLE}) are made of histograms defined on
a fixed partition of $\mathcal{Z}$. More precisely, for a finite partition $%
\Lambda _{m}$ of $\mathcal{Z}$ of cardinality $\left\vert \Lambda
_{m}\right\vert =D_{m}+1$, $D_{m}\in \mathbb{N}$, we set 
\begin{equation*}
m=\left\{ f=\sum_{I\in \Lambda _{m}}\beta _{I}\mathbbm{1}_{I}\text{ ; }%
\left( \beta _{I}\right) _{I\in \Lambda _{m}}\in \mathbb{R}_{+}^{D_{m}+1},%
\text{ }f\geq 0\text{ and }\sum_{I\in \Lambda _{m}}\beta _{I}\mu \left(
I\right) =1\right\} \text{ .}
\end{equation*}%
Note that the smallest affine space contained in $m$ is of dimension $D_{m}$%
. The quantity $D_{m}$ can thus be interpreted as the number of degrees of
freedom in the (parametric) model $m$. We assume that any element $I$ of the
partition $\Lambda _{m}$ is of positive measure with respect to $\mu $: for
all $I\in \Lambda _{m}$, \ \ $\mu \left( I\right) >0$ . As the partition $%
\Lambda _{m}$ is finite, we have $P\left( \ln f\right) _{+}<\infty $ for all 
$f\in m$ and so $m\subset \mathcal{S}$. We state in the next proposition
some well-known properties that are satisfied by histogram models submitted
to the procedure of MLE (see for example \cite[Section 7.3]{Massart:07}).

\begin{proposition}
Let 
\begin{equation*}
f_{m}=\sum_{I\in \Lambda _{m}}\frac{P\left( I\right) }{\mu \left( I\right) }%
\mathbf{1}_{I}\text{ }.
\end{equation*}%
Then $f_{m}\in m$ and $f_{m}$ is called the KL projection of $f_{\ast }$
onto $m$. Moreover, it holds 
\begin{equation*}
f_{m}=\arg \min_{f\in m}P\left( \gamma (f)\right) \text{ .}
\end{equation*}%
The following Pythagorean-like identity for the KL divergence holds, for
every $f\in m$, 
\begin{equation}
\mathcal{K}\left( f_{\ast },f\right) =\mathcal{K}\left( f_{\ast
},f_{m}\right) +\mathcal{K}\left( f_{m},f\right) \text{ .}
\label{pythagore_relation_MLE}
\end{equation}%
The maximum likelihood estimator on $m$ is well-defined and
corresponds to the so-called frequencies histogram associated to the
partition $\Lambda _{m}$. We also have the following formulas,%
\begin{equation*}
\hat{f}_{m}=\sum_{I\in \Lambda _{m}}\frac{P_{n}\left( I\right) }{\mu \left(
I\right) }\mathbbm{1}_{I}\text{        and        } P_n(\gamma(f_m)-\gamma(\hat{f}_m))=\mathcal{K}(\hat{f}_m,f_m)\text{ }.
\end{equation*}%
\end{proposition}

\begin{remark}
Histogram models are special cases of general exponential families exposed
for example in Barron and Sheu \cite{BarSheu:91} (see also Castellan \cite%
{Castellan:03} for the case of exponential models of piecewise polynomials).
The projection property (\ref{pythagore_relation_MLE}) can be generalized to
exponential models (see \cite[Lemma 3]{BarSheu:91} and Csisz\'ar \cite%
{Csiszar:75}).
\end{remark}

\begin{remark}
As by \eqref{def_excess_risk_kullback} we have 
\begin{equation*}
P\left( \gamma (f_{m})-\gamma \left( f_{\ast }\right) \right) =\mathcal{K}%
\left( f_{\ast },f_{m}\right)
\end{equation*}
and for any $f\in m$, 
\begin{equation*}
P\left( \gamma (f)-\gamma \left( f_{\ast }\right) \right) =\mathcal{K}\left(
f_{\ast },f\right)
\end{equation*}%
we easily deduce from (\ref{pythagore_relation_MLE}) that the excess risk on 
$m$ is still a KL divergence, as we then have for any $f\in m$,%
\begin{equation*}
P\left( \gamma (f)-\gamma (f_{m})\right) =\mathcal{K}\left( f_{m},f\right) 
\text{ }.
\end{equation*}
\end{remark}

\subsection{Over-Penalization}
\label{section_overpen}

Now let's define our model selection procedure. We propose three ways to understand the benefits of over-penalization. Of course, the three points of view are interrelated, but they provide different and complementary insights on the behavior of over-penalization.

\subsubsection{Over-Penalization as Estimation of the Ideal Penalty}\label{ssection_ideal_pen}

We are given a collection of histogram models denoted $\mathcal{M}_{n}$, with finite cardinality depending on the sample size $n$, and its associated collection of maximum likelihood estimators $\left\{ \hat{f}_{m};m\in 
\mathcal{M}_{n}\right\} $. By taking a (nonnegative) penalty function $\pen$
on $\mathcal{M}_{n}$, 
\begin{equation*}
\pen:m\in \mathcal{M}_{n}\longmapsto \pen\left( m\right) \in \mathbb{R}^{+}%
\text{ ,}
\end{equation*}%
the output of the penalization procedure (also called the selected model) is
by definition any model satisfying, 
\begin{equation}
\widehat{m}\in \arg \min_{m\in \mathcal{M}_{n}}\left\{ P_{n}(\gamma (\hat{f}%
_{m}))+\pen\left( m\right) \right\} \text{ .}  \label{def_proc_2_MLE}
\end{equation}%
We aim at selecting an estimator $\hat{f}_{\widehat{m}}$ with a KL
divergence, pointed on the true density $f_{\ast }$, as small as possible.
Hence, we want our selected model to have a performance as close as possible
to the excess risk achieved by an oracle model (possibly non-unique),
defined to be, 
\begin{align}
m_{\ast }\in & \arg \min_{m\in \mathcal{M}_{n}}\left\{ \mathcal{K}(f_{\ast },%
\hat{f}_{m})\right\}  \label{def_oracle} \\
=& \arg \min_{m\in \mathcal{M}_{n}}\left\{ P(\gamma (\hat{f}_{m}))\right\} 
\text{ .}  \label{def_oracle_2}
\end{align}%
From (\ref{def_oracle_2}), it is seen that an ideal penalty in the
optimization task (\ref{def_proc_2_MLE}) is given by%
\begin{equation*}
\pen_{\mathrm{id}}\left( m\right) =P(\gamma (\hat{f}_{m}))-P_{n}(\gamma (%
\hat{f}_{m}))\text{ ,}
\end{equation*}%
since in this case, the criterion $\crit_{\mathrm{id}}\left( m\right)
=P_{n}(\gamma (\hat{f}_{m}))+\pen_{\text{id}}\left( m\right) $ is equal to
the true risk $P(\gamma (\hat{f}_{m}))$. However $\pen_{\mathrm{id}}$ is
unknown and, at some point, we need to give some estimate of it. In
addition, $\pen_{\mathrm{id}}$ is random, but we may not be able to provide
a penalty, even random, whose fluctuations at a fixed model $m$ would be
positively correlated to the fluctuations of $\pen_{\text{id}}\left(
m\right) $. This means that we are rather searching for an estimate of a 
\textit{deterministic functional} of $\pen_{\mathrm{id}}$. But which
functional would be convenient? The answer to this question is essentially
contained in the solution of the following problem.

\medskip

\noindent \textbf{Problem 1. }\textit{For any fixed} $\beta \in \left(
0,1\right) $ \textit{find the} deterministic penalty $\pen_{\mathrm{id}%
,\beta }:\mathcal{M}_{n}\rightarrow R_{+}$, \textit{that minimizes the value
of} $C$, \textit{among constants} $C>0$ \textit{which satisfy the following
oracle inequality}, 
\begin{equation}
\mathbb{P}\left( \mathcal{K}(f_{\ast },\hat{f}_{\widehat{m}})\leq
C\inf_{m\in \mathcal{M}_{n}}\left\{ \mathcal{K(}f_{\ast },\hat{f}%
_{m})\right\} \right) \geq 1-\beta \text{ .}  \label{problem_beta}
\end{equation}

\noindent The solution---or even the existence of a solution---to the
problem given in (\ref{problem_beta}) is not easily accessible and depends
on assumptions on the law $P$ of data and on approximation properties of the
models, among other things. In the following, we give a reasonable candidate
for $\pen_{\mathrm{id},\beta }$. Indeed, let us set $\beta _{\mathcal{M}%
}=\beta /$Card$\left( \mathcal{M}_{n}\right) $ and define 
\begin{equation}
\pen_{\mathrm{opt},\beta }\left( m\right) =q_{1-\beta _{\mathcal{M}}}\left\{
P(\gamma (\hat{f}_{m})-\gamma (f_{m}))+P_{n}(\gamma (f_{m})-\gamma (\hat{f}%
_{m}))\right\} \text{ ,}  \label{pen_opt_quantile}
\end{equation}%
where $q_{\lambda }\left\{ Z\right\} =\inf \left\{ q\in \mathbb{R};\mathbb{P}%
\left( Z\leq q\right) \geq \lambda \right\} $ is the quantile of level $%
\lambda $ for the real random variable $Z$. Our claim is that $\pen_{\text{%
opt,}\beta }$ gives in (\ref{problem_beta}) a constant $C$ which is close to
one, under some general assumptions (see Section \ref{section_results} for
precise results). Let us explain now why $\pen_{\mathrm{opt},\beta }$ should
lead to a nearly optimal model selection.

We set%
\begin{equation*}
\Omega _{0}=\bigcap_{m\in \mathcal{M}_{n}}\left\{ P(\gamma (\hat{f}%
_{m})-\gamma (f_{m}))+P_{n}(\gamma (f_{m})-\gamma (\hat{f}_{m}))\leq \pen_{%
\mathrm{opt},\beta }\left( m\right) \right\} \text{ .}
\end{equation*}%
We see, by definition of $\pen_{\mathrm{opt},\beta }$ and by a simple union
bound over the models $m\in \mathcal{M}_{n}$, that the event $\Omega _{0}$
is of probability at least $1-\beta $. Now, by definition of $\widehat{m}$,
we have, for any $m\in \mathcal{M}_{n}$, 
\begin{equation}
P_{n}(\gamma (\hat{f}_{\widehat{m}}))+\pen_{\mathrm{opt},\beta }(\widehat{m}%
)\leq P_{n}(\gamma (\hat{f}_{m}))+\pen_{\mathrm{opt},\beta }\left( m\right) 
\text{ .}  \label{def_min_crit}
\end{equation}%
By centering by $P\left( \gamma \left( f_{\ast }\right) \right) $ and using
simple algebra, Inequality (\ref{def_min_crit}) can be written as, 
\begin{align*}
& P(\gamma (\hat{f}_{\widehat{m}})-\gamma \left( f_{\ast }\right) )+\left[ %
\pen_{\mathrm{opt},\beta }(\widehat{m})-(P(\gamma (\hat{f}_{\widehat{m}%
})-\gamma (f_{\widehat{m}}))+P_{n}(\gamma (f_{\widehat{m}})-\gamma (\hat{f}_{%
\widehat{m}})))\right] \\
\leq & P(\gamma (\hat{f}_{m})-\gamma \left( f_{\ast }\right) )+\left[ \pen_{%
\mathrm{opt},\beta }\left( m\right) -(P(\gamma (\hat{f}_{m})-\gamma
(f_{m}))+P_{n}(\gamma (f_{m})-\gamma (\hat{f}_{m})))\right] \\
& +\left( P_{n}-P\right) \left( \gamma (f_{m})-\gamma (f_{\widehat{m}%
})\right) \text{ .}
\end{align*}%
Now, on $\Omega _{0}$, we have $\pen_{\mathrm{opt},\beta }(\widehat{m}%
)-(P(\gamma (\hat{f}_{\widehat{m}})-\gamma (f_{\widehat{m}}))+P_{n}(\gamma
(f_{\widehat{m}})-\gamma (\hat{f}_{\widehat{m}})))\geq 0$, so we get on $%
\Omega _{0}$,

\begin{eqnarray*}
&&P(\gamma (\hat{f}_{\widehat{m}})-\gamma \left( f_{\ast }\right) ) \\
&\leq &P(\gamma (\hat{f}_{m})-\gamma \left( f_{\ast }\right) )+\left[ \pen_{%
\mathrm{opt},\beta }\left( m\right) -(P(\gamma (\hat{f}_{m})-\gamma
(f_{m}))+P_{n}(\gamma (f_{m})-\gamma (\hat{f}_{m})))\right] \\
&&+\left( P_{n}-P\right) \left( \gamma (f_{m})-\gamma (f_{\widehat{m}%
})\right) \text{ .}
\end{eqnarray*}%
Specifying to the MLE context, the latter inequality writes,%
\begin{align*}
\mathcal{K(}& f_{\ast },\hat{f}_{\widehat{m}}) \\
\leq & \mathcal{K(}f_{\ast },\hat{f}_{m})+\underset{(a)}{\underbrace{\left[ %
\pen_{\mathrm{opt},\beta }\left( m\right) -(\mathcal{K(}f_{m},\hat{f}_{m})+%
\mathcal{K(}\hat{f}_{m},f_{m}))\right] }}+\underset{(b)}{\underbrace{\left(
P_{n}-P\right) \left( \gamma (f_{m})-\gamma (f_{\widehat{m}})\right) }}\text{
.}
\end{align*}

In order to get an oracle inequality as in (\ref{problem_beta}), it remains
to control $(a)$ and $(b)$ in terms of the excess risks $\mathcal{K(}f_{\ast
},\hat{f}_{m})$ and $\mathcal{K(}f_{\ast },\hat{f}_{\widehat{m}})$. Quantity 
$(a)$ is related to deviations bounds for the true and empirical excess
risks of the M-estimators $\hat{f}_{m}$ and quantity $(b)$ is related to
fluctuations of empirical bias around the bias of the models. Suitable
controls of these quantities (as achieved in our proofs) will give sharp
oracle inequalities (see Section \ref{section_results} below).

\begin{remark}
Notice that our reasoning is not based on the particular value of the
contrast, so that to emphasize this point we choose to keep $\gamma$ in
most of our calculations rather than to specify to the KL divergence related
to the MLE case. As a matter of fact, the penalty $\pen_{\mathrm{opt},\beta
} $ given in (\ref{pen_opt_quantile}) is a good candidate in the general
context of M-estimation.
\end{remark}

We define an over-penalization procedure as follows.

\begin{definition}
A penalization procedure as defined in (\ref{def_proc_2_MLE}) is said to be
an over-penalization procedure if the penalty $\pen$ that is used satisfies $%
\pen\left( m\right) \geq \pen_{\mathrm{opt},\beta }\left( m\right) $ for all 
$m\in \mathcal{M}_{n}$ and for some $\beta \in \left( 1/2,1\right) $.
\end{definition}

Based on concentration inequalities for the excess risks (see Section \ref%
{section_risks_bounds_MLE}) we propose the following over-penalization
penalty for histogram selection,%
\begin{equation}
\pen_{+}\left( m\right) =\left( 1+C\max \left\{ \sqrt{\frac{D_{m}\ln (n+1)}{n%
}};\sqrt{\frac{\ln (n+1)}{D_{m}}};\frac{\ln (n+1)}{D_{m}}\right\} \right) 
\frac{D_{m}}{n}\text{ ,}
\label{pen_a}
\end{equation}%
where $C$ is a constant that should be either fixed a priori ($C=1$ or $2$
are typical choices) or estimated using data (see Section \ref{section_expe} for details about the choice of $C$). The logarithmic terms appearing in (\ref{pen_a}) are linked to the cardinal of the collection of models, since in our proofs we take a constant $\alpha$ such that $\ln\rm{Card} (\mathcal{M}_n)  + 2\ln(n+1) \leq \alpha \ln(n+1)$. The constant $\alpha$ then enters in the constant $C$ of (\ref{pen_a}). We show below nonasymptotic accuracy of
such procedure, both theoretically and practically.

\subsubsection{Over-Penalization through a pseudo-testing approach\label%
{ssection_testing}}

Let us now present a multiple test point of view on the model selection problem. The goal is to infer the oracle model (\ref{def_oracle}) so that an oracle inequality of the type of (\ref{problem_beta}) is ensured.

This task can be formulated by solving iterative pseudo-tests. Indeed, set the following collection of null and alternative hypotheses indexed by pairs of models: for $\left( m,m^{\prime }\right)\in\mathcal{M}_{n}^{2}$,
\begin{equation*}
\left\{ 
\begin{array}{c}
H_{0}\left( m,m^{\prime }\right) :\mathcal{K}(f_{\ast },\hat{f}%
_{m})>C_{n}\cdot \mathcal{K}(f_{\ast },\hat{f}_{m^{\prime }}) \\ 
H_{1}\left( m,m^{\prime }\right) :\mathcal{K}(f_{\ast },\hat{f}_{m})\leq
C_{n}\cdot \mathcal{K}(f_{\ast },\hat{f}_{m^{\prime }})%
\end{array}%
\right. \text{ ,}
\end{equation*}%
where the constant $C_{n}>1$ is uniform in $m,m^{\prime }\in \mathcal{M}_{n}$. To each pair $\left( m,m^{\prime }\right) \in \mathcal{M}_{n}^{2}$, let us assume that we are given a test $T\left( m,m^{\prime }\right) $ that is equal to one if $H_{0}\left( m,m^{\prime }\right) $ is rejected and zero otherwise.

It should be noted at this stage that what we have just called ``pseudo-test'' does not enter directly into the classical theory of statistical tests, since the null and alternative hypotheses that we consider are random events. However, as
we will see, the only notion related to our pseudo-tests needed for our model selection study is the notion of the ``size'' of a pseudo-test, which we will give in the following and which will provide a mathematically consistent analysis of the statistical situation. Moreover,it seems that random assumptions naturally arise in some statistical frameworks. For instance, in the context of variable selection along along the Lasso path, Lockhart \textit{et al.}\cite{MR3210970} consider sequential random null hypotheses based on the active sets of the variable included up to a step on the Lasso path (see especially \cite[Section 2.6]{MR3210970} as well as the discussion of B\"{u}hlmann \textit{et al.} \cite[Section 2]{MR3210971} and the rejoinder \cite[Section 2]{MR3210977}).

Finally, if the testing of random hypotheses disturbs the reader, we suggest taking our multiple pseudo-testing interpretation of the model selection task in its minimal sense, that is, as a mathematical description aiming at investigating tight conditions on the penalty, that would allow for near-optimal oracle guarantees for the penalization scheme (\ref{def_proc_2_MLE}).

In order to ensure an oracle inequality such as in (\ref{problem_beta}), we want to avoid as far as possible selecting a model whose excess risk is far greater than the one of the oracle. In terms of the preceding tests, we will see that this exactly corresponds to controlling the ``size'' of the pseudo-tests $T\left( m,m^{\prime }\right) $.

Let us note $\mathcal{R}\left( m,m^{\prime }\right) =\left\{ T\left(
m,m^{\prime }\right) =1\right\} $ the event where the pseudo-test $T\left(
m,m^{\prime }\right) $ rejects the null hypothesis $H_{0}\left( m,m^{\prime
}\right) $ and $\mathcal{T}\left( m,m^{\prime }\right) =\left\{ \mathcal{K}%
(f_{\ast },\hat{f}_{m})>C_{n}\cdot \mathcal{K}(f_{\ast },\hat{f}_{m^{\prime
}})\right\} $ the event where the hypothesis $H_{0}\left( m,m^{\prime
}\right) $ is true. By extension with the classical theory of statistical
testing, we denote $\alpha \left( m,m^{\prime }\right) $ the size of the
pseudo-test $T\left( m,m^{\prime }\right) $, given by%
\begin{equation*}
\alpha \left( m,m^{\prime }\right) =\mathbb{P}\left( \mathcal{R}\left(
m,m^{\prime }\right) \left\vert \mathcal{T}\left( m,m^{\prime }\right)
\right. \right) \text{ .}
\end{equation*}

To explain how we select a model $\widehat{m}$ that is close to the oracle
model $m_{\ast }$, let us enumerate the models of the collection: $\mathcal{M%
}_{n}=\left\{ m_{1},...,m_{k}\right\} $. Note that we do not assume that the
models are nested. Now, we test $T\left( m_{1},m_{j}\right) $ for $j$
increasing from $2$ to $k$. If there exists $i_{1}\in \left\{ 2,..,k\right\} 
$ such that $T\left( m_{1},m_{i_{1}}\right) =0$, then we perform the pseudo-tests $%
T\left( m_{i_{1}},m_{j}\right) $ for $j$ increasing from $i_{1}+1$ to $k$ or
choose $m_{k}$ as our oracle candidate if $i_{1}=k$. Otherwise, we choose $%
m_{1}$ as our oracle candidate. In general, we can thus define a finite
increasing sequence $m_{i_{l}}$ of models and a selected model $\widehat{m}=%
\widehat{m}\left( T\right) $ through the use of the collection of pseudo-tests $%
\left\{ T\left( m,m^{\prime }\right) ;m,m^{\prime }\in \mathcal{M}%
_{n}\right\} $. Also, the number of pseudo-tests that are needed to define $%
\widehat{m}$ is equal to $%
\card%
\left( \mathcal{M}_{n}\right) -1$.

\begin{figure}[tbp]
\centering
\includegraphics[width=0.7\textwidth]{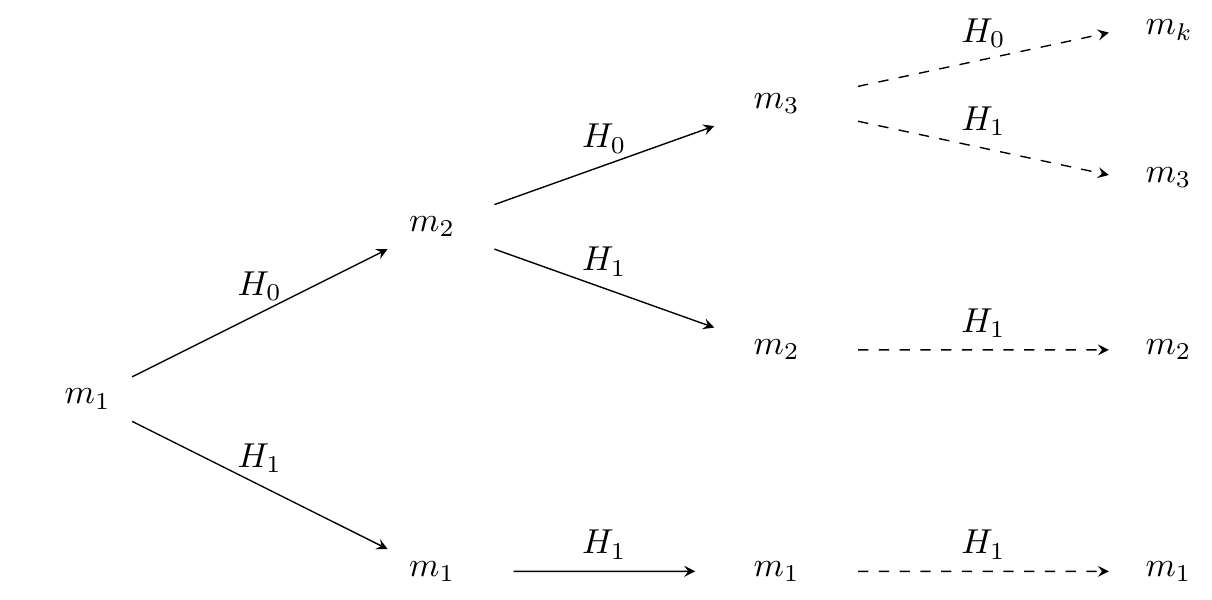}
\caption{Iteration of the pseudo-tests along the collection of models.}
\end{figure}

Let us denote $\mathcal{P}$ the set of pairs of models that have effectively
been tested along the iterative procedure. We thus have $%
\card%
(\mathcal{P})=%
\card%
\left( \mathcal{M}_{n}\right) -1$. Now define the event $\Omega _{T}$ under
which there is a first kind error along the collection of pseudo-tests in $\mathcal{%
P}$,%
\begin{equation*}
\Omega _{T}=\bigcup_{\left( m,m^{\prime }\right) \in \mathcal{P}}\left[ 
\mathcal{R}\left( m,m^{\prime }\right) \bigcap \mathcal{T}\left( m,m^{\prime
}\right) \right] 
\end{equation*}%
and assume that the sizes of the pseudo-tests are chosen such that 
\begin{equation}
\beta =\sum_{\left( m,m^{\prime }\right) \in \mathcal{M}^2_n}\alpha \left(
m,m^{\prime }\right) <1\text{ .}  \label{control_sizes}
\end{equation}

Assume also that the selected model $\widehat{m}$ has the following
Transitivity Property (\textbf{TP}):\medskip

\begin{description}
\item[(TP)] If for any $(m,m^{\prime })\in \mathcal{M}_{n}^{2},$ $T\left(
m,m^{\prime }\right) =1$ then $T\left( \widehat{m},m^{\prime }\right) =1$%
.\medskip
\end{description}

We believe that the transitivity property (\textbf{TP}) is intuitive and
legitimate as it amounts to assuming that there is no contradiction in
choosing $\widehat{m}$ as a candidate oracle model. Indeed, if a model $m$
is thought to be better than a model $m^{\prime }$, that is $T\left(
m,m^{\prime }\right) =1$, then the selected model $\widehat{m}$ should also
be thought to be better that $m^{\prime }$, $T\left( \widehat{m},m^{\prime
}\right) =1$.

The power of the formalism we have just introduced lies in the fact that the combination of Assumptions (\ref{control_sizes}) and (\textbf{TP}) automatically implies that an oracle inequality as in (\ref{problem_beta}) is satisfied. Indeed, property (\textbf{TP}) ensures that $\mathcal{T}\left( 
\widehat{m},m_{\ast }\right) \subset \mathcal{R}\left( \widehat{m},m_{\ast
}\right) $ because on the event $\mathcal{T}\left( \widehat{m},m_{\ast
}\right) $, we have $\widehat{m}\neq m_{\ast }$, so there exists $m$ such
that $\left( m,m_{\ast }\right) \in \mathcal{P}$ and $T\left( m,m_{\ast
}\right) =1$ (otherwise $\widehat{m}=m_{\ast }$) which in turn ensures $%
T\left( \widehat{m},m_{\ast }\right) =1$. Consequently, $\mathcal{T}\left( 
\widehat{m},m_{\ast }\right) =\mathcal{T}\left( \widehat{m},m_{\ast }\right)
\cap \mathcal{R}\left( \widehat{m},m_{\ast }\right) \subset \Omega _{T}$,
which gives%
\begin{equation}
\mathbb{P}\left( \mathcal{K}(f_{\ast },\hat{f}_{\widehat{m}})>C_{n}\cdot 
\mathcal{K}(f_{\ast },\hat{f}_{m_{\ast }})\right) =\mathbb{P}\left( \mathcal{%
T}\left( \widehat{m},m_{\ast }\right) \right) \leq \mathbb{P}\left( \Omega
_{T}\right) \leq \beta \text{ ,}  \label{oracle_test}
\end{equation}%
that is equivalent to (\ref{problem_beta}).

Let us turn now to a specific choice of pseudo-tests corresponding to model
selection by penalization. If we define for a penalty $%
\pen%
$, the penalized criterion $%
\crit%
_{%
\pen%
}\left( m\right) =P_{n}(\gamma (\hat{f}_{m}))+%
\pen%
\left( m\right) $, $m\in \mathcal{M}_{n}$ and take the following pseudo-tests%
\begin{equation*}
T\left( m,m^{\prime }\right) =T_{%
\pen%
}\left( m,m^{\prime }\right) =\mathbbm{1} _{\left\{ 
\crit%
_{%
\pen%
}\left( m\right) \leq 
\crit%
_{%
\pen%
}\left( m^{\prime }\right) \right\} }\text{ ,}
\end{equation*}%
then it holds%
\begin{equation*}
\widehat{m}(T_{%
\pen%
})\in \arg \min_{m\in \mathcal{M}_{n}}\left\{ P_{n}(\gamma (\hat{f}_{m}))+%
\pen%
\left( m\right) \right\}
\end{equation*}%
and property (\textbf{TP}) is by consequence satisfied.

It remains to choose the penalty $%
\pen%
$ such that the sizes of the pseudo-tests $T_{%
\pen%
}\left( m,m^{\prime }\right) $ are controlled. This is achieved by taking $%
\pen%
=%
\pen%
_{\text{opt},\tilde{\beta} /2}$ as defined in (\ref{pen_opt_quantile}), with $\tilde{\beta}=\beta/\text{Card}(\mathcal{M}_n)$. Indeed, in
this case,%
\begin{align}
\alpha \left( m,m^{\prime }\right) =&\mathbb{P(}P_{n}(\gamma (\hat{f}_{m}))+%
\pen\left( m\right) \leq P_{n}(\gamma (\hat{f}_{m^{\prime }}))+\pen\left(
m^{\prime }\right) \left\vert \mathcal{T}\left( m,m^{\prime }\right)\right. )
\notag \\
\approx&\mathbb{P(}\mathcal{K}(f_{\ast },\hat{f}_{m})+\left[ \pen_{\text{opt,%
}\tilde{\beta} /2}\left( m\right) -(\mathcal{K}(f_{m},\hat{f}_{m})+\mathcal{K}(\hat{f%
}_{m},f_{m}))\right]  \label{approx} \\
&\leq \mathcal{K}(f_{\ast },\hat{f}_{m^{\prime }})+\left[ \pen_{\mathrm{opt}%
,\tilde{\beta}/2}\left( m^{\prime }\right) -(\mathcal{K}(f_{m^{\prime }},\hat{f}%
_{m^{\prime }})+\mathcal{K}(\hat{f}_{m^{\prime }},f_{m^{\prime }}))\right]
\left\vert \mathcal{T}\left( m,m^{\prime }\right) \right. )  \notag \\
\leq & \frac{\beta }{2\text{Card}\left( \mathcal{M}^2_{n}\right)} +\mathbb{P}%
\left(\mathcal{K}(f_{\ast },\hat{f}_{m}) \leq \mathcal{K}(f_{\ast },\hat{f}%
_{m^{\prime }})\right.  \notag \\
&\hspace{2.8cm}\left.+\left[ \pen_{\mathrm{opt},\tilde{\beta} /2}\left(
m^{\prime}\right) -(\mathcal{K}(f_{m^{\prime }},\hat{f}_{m^{\prime }})+%
\mathcal{K}(\hat{f}_{m^{\prime }},f_{m^{\prime }}))\right] \left\vert 
\mathcal{T}\left(m,m^{\prime }\right)\right.\right)\text{ .}  \notag
\end{align}%
In line (\ref{approx}), the equality is only approximated since we neglected
the centering of model biases by their empirical counterparts, as these centered random variables should be small compared to the other quantities for models of interest. Now assume that 
\begin{equation*}
\mathbb{P}\left( \pen_{\mathrm{opt},\tilde{\beta} /2}\left( m^{\prime }\right) -(%
\mathcal{K}(f_{m^{\prime }},\hat{f}_{m^{\prime }})+\mathcal{K}(\hat{f}%
_{m^{\prime }},f_{m^{\prime }}))\geq \varepsilon _{n}\mathcal{K}(f_{\ast },%
\hat{f}_{m^{\prime }})\right) \leq \frac{\beta }{2\text{Card}\left( \mathcal{%
M}^2_{n}\right) }\text{ ,}
\end{equation*}%
for some deterministic sequence $\varepsilon _{n}$ not depending on $%
m^{\prime }$. Such result is obtained in Section~\ref%
{section_risks_bounds_MLE} and is directly related to the concentration
behavior of the true and empirical excess risks. Then we get 
\begin{align*}
&\mathbb{P}\left( \mathcal{K}(f_{\ast },\hat{f}_{m})\leq \mathcal{K}(f_{\ast
},\hat{f}_{m^{\prime }})+\left[ \pen_{\mathrm{opt},\tilde{\beta}/2 }\left( m^{\prime
}\right) -(\mathcal{K}(f_{m^{\prime }},\hat{f}_{m^{\prime }})+\mathcal{K}(%
\hat{f}_{m^{\prime }},f_{m^{\prime }}))\right] \left\vert \mathcal{T}\left(
m,m^{\prime }\right) \right. \right) \\
\leq &\frac{\beta }{2\text{Card}\left( \mathcal{M}^2_{n}\right) }+\mathbb{P}%
\left( \mathcal{K}(f_{\ast },\hat{f}_{m})\leq (1+\varepsilon _{n})\mathcal{K}%
(f_{\ast },\hat{f}_{m^{\prime }})\left\vert \mathcal{K}(f_{\ast },\hat{f}%
_{m})>C_{n}\mathcal{K}(f_{\ast },\hat{f}_{m^{\prime }})\right. \right) \\
=&\frac{\beta }{2\text{Card}\left( \mathcal{M}^2_{n}\right) }\text{ ,}
\end{align*}%
where the last equality is valid if $C_{n}\geq 1+\varepsilon _{n}$. In this
case,%
\begin{equation*}
\alpha \left( m,m^{\prime }\right) \leq \frac{\beta }{\text{Card}\left( 
\mathcal{M}^2_{n}\right) }
\end{equation*}%
and inequality (\ref{control_sizes}) is satisfied and so is the oracle
inequality (\ref{oracle_test}).

\begin{remark}
There is a gap between the penalty $\pen_{\mathrm{opt},\beta}$ considered in Section \ref{ssection_ideal_pen} to ensure an oracle inequality and the penalty $\pen_{\mathrm{opt},\tilde{\beta}/2}$ defined in this section. This comes from the fact that using the pseudo-testing framework, we aim at controlling the probability of the event $\Omega_T$ under which there is a first kind error along the pseudo-tests performed in $\mathcal{P}$. Despite the fact that the set $\mathcal{P}$ consists in $\rm{Card}(\mathcal{M}_n)-1$ pseudo-tests, we give a bound that takes into account the $\rm{Card}(\mathcal{M}^2_n)$ pseudo-tests defined from the pairs of models. There is a possible loss here that consists in inflating the set $\mathcal{P}$ in order to make the union bound valid. However, such loss would only affect the constant $C$ in our over-penalization procedure (\ref{pen_a}) by a factor $2$, since the modification of the order of the quantile affects the penalty through a logarithmic factor.

\end{remark}

\begin{remark}
The Transitivity Property (\textbf{TP}) allows to unify most the selection
rules. Indeed, as soon as we want to select an
estimator $\tilde{f}$ (or a model) that optimizes a criterion, 
\begin{equation*}
\tilde{f}\in \arg \min_{f\in \mathcal{F}}\left\{ 
\crit%
\left( f\right) \right\} \text{ ,}
\end{equation*}%
where $\mathcal{F}$ is a collection of candidate functions (or models), then 
$\tilde{f}$ is also defined by the collection of tests $T\left( f,f^{\prime
}\right) =\mathbf{1}_{\left\{ 
\crit%
\left( f\right) \leq 
\crit%
\left( f^{\prime }\right) \right\} }$. In particular, in T-estimation (\cite%
{MR2219712}) as well as in $\rho $-estimation \cite{BarBirgSart:17}) the
estimator is indeed a minimizer of a criterion that is interpreted as a
diameter of a subset of functions in $\mathcal{F}$. Note that in
T-estimation, the criterion is itself constructed through the use of some
(robust) tests, that by consequence do not act at the same level as our
tests in this case.
\end{remark}

\subsubsection{Graphical insights on over-penalization}

Finally, let us provide a graphic perspective on our over-penalization procedure.

If the penalty $%
\pen%
$ is chosen accordingly to the unbiased risk estimation principle, then it
should satisfy, for any model $m\in \mathcal{M}_{n}$,%
\begin{equation*}
\mathbb{E}\left[ P_{n}(\gamma (\hat{f}_{m}))+%
\pen%
\left( m\right) \right] \sim \mathbb{E}\left[ P\left( \gamma (\hat{f}%
_{m})\right) \right] \text{ .}
\end{equation*}%
In other words, the curve $\mathcal{C}_{n}:m\mapsto P_{n}(\gamma (\hat{f}%
_{m}))+%
\pen%
\left( m\right) $ fluctuates around its mean, which is essentially the curve 
$\mathcal{C}_{P}:m\mapsto \mathbb{E}\left[ P\left( \gamma (\hat{f}%
_{m})\right) \right] $, see Figure~\ref{fig:fig1}. Furthermore, the largest
is the model $m$, the largest are the fluctuations of $P_{n}(\gamma (\hat{f}%
_{m}))=\mathcal{K}\left( \hat{f}_{m},f_{m}\right) +P_{n}(\gamma (f_{m}))$.
This is seen for instance through the concentration inequalities established
in Section \ref{section_deviation_bounds}\ below for the empirical excess
risk $\mathcal{K}\left( \hat{f}_{m},f_{m}\right) $. Consequently, it can
happen that the curve $\mathcal{C}_{n}$ is rather flat for the largest
models and that the selected model is among the largest of the collection,
see Figure~\ref{fig:fig1}.

\begin{figure}[tbp]
\centering
\includegraphics[width=0.65\textwidth]{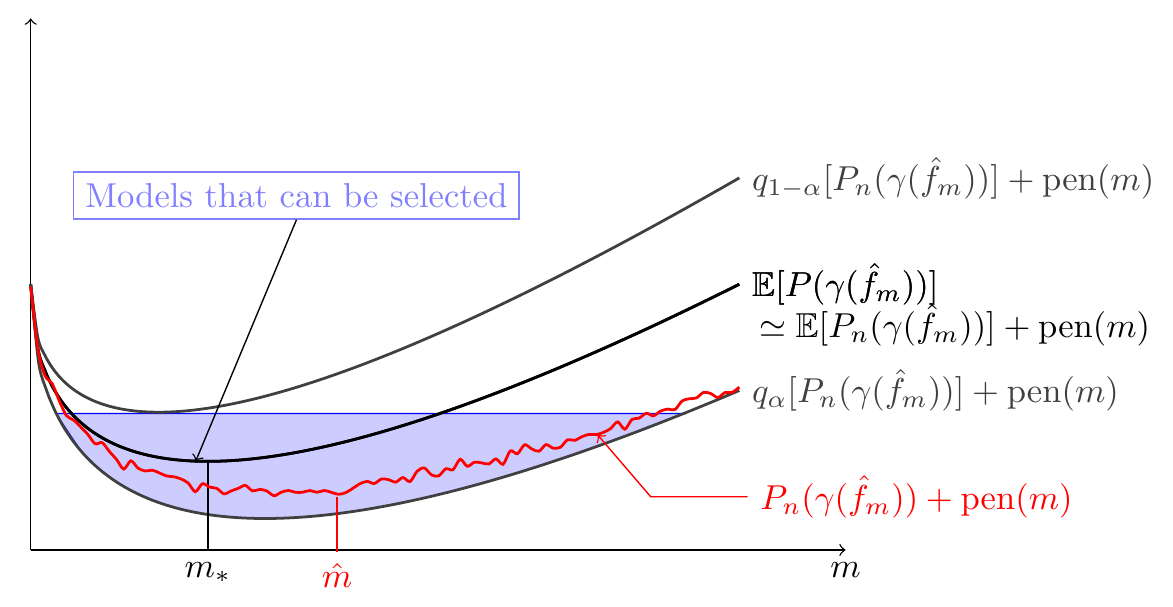}
\caption{A schematic view of the situation corresponding to a selection procedure based on the unbiased risk principle. The penalized empirical risk (in red) fluctuates around the expectation of the true risk. The size of the deviations typically increase with the model size, making the shape of the curves possibly flat for the largest models of the collection. Consequently, the chosen model can potentially be very large and lead to overfitting.}
\label{fig:fig1}
\end{figure}

By using an over-penalization procedure instead of the unbiased risk estimation principle, we will compensate the deviations for the largest models and thus obtain a thinner region of potential selected models, see
Figures~\ref{fig:fig2} and \ref{fig:fig3}. In other words, we will avoid
overfitting and by doing so, we will ensure a reasonable performance of our
over-penalization procedure in situations where unbiased risk estimation
fails. This is particularly the case when the amount of data is small to
moderate.

\begin{figure}[tbp]
\centering
\includegraphics[width=0.65\textwidth]{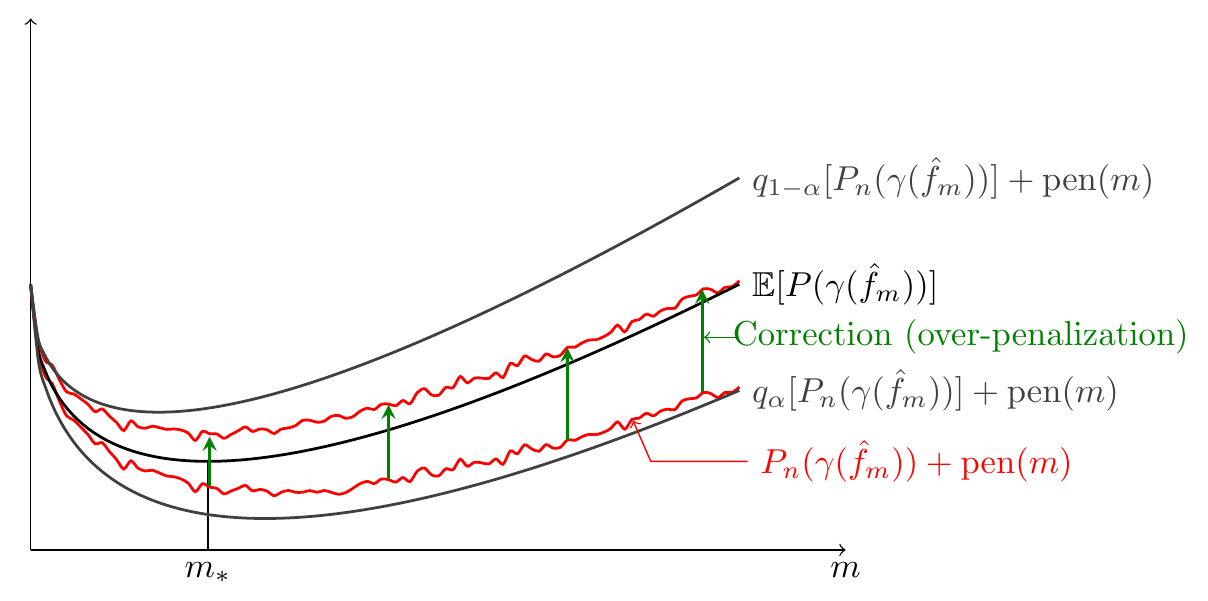}
\caption{The correction that should be applied to an unbiased risk estimation procedure would ideally be of the size of the deviations of the risk for each model of the collection.}
\label{fig:fig2}
\end{figure}

\begin{figure}[tbp]
\centering
\includegraphics[width=0.65\textwidth]{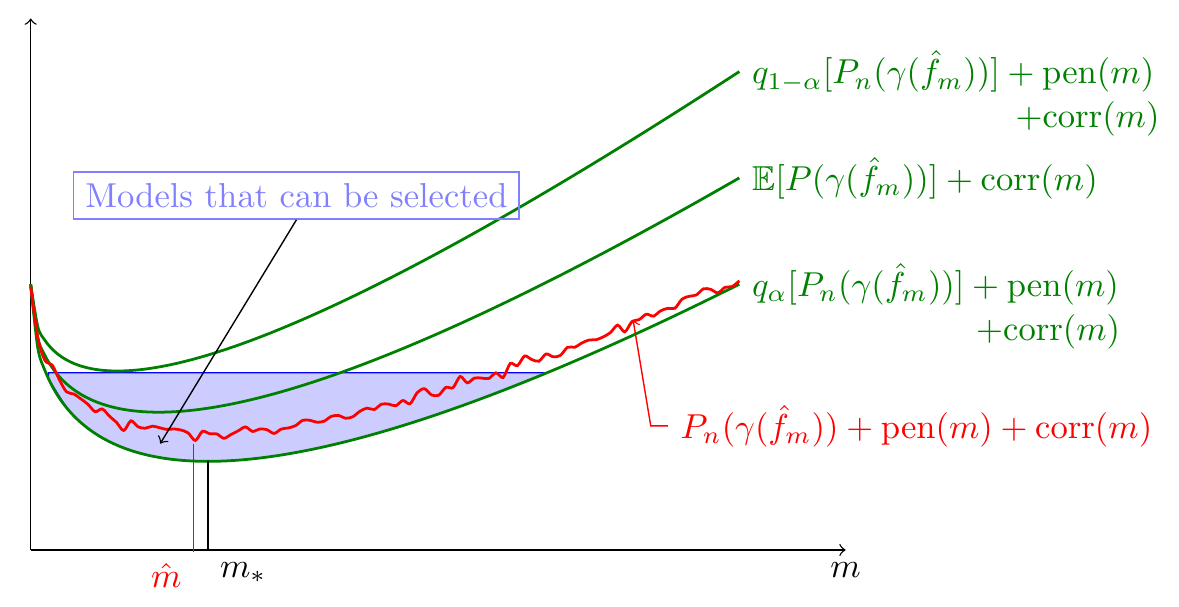}
\caption{After a suitable correction, the minimum of the red curve has a better shape. In addition, the region of models that can be possibly selected is substantially smaller and in particular avoids the largest models of the collection.}
\label{fig:fig3}
\end{figure}

\section{Theoretical Guarantees}
\label{section_results}

We state here our theoretical results related to the behavior of our
over-penalization procedure.

As explained in Section~\ref{section_overpen}, concentration inequalities for true and empirical excess risks are essential tools for understanding our model selection problem. As the theme of concentration inequalities for the excess risk constitutes a very recent and exciting area of research, these inequalities also have an interest in themselves and we state them out in Section~\ref{section_risks_bounds_MLE}.

In Section~\ref{section_results_bounded_setting}, we give a sharp oracle inequality proving the optimality of our procedure. We also compare our result to what would be obtained for AIC, which suggests the superiority of over-penalization in the nonasymptotic regime, i.e. for a small to medium sample size.

\subsection{True and empirical excess risks' concentration\label%
{section_risks_bounds_MLE}}

In this section, we fix the linear model $m$ made of histograms and we are
interested by concentration inequalities for the true excess risk $P\left(
\gamma (\hat{f}_{m})-\gamma (f_{m})\right) $ on $m$ and for its empirical
counterpart $P_{n}\left( \gamma (f_{m})-\gamma (\hat{f}_{m})\right) .$

\begin{theorem}
\label{opt_bounds}Let $n\in \mathbb{N}$, $n\geq 1$ and let $\alpha
,A_{+},A_{-}$ and $A_{\Lambda }$ be positive constants. Take $m$ a model of
histograms defined on a fixed partition $\Lambda _{m}$ of $\mathcal{Z}$. The
cardinality of $\Lambda _{m}$ is denoted by $D_{m}.$ Assume that $%
1<D_{m}\leq A_{+}n/(\ln (n+1))\leq n$ and%
\begin{equation}
0<A_{\Lambda }\leq D_{m}\inf_{I\in \Lambda _{m}}\left\{ P\left( I\right)
\right\} \text{ .}  \label{lower_reg-1}
\end{equation}%
If $\left( \alpha +1\right) A_{+}/A_{\Lambda }\leq \tau =\sqrt{3}-\sqrt{2}< 0.32$,
then a positive constant $A_{0}$ exists, only depending on $\alpha ,A_{+}$
and $A_{\Lambda }$, such that by setting%
\begin{equation}
\varepsilon _{n}^{+}\left( m\right) =\max \left\{ \sqrt{\frac{D_{m}\ln (n+1)%
}{n}};\sqrt{\frac{\ln (n+1)}{D_{m}}};\frac{\ln (n+1)}{D_{m}}\right\}
\label{def_A_0_MLE-1}
\end{equation}%
and 
\begin{equation*}
\varepsilon _{n}^{-}\left( m\right) =\max \left\{ \sqrt{\frac{D_{m}\ln (n+1)%
}{n}};\sqrt{\frac{\ln (n+1)}{D_{m}}}\right\} \text{ ,}
\end{equation*}%
we have, on an event of probability at least $1-4(n+1)^{-\alpha }$, 
\begin{equation}
\left( 1-A_{0}\varepsilon _{n}^{-}\left( m\right) \right) \frac{D_{m}}{2n}%
\leq \mathcal{K}\left( f_{m},\hat{f}_{m}\right) \leq \left(
1+A_{0}\varepsilon _{n}^{+}\left( m\right) \right) \frac{D_{m}}{2n}\text{ ,}
\label{upper_true_risk}
\end{equation}%
\begin{equation}
\left( 1-A_{0}\varepsilon _{n}^{-}\left( m\right) \right) \frac{D_{m}}{2n}%
\leq \mathcal{K}\left( \hat{f}_{m},f_{m}\right) \leq \left(
1+A_{0}\varepsilon _{n}^{+}\left( m\right) \right) \frac{D_{m}}{2n}\text{ .}
\label{upper_emp_risk}
\end{equation}
\end{theorem}

In the previous theorem, we obtain sharp upper and lower bounds for true and empirical excess risk on $m$. They are optimal at the first order since the leading constants are equal in the upper and lower bounds. They show the concentration of the true and empirical excess risks around the value $D_{m}/(2n)$. Moreover, Theorem \ref{opt_bounds} establishes equivalence
with high probability of the true and empirical excess risks for models of reasonable dimension.

Concentration inequalities for the excess risks as in Theorem\ref{opt_bounds} is a new and exciting direction of research related to the theory of statistical learning and to high-dimensional statistics. Boucheron and Massart \cite{BouMas:10} obtained a pioneering result describing the concentration of the empirical excess risk around its mean, a property that they call a high-dimensional Wilks phenomenon. Then a few authors obtained results describing the concentration of the true excess risk around its mean  \cite{saum:12}, \cite{chatterjee2014}, \cite{MurovandeGeer:15} or around its
median \cite{bellec2016bounds}, \cite{bellec2017towards} for (penalized)
least square regression\ and in an abstract M-estimation framework \cite%
{vandeGeerWain:16}. In particular, recent results of \cite{vandeGeerWain:16}
include the case of MLE on exponential models and as a matter of fact, on
histograms. Nevertheless, Theorem \ref{opt_bounds} is a valuable addition to
the literature on this line of research since we obtain here nonetheless the
concentration around a fixed point, but an explicit value $D_{m}/2n$ for
this point. On the contrary, the concentration point is available in \cite%
{vandeGeerWain:16} only through an implicit formula involving local suprema
of the underlying empirical process.

The principal assumption in Theorem \ref{opt_bounds} is inequality (\ref%
{lower_reg-1}) of lower regularity of the partition with respect to $P$. It
is ensured as soon as the density $f_{\ast }$ is uniformly bounded from
below and the partition is lower regular with respect to the reference
measure $\mu $ (which will be the Lebesgue measure in our experiments). No
restriction on the largest values of $f_{\ast }$ are needed. In particular,
we do not restrict to the bounded density estimation setting.

Castellan \cite{Castellan:99} proved related, but weaker inequalities than
in Theorem \ref{opt_bounds}\ above. She also asked for a lower regularity
property of the partition, as in Proposition 2.5 \cite{Castellan:99}, where
she derived a sharp control of the KL divergence of the histogram estimator
on a fixed model. More precisely, Castellan assumes that there exists a
positive constant $B$ such that 
\begin{equation}
\inf_{I\in \Lambda _{m}}\mu \left( I\right) \geq B\frac{\left( \ln
(n+1)\right) ^{2}}{n}\text{ }.  \label{castellan_regularity}
\end{equation}%
This latter assumption is thus weaker than (\ref{lower_reg-1}) for the
considered model as its dimension $D_{m}$ is less than the order $n\left(
\ln (n+1)\right) ^{-2}$. We could assume (\ref{castellan_regularity})
instead of (\ref{lower_reg-1}) in order to derive Theorem \ref{opt_bounds}.
This would lead to less precise results for second order terms in the
deviations of the excess risks but the first order bounds would be
preserved. More precisely, if we replace assumption (\ref{lower_reg-1}) in
Theorem \ref{opt_bounds} by Castellan's assumption (\ref%
{castellan_regularity}), a careful look at the proofs show that the
conclusions of Theorem \ref{opt_bounds} are still valid for $\varepsilon
_{n}=A_{0}\left( \ln (n+1)\right) ^{-1/2}$, where $A_{0}$ is some positive
constant. Thus assumption (\ref{lower_reg-1}) is not a fundamental
restriction in comparison to Castellan's work \cite{Castellan:99}, but it
leads to more precise results in terms of deviations of the true and
empirical excess risks of the histogram estimator.

The proof of Theorem \ref{opt_bounds}, that can be found in Section \ref%
{section_deviation_bounds}, is based on an improvement of independent
interest of the previously best known concentration inequality for the
chi-square statistics. See Section \ref{section_chi_square}\ for the precise
result.

\subsection{An Oracle Inequality}
\label{section_results_bounded_setting}

First, let us state the set of five structural assumptions required to establish the nonasymptotic optimality of the over-penalization procedure. These assumptions will be discussed in more detail at the end of this section, following the statement of a sharp oracle inequality.

\underline{Set of assumptions (\textbf{SA})}

\begin{description}
\item[(P1)] Polynomial complexity of $\mathcal{M}_{n}$: $\card\left( 
\mathcal{M}_{n}\right) \leq n^{\alpha _{\mathcal{M}}}.$

\item[(P2)] Upper bound on dimensions of models in $\mathcal{M}_{n}$: there
exists a positive constant $A_{\mathcal{M},+}$ such that for every $m\in 
\mathcal{M}_{n},$ 
\begin{equation*}
D_{m}\leq A_{\mathcal{M},+}\frac{n}{\left( \ln (n+1)\right) ^{2}}\leq n\text{
}.
\end{equation*}

\item[(Asm)] The unknown density $f_{\ast }$ satisfies some moment condition
and is uniformly bounded from below: there exist some constants $A_{\min }>0 
$ and $p>1$ such that, 
\begin{equation*}
\int_{\mathcal{Z}}f_{\ast }^{p}\left[ (\ln f_{\ast })^{2}\vee 1\right] d\mu
<+\infty
\end{equation*}%
and 
\begin{equation}
\inf_{z\in \mathcal{Z}}f_{\ast }\left( z\right) \geq A_{\min }>0\text{ }.
\label{lower_target}
\end{equation}

\item[(Alr)] Lower regularity of the partition with respect to $\mu $: there
exists a positive finite constant $A_{\Lambda }$ such that, for all $m\in 
\mathcal{M}_{n}$, 
\begin{equation*}
D_{m}\inf_{I\in \Lambda _{m}}\mu \left( I\right) \geq A_{\Lambda }\geq A_{\mathcal{M},+}(\alpha _{\mathcal{M}}+6)/\tau\text{ },
\end{equation*}
where $\tau=\sqrt{3}-\sqrt{2}>0$.

\item[(\textbf{Ap)}] The bias decreases like a power of $D_{m}$: there exist 
$\beta _{-}\geq \beta _{+}>0$ and $C_{+},C_{-}>0$ such that%
\begin{equation*}
C_{-}D_{m}^{-\beta _{-}}\leq \mathcal{K}\left( f_{\ast },f_{m}\right) \leq
C_{+}D_{m}^{-\beta _{+}}\text{ }.
\end{equation*}
\end{description}

We are now ready to state our main theorem related to optimality of
over-penalization.

\begin{theorem}
\label{theorem_opt_pen_MLE}Take $n\geq 1$ and $r\in \left( 0,p-1\right) $.
For some $\Delta >0$, consider the following penalty,%
\begin{equation*}
\pen%
\left( m\right) =\left( 1+\Delta \varepsilon _{n}^{+}\left( m\right) \right) 
\frac{D_{m}}{n}\text{ , \ \ \ \ \ \ \ \ \ \ for all\ }m\in \mathcal{M}_{n}%
\text{ .}  
\end{equation*}%
Assume that the set of assumptions (\textbf{SA}) holds and that 
\begin{equation}
\beta _{-}<p\left( 1+\beta _{+}\right) /(1+p+r)\text{ \ \ or \ \ }%
p/(1+r)>\beta _{-}+\beta _{-}/\beta _{+}-1\text{ }.  \label{condition_beta}
\end{equation}%
Then there exists an event $\Omega _{n}$ of probability at least $%
1-(n+1)^{-2}$ and some positive constant $A_{1}$ depending only on the
constants defined in (\textbf{SA}) such that, if $\Delta \geq A_{1}>0$ then
we have on $\Omega _{n}$,%
\begin{equation}
\mathcal{K}\left( f_{\ast },\hat{f}_{\widehat{m}}\right) \leq (1+\delta
_{n})\inf_{m\in \mathcal{M}_{n}}\left\{ \mathcal{K}\left( f_{\ast },\hat{f}%
_{m}\right) \right\} \text{ ,}  \label{oracle_um_opt}
\end{equation}%
where $\delta _{n}=L_{\text{(\textbf{SA})},\Delta ,r}\left( \ln (n+1)\right)
^{-1/2}$ is convenient.
\end{theorem}

We derive in Theorem \ref{theorem_opt_pen_MLE} a pathwise oracle inequality
for the KL excess risk of the selected estimator, with constant almost one.
Our result thus establishes the nonasymptotic quasi-optimality of
over-penalization with respect to the KL divergence.

It is worth noting that three very strong features related to oracle
inequality (\ref{oracle_um_opt}) significantly improve upon the literature.
Firstly, inequality (\ref{oracle_um_opt}) expresses the performance of the
selected estimator through its KL divergence and compare it to the KL
divergence of the oracle. Nonasymptotic results pertaining to (robust)
maximum likelihood based density estimation usually control the Hellinger
risk of the estimator, \cite{Castellan:03}, \cite{Massart:07}, \cite%
{BirRozen:06}, \cite{MR2219712}, \cite{BarBirgSart:17}. The main reason is
that the Hellinger risk is much easier to handle that the KL divergence from
a mathematical point of view. For instance, the Hellinger distance is
bounded by one while the KL divergence can be infinite. However, from a
M-estimation perspective, the natural excess risk associated to likelihood
optimization is indeed the KL divergence and not the Hellinger distance.
These two risks are provably close to each other in the bounded setting \cite%
{Massart:07}, but may behave very differently in general.

Second, nonasymptotic results describing the performance of procedures
based on penalized likelihood, by comparing more precisely the (Hellinger)
risk of the estimator to the KL divergence of the oracle, all deal with the
case where the log-density to be estimated is bounded (\cite{Castellan:03}, 
\cite{Massart:07}). Here, we substantially extend the setting by considering
only the existence of a finite polynomial moment for the large values of the
density to be estimated.

Finally, the oracle inequality (\ref{oracle_um_opt}) is valid with positive probability---larger than $3/4$---as soon as one data is
available. To our knowledge, any other oracle inequality describing penalization performance for maximum likelihood density estimation is valid with positive probability only when the sample size $n$ is greater than an integer $n_{0}$ which depends on the constants defining the problem and that is thus unknown. We emphasize that we control the risk of the selected estimator \textit{for any sample size} and that this property is essential in practice when dealing with small to medium sample sizes.
Based on the arguments developed in Section \ref{section_overpen},\ we
believe that such a feature of Theorem \ref{theorem_opt_pen_MLE} is accessible only through the use of over-penalization and we conjecture in particular that  \textit{it is impossible} \textit{using AIC} \textit{to achieve such a control the
KL divergence of the selected estimator for any sample size}.

The oracle inequality (\ref{oracle_um_opt}) is valid under conditions (\ref%
{condition_beta}) relating the values of the bias decaying rates $\beta _{-}$%
and $\beta _{+}$ to the order $p$ of finite moment of the density $f_{\ast }$
and the parameter $r$. In order to understand these latter conditions, let
us assume for simplicity that $\beta _{-}=\beta _{+}=:\beta $. Then the
conditions (\ref{condition_beta}) both reduce to $\beta <p/(1+r).$ As $r$
can be taken as close to zero as we want, the latter inequality reduces to $%
\beta <p$. In particular, if the density to be estimated is bounded ($%
p=\infty $), then conditions (\ref{condition_beta}) are automatically
satisfied. If on the contrary the density $f_{\ast }$ only has finite polynomial
moment $p$, then the bias should not decrease too fast. In light of the
following comments, if $f_{\ast }$ is assumed to be $\alpha $-H\"{o}lderian, 
$\alpha \in (0,1]$, then $\beta \leq 2\alpha \leq 2$ and the conditions (\ref%
{condition_beta}) are satisfied, in the case where $\beta _{-}=\beta _{+}$,
as soon as $p\geq 2$.

To conclude this section,  let's now comment on the set of assumptions (\textbf{SA}). Assumption (\textbf{P1}) indicates that the collection of models  has increasing polynomial complexity. This is well suited to bin size selection because in this case we usually select among a number of models which is strictly bounded from above by the sample size. In the same manner,
Assumption (\textbf{P2}) is legitimate enough for us and corresponds to practice, where we aim at considering bin sizes for which each element of the partition contains a few sample points.

Assumption (\textbf{Asm}) imposes conditions on the density to be estimated.
More precisely, Assumption (\ref{lower_target}) stating that the unknown
density is uniformly bounded from below can also be found in the work of
Castellan \cite{Castellan:99}. The author assumes moreover in Theorem 3.4
where she derives an oracle inequality for the (weighted) KL excess risk of
the histogram estimator, that the target is of finite sup-norm. Furthermore,
from a statistical perspective, the lower bound (\ref{lower_target}) is
legitimate since, by Assumption (\textbf{Alr}), we use models of
lower-regular partitions with respect to the Lebesgue measure. In the case
where Inequality (\ref{lower_target}) would not hold, one would typically
have to consider exponentially many irregular histograms to take into
account the possibly vanishing mass of some elements of the partitions (for
more details on this aspect that goes beyond the scope of the present paper,
see for instance \cite{Massart:07}).

We require in (\textbf{Ap}) that the quality of the approximation of the
collection of models is good enough in terms of bias. More precisely, we
require a polynomially decreasing of excess risk of KL projections of the
unknown density onto the models. For a density $f_{\ast }$ uniformly bounded
away from zero, the upper bound on the bias is satisfied when for example, $%
\mathcal{Z}$ is the unit interval, $\mu =\leb$ is the Lebesgue measure on
the unit interval, the partitions $\Lambda _{m}$ are regular and the density 
$f_{\ast }$ belongs to the set $\mathcal{H}\left( H,\alpha \right) $ of $%
\alpha $-h\"{o}lderian functions for some $\alpha \in \left( 0,1\right] $:
if $f\in \mathcal{H}\left( H,\alpha \right) $, then for all $\left(
x,y\right) \in \mathcal{Z}^{2}$ 
\begin{equation*}
\left\vert f\left( x\right) -f\left( y\right) \right\vert \leq H\left\vert
x-y\right\vert ^{\alpha }\text{ .}
\end{equation*}%
In that case, $\beta _{+}=2\alpha $ is convenient and AIC-type procedures
are adaptive to the parameters $H$ and $\alpha $, see Castellan \cite%
{Castellan:99}.

In assumption (\textbf{Ap}) of Theorem \ref{theorem_opt_pen_MLE} we also
assume that the bias $\mathcal{K}\left( f_{\ast },f_{m}\right) $ is bounded
from below by a power of the dimension $D_{m}$ of the model $m$. This
hypothesis is in fact quite classical as it has been used by Stone \cite%
{Stone:85} and Burman \cite{Burman:02}\ for the estimation of density on
histograms and also by Arlot and Massart \cite{ArlotMassart:09} and Arlot 
\cite{Arl:2008a}, \cite{Arl:09} in the regression framework. Combining Lemma
1 and 2 of Barron and Sheu \cite{BarSheu:91} -\ see also inequality (\ref%
{margin_like_opt}) of Proposition \ref{lemma_margin_like_bounded_below}\
below -- we can show that 
\begin{equation*}
\frac{1}{2}e^{-3\left\Vert \ln \left( \frac{f_{\ast }}{f_{m}}\right)
\right\Vert _{\infty }}\int_{\mathcal{Z}}\frac{\left( f_{m}-f_{\ast }\right)
^{2}}{f_{\ast }}d\mu \leq \mathcal{K}\left( f_{\ast },f_{m}\right)
\end{equation*}%
and thus assuming for instance that the target is uniformly bounded, $%
\left\Vert f_{\ast }\right\Vert _{\infty }\leq A_{\ast }$, we get 
\begin{equation*}
\frac{A_{\min }^{3}}{2A_{\ast }^{4}}\int_{\mathcal{Z}}\left( f_{m}-f_{\ast
}\right) ^{2}d\mu \leq \mathcal{K}\left( f_{\ast },f_{m}\right) \text{ .}
\end{equation*}%
Now, since in the case of histograms the KL projection $f_{m}$ is also the $%
L_{2}\left( \mu \right) $ projection of $f_{\ast }$ onto $m$, we can apply
Lemma 8.19 in Section 8.10 of Arlot \cite{Arlot:07} to show that assumption (%
\textbf{Ap}) is indeed satisfied for $\beta _{-}=1+\alpha ^{-1}$, in the
case where $\mathcal{Z}$ is the unit interval, $\mu =\leb$ is the Lebesgue
measure on the unit interval, the partitions $\Lambda _{m}$ are regular and
the density $f_{\ast }$ is a non-constant $\alpha $-h\"{o}lderian function.

The proof of Theorem \ref{theorem_opt_pen_MLE} and further descriptions of
the behavior of the procedure can be found in the supplementary material,
Section \ref{section_oracle_proofs}.

\section{Probabilistic and Analytical Tools}
\label{section_prob_tools}

In this section we set out some general results that are of independent interest and serve as tools for the mathematical description of our statistical procedure.

The first two sections contain new or improved concentration inequalities, for the chi-square statistics (Section \ref{section_chi_square}) and for general log-densities (Section \ref{section_concentration_inequalities}). We
established in Section \ref{section_margin_like_relations} some results that are related to the so-called margin relation in statistical learning and that are analytical in nature.

\subsection{Chi-Square Statistics' Concentration}
\label{section_chi_square}

The chi-square statistics plays an essential role in the proofs related to
Section \ref{section_risks_bounds_MLE}. Let us recall its definition.

\begin{definition}
Given some histogram model $m$, the statistics $\chi _{n}^{2}\left( m\right) 
$ is defined by%
\begin{equation*}
\chi _{n}^{2}\left( m\right) =\int_{\mathcal{Z}}\frac{\left( \hat{f}%
_{m}-f_{m}\right) ^{2}}{f_{m}}d\mu =\sum_{I\in m}\frac{\left( P_{n}\left(
I\right) -P\left( I\right) \right) ^{2}}{P\left( I\right) }\text{ .}
\end{equation*}
\end{definition}

The following proposition provides an improvement upon the previously best known concentration inequality for the right tail of the chi-square
statistics, available in \cite{Castellan:03}---see also \cite[Proposition 7.8]{Massart:07} and \cite[Theorem 12.13]{BouLugMas:13}.

\begin{proposition}
\label{prop_upper_dev_chi_2} For any $x,\theta >0$, it holds 
\begin{equation}
\mathbb{P}\left( \chi _{n}\left( m\right) \boldsymbol{1}_{\Omega _{m}\left(
\theta \right) }\geq \sqrt{\frac{D_{m}}{n}}+\left( 1+\sqrt{2\theta }+\frac{%
\theta }{6}\right) \sqrt{\frac{2x}{n}}\right) \leq \exp \left( -x\right) 
\text{ ,}  \label{upper_chi_part}
\end{equation}%
where we set $\Omega _{m}\left( \theta \right) =\bigcap_{I\in m}\left\{
\left\vert P_{n}\left( I\right) -P\left( I\right) \right\vert \leq \theta
P\left( I\right) \right\} $. More precisely, for any $x,\theta >0$, it holds
with probability at least $1-e^{-x}$,%
\begin{equation}
\chi _{n}\left( m\right) \boldsymbol{1}_{\Omega _{m}\left( \theta \right) }<%
\sqrt{\frac{D_{m}}{n}}+\sqrt{\frac{2x}{n}}+2\sqrt{\frac{\theta }{n}}\left( 
\sqrt{x}\wedge \left( \frac{xD_{m}}{2}\right) ^{1/4}\right) +\frac{\theta }{3%
}\sqrt{\frac{x}{n}}\left( \sqrt{\frac{x}{D_{m}}}\wedge \frac{1}{\sqrt{2}}%
\right) \text{ .}  \label{upper_chi_gene}
\end{equation}
\end{proposition}

The proof of Theorem \ref{prop_upper_dev_chi_2} can be found in Section \ref%
{ssection_chisquare_proof}.

Let us detail its relationship with the bound in Proposition 7.8 in \cite%
{Massart:07}, which is: for any $x,\varepsilon >0$,%
\begin{equation}
\mathbb{P}\left( \chi _{n}\left( m\right) \boldsymbol{1}_{\Omega _{m}\left(
\varepsilon ^{2}/\left( 1+\varepsilon /3\right) \right) }\geq \left(
1+\varepsilon \right) \left( \sqrt{\frac{D_{m}}{n}}+\sqrt{\frac{2x}{n}}%
\right) \right) \leq \exp \left( -x\right) \text{ }.  \label{bound_Massart}
\end{equation}%
By taking $\theta =\varepsilon ^{2}/\left( 1+\varepsilon /3\right) >0$, we
get $\varepsilon =\theta /6+\sqrt{\theta ^{2}/36+\theta }>\theta /6+\sqrt{%
\theta }>0$. Assume that $D_{m}\geq 2x$. We obtain by (\ref{upper_chi_part}%
), with probability at least $1-\exp \left( -x\right) $,%
\begin{align*}
\chi _{n}\left( m\right) \boldsymbol{1}_{\Omega _{m}\left( \theta \right)
}<& \sqrt{\frac{D_{m}}{n}}+\left( 1+\sqrt{2\theta }+\frac{\theta }{6}\right) 
\sqrt{\frac{2x}{n}} \\
\leq & \sqrt{\frac{D_{m}}{n}}+\left( 1+\sqrt{2}\varepsilon \right) \sqrt{%
\frac{2x}{n}} \\
<& \left( 1+\varepsilon \right) \left( \sqrt{\frac{D_{m}}{n}}+\sqrt{\frac{2x%
}{n}}\right) \text{ .}
\end{align*}%
So in this case, inequality (\ref{upper_chi_part}) improves upon (\ref%
{bound_Massart}). Notice that in our statistical setting (see Theorem \ref%
{opt_bounds}) the restriction on $D_{m}$ is as follows $D_{m}\leq
A_{+}n\left( \ln (n+1)\right) ^{-1}$. Furthermore, in our proofs we apply (%
\ref{upper_chi_part}) with $x$ proportional to $\ln (n+1)$ (see the proof of
Theorem \ref{opt_bounds}). Hence, for a majority of models, we have $x\ll
D_{m}$ and so 
\begin{equation*}
\sqrt{\frac{2x}{n}}\ll \sqrt{\frac{D_{m}}{n}}\text{ .}
\end{equation*}%
As a result, the bounds that we obtain in Theorem \ref{opt_bounds} by the
use of Inequality (\ref{upper_chi_part}) are substantially better than the
bounds we would obtain by using Inequality (\ref{bound_Massart}) of \cite%
{Massart:07}. More precisely, the term $\sqrt{D_{m}\ln (n+1)/n}$ in (\ref%
{def_A_0_MLE-1}) would be replaced by $\left( D_{m}\ln (n+1)/n\right) ^{1/4}$%
, thus changing the order of magnitude for deviations of the excess risks.

Now, if $D_{m}\leq 2x$ then (\ref{upper_chi_part}) gives that with
probability at least $1-e^{-x}$,%
\begin{align*}
\chi _{n}\left( m\right) \boldsymbol{1}_{\Omega _{m}\left( \theta \right) }
<&\sqrt{\frac{D_{m}}{n}}+\sqrt{\frac{2x}{n}}+2\sqrt{\frac{\theta }{n}}\left( 
\frac{xD_{m}}{2}\right) ^{1/4}+\frac{\theta }{6}\sqrt{\frac{2x}{n}} \\
\leq &\left( 1+\frac{\sqrt{\theta }}{2^{1/4}}\right) \sqrt{\frac{D_{m}}{n}}%
+\left( 1+\frac{\sqrt{\theta }}{2^{3/4}}+\frac{\theta }{6}\right) \sqrt{%
\frac{2x}{n}} \\
<&\left( 1+\varepsilon \right) \left( \sqrt{\frac{D_{m}}{n}}+\sqrt{\frac{2x}{%
n}}\right) \text{ .}
\end{align*}%
So in this case again, inequality (\ref{upper_chi_gene}) improves upon (\ref%
{bound_Massart}), with an improvement that can be substantial depending on
the value of the ratio $x/D_{m}$.

The following result describes the concentration from the left of the
chi-square statistics and is proved in Section \ref{ssection_chisquare_proof}%
.

\begin{proposition}
\label{prop_dev_gauche_chi}Let $\alpha ,$ $A_{\Lambda }>0$. Assume $%
0<A_{\Lambda }\leq D_{m}\inf_{I\in m}\left\{ P\left( I\right) \right\} $.
Then there exists a positive constant $A_{g}$ depending only on$\,A_{\Lambda
}$ and$\,\alpha $ such that 
\begin{equation*}
\mathbb{P}\left( \chi _{n}\left( m\right) \leq \left( 1-A_{g}\left( \sqrt{%
\frac{\ln (n+1)}{D_{m}}}\vee \frac{\sqrt{\ln (n+1)}}{n^{1/4}}\right) \right) 
\sqrt{\frac{D_{m}}{n}}\right) \leq \left( n+1\right) ^{-\alpha }\text{ }.
\end{equation*}
\end{proposition}

\subsection{Bernstein type concentration inequalities for log-densities\label%
{section_concentration_inequalities}}

The following proposition gives concentration inequalities for the empirical
bias at the right of its mean.

\begin{proposition}
\label{delta_bar_dev_droite_MLE} Consider a density $f\in \mathcal{S}$. We
have, for all $z\geq 0$, 
\begin{equation}
\mathbb{P}\left( P_{n}\left( \ln \left( \left. f\right/ f_{\ast }\right)
\right) \geq \frac{z}{n}\right) \leq \exp \left( -z\right) \text{ .}
\label{dev_droite_gene_2}
\end{equation}%
Moreover, if we can take a finite quantity $v$ which satisfies $v\geq \int
\left( f\vee f_{\ast }\right) \left( \ln \left( \frac{f}{f_{\ast }}\right)
\right) ^{2}d\mu $, we have for all $z\geq 0$,%
\begin{equation}
\mathbb{P}\left( \left( P_{n}-P\right) \left( \ln \left( \left. f\right/
f_{\ast }\right) \right) \geq \sqrt{\frac{2vz}{n}}+\frac{2z}{n}\right) \leq
\exp \left( -z\right) \text{ .}  \label{bernstein_dev_droite_2}
\end{equation}
\end{proposition}

One can notice, with Inequality (\ref{dev_droite_gene_2}), that the
empirical bias always satisfies some exponential deviations at the right of
zero. In the Information Theory community, this inequality is also known as
the ``No Hyper-compression Inequality'' (\cite{grunwald2007minimum}).

Inequality (\ref{bernstein_dev_droite_2}) seems to be new and takes the form
of a Bernstein-like inequality, even if the usual assumptions of Bernstein's
inequality are not satisfied. In fact, we are able to recover such a
behavior by inflating the usual variance to the quantity $v$.

We now turn to concentration inequalities for the empirical bias at the left of its mean.

\begin{proposition}
\label{delta_bar_dev_gauche_MLE}Let $r>0$. For any density $f\in \mathcal{S}$
and for all $z\geq 0$, we have%
\begin{equation}
\mathbb{P}\left( P_{n}\left( \ln \left( \left. f\right/ f_{\ast }\right)
\right) \leq -z/nr-\left( 1/r\right) \ln \left( P\left[ \left( \left.
f_{\ast }\right/ f\right) ^{r}\right] \right) \right) \leq \exp \left(
-z\right) \text{ }.  \label{delta_dev_gauche_2}
\end{equation}%
Moreover, if we can set a quantity $w_{r}$ which satisfies $w_{r}\geq \int
\left( \frac{f_{\ast }^{r+1}}{f^{r}}\vee f_{\ast }\right) \left( \ln \left( 
\frac{f}{f_{\ast }}\right) \right) ^{2}d\mu $ , then we get, for all $z\geq
0 $,%
\begin{equation}
\mathbb{P}\left( \left( P_{n}-P\right) \left( \ln \left( \left. f\right/
f_{\ast }\right) \right) \leq -\sqrt{\frac{2w_{r}z}{n}}-\frac{2z}{nr}\right)
\leq \exp \left( -z\right) \text{ }.  \label{delta_dev_gauche_gauss_2}
\end{equation}
\end{proposition}

\subsection{Margin-Like Relations}
\label{section_margin_like_relations}

Our objective in this section is to control the variance terms $v$ and $w_{r}$,
appearing respectively in Lemma \ref{delta_bar_dev_droite_MLE}\ and \ref%
{delta_bar_dev_gauche_MLE}, in terms of the KL divergence pointed on the
target $f_{\ast }$. This is done in Lemma \ref{lem:margin_like_unbounded}\
below under moment assumptions for $f_{\ast }$.

\begin{proposition}
\label{lem:margin_like_unbounded}Let $p>1$ and $c_{+},c_{-}>0$.\textup{\ }%
Assume that the density $f_{\ast }$ satisfies%
\begin{equation}
J:=\int_{\mathcal{Z}}f_{\ast }^{p}\left( \left( \ln \left( f_{\ast }\right)
\right) ^{2}\vee 1\right) d\mu <+\infty \text{ and }Q:=\int_{\mathcal{Z}}%
\frac{\left( \ln \left( f_{\ast }\right) \right) ^{2}\vee 1}{f_{\ast }^{p-1}}%
d\mu <+\infty  \label{moment_p_bis}
\end{equation}%
Take a density $f$ such that $0<c_{-}\leq {\inf_{z\in \mathcal{Z}}\left\{
f\left( z\right) \right\} \leq \sup_{z\in \mathcal{Z}}}\left\{ f\left(
z\right) \right\} \leq c_{+}<+\infty $. Then, for some $A_{MR,d}>0$ only
depending on $J,Q,p,c_{+}$ and $c_{-}$, it holds%
\begin{equation}
P\left[ \left( \frac{f}{f_{\ast }}\vee 1\right) \left( \ln \left( \frac{f}{%
f_{\ast }}\right) \right) ^{2}\right] \leq A_{MR,d}\times \mathcal{K}\left(
f_{\ast },f\right) ^{1-\frac{1}{p}}\text{ .}
\label{margin_like_soft_droite_gene}
\end{equation}%
More precisely, $A_{MR,d}=\left( 4c_{-}^{1-p}\left( \left( \ln c_{-}\right)
^{2}\vee 1\right) J+4c_{+}^{p}\left( \left( \ln c_{+}\right) ^{2}\vee
1\right) Q\right) ^{1/p}$ holds. For any $0<r\leq p-1$, we have the
following inequality,%
\begin{equation}
P\left[ \left( \frac{f_{\ast }}{f}\vee 1\right) ^{r}\left( \ln \left( \frac{f%
}{f_{\ast }}\right) \right) ^{2}\right] \leq A_{MR,g}\times \mathcal{K}%
\left( f_{\ast },f\right) ^{1-\frac{r+1}{p}}\text{ ,}
\label{margin_like_soft_gauche_gene}
\end{equation}%
available with $A_{MR,g}=\left( 4c_{-}^{1-p}\left( \left( \ln c_{-}\right)
^{2}\vee 1\right) J+2\left( \left( \ln \left( c_{+}\right) \right)
^{2}+J+Q\right) \right) ^{\frac{r+1}{p}}$.
\end{proposition}

Theorem \ref{lem:margin_like_unbounded} states that the variance terms,
appearing in the concentration inequalities of Section~\ref%
{section_concentration_inequalities}, are bounded from above, under moment
restrictions on the density $f_{\ast }$, by a power less than one of the KL
divergence pointed on $f_{\ast }$. The stronger are the moment assumptions,
given in (\ref{moment_p_bis}), the closer is the power to one. One can
notice that $J$ is a restriction on large values of $f_{\ast}$, whereas $Q$
is related to values of $f_{\ast }$ around zero.

We call these inequalities ``margin-like relations'' because of their
similarity with the margin relations known first in binary classification (%
\cite{MamTsy:99}, \cite{Tsy:04a}) and then extended to empirical risk
minimization (see \cite{MassartNedelec:06}\ and \cite{Arl_Bar:2008}\ for
instance). Indeed, from a general point of view, margin relations relate the
variance of contrasted functions (logarithm of densities here) pointed on
the contrasted target to a function (in most cases, a power) of their excess
risk.

Now we're tightening the restrictions on the values of $f_{\ast }$ around zero.
Indeed, we ask in the following lemma that the target is uniformly bounded
away from zero.

\begin{proposition}
\label{lemma_margin_like_bounded_below}Let $p>1$ and $A_{\min
},c_{+},c_{-}>0 $.\textup{\ }Assume that the density $f_{\ast }$ satisfies%
\begin{equation*}
J:=\int_{\mathcal{Z}}f_{\ast }^{p}\left( \left( \ln \left( f_{\ast }\right)
\right) ^{2}\vee 1\right) d\mu <+\infty \text{ and }0<{A_{\min }\leq
\inf_{z\in \mathcal{Z}}f_{\ast }\left( z\right) }\text{ \textup{.}}
\end{equation*}%
Then there exists a positive constant $A_{MR,-}$ only depending on $A_{\min
},J,r$ and $p$ such that, for any $m\in \mathcal{M}_{n}$,%
\begin{equation}
P\left[ \left( \frac{f_{m}}{f_{\ast }}\vee 1\right) \left( \ln \left( \frac{%
f_{m}}{f_{\ast }}\right) \right) ^{2}\right] \leq A_{MR,-}\times \mathcal{K}%
\left( f_{\ast },f_{m}\right) ^{1-1/p}  \label{margin_like_soft}
\end{equation}%
and for any $0<r\leq p-1$, 
\begin{equation}
P\left[ \left( \frac{f_{\ast }}{f_{m}}\vee 1\right) ^{r}\left( \ln \left( 
\frac{f_{m}}{f_{\ast }}\right) \right) ^{2}\right] \leq A_{MR,-}\times 
\mathcal{K}\left( f_{\ast },f_{m}\right) ^{1-\frac{r+1}{p}}\text{ .}
\label{margin_like_gauche_soft}
\end{equation}%
If moreover $\ln \left( f_{\ast }\right) \in L_{\infty }\left( \mu \right) $%
, i.e. $0<{A_{\min }\leq \inf_{z\in \mathcal{Z}}}$\textup{$f_{\ast }$}${%
\left( z\right) \leq \left\Vert f_{\ast }\right\Vert _{\infty }<+\infty }$,
then there exists $\tilde{A}>0$ only depending on $r,{A_{\min }}$ and ${%
\left\Vert f_{\ast }\right\Vert _{\infty }}$ such that, for any $m\in 
\mathcal{M}_{n}$,%
\begin{equation}
P\left[ \left( \frac{f_{m}}{f_{\ast }}\vee 1\right) \left( \ln \left( \frac{%
f_{m}}{f_{\ast }}\right) \right) ^{2}\right] \vee P\left[ \left( \frac{%
f_{\ast }}{f_{m}}\vee 1\right) ^{r}\left( \ln \left( \frac{f_{m}}{f_{\ast }}%
\right) \right) ^{2}\right] \leq \tilde{A}\times \mathcal{K}\left( f_{\ast
},f_{m}\right) \text{ }.  \label{margin_like_opt}
\end{equation}
\end{proposition}

It is worth noting that Lemma \ref{lemma_margin_like_bounded_below} is
stated only for projections $f_{m}$ because we actually take advantage of
their special form (as local means of the target) in the proof of the lemma.
The benefit, compared to results of Lemma \ref{lem:margin_like_unbounded},
is that Inequalities (\ref{margin_like_soft}), (\ref{margin_like_gauche_soft}%
) and (\ref{margin_like_opt}) do not involve assumptions on the values of $%
f_{m}$ (and in particular they do not involve the sup-norm of $f_{m}$).

\section{Experiments}
\label{section_expe}

A simulation study is conducted to compare the numerical performance of the model selection procedures we discussed. We demonstrate the usefulness of our procedure on simulated data examples. The numerical experiments were performed using R. 

\begin{figure}[t]
\centering
\includegraphics[width=0.4\textwidth]{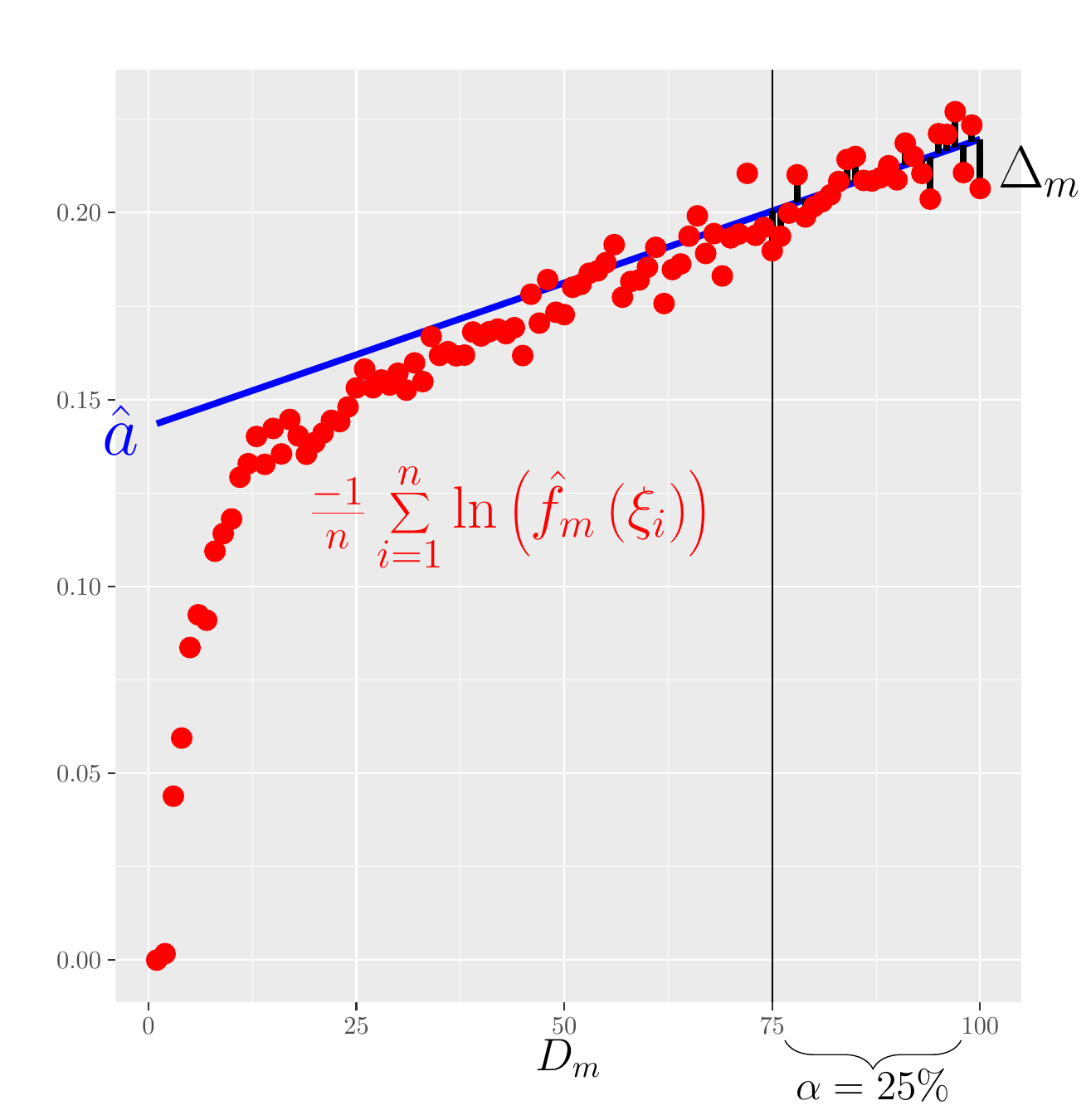}
\includegraphics[width=0.4\textwidth]{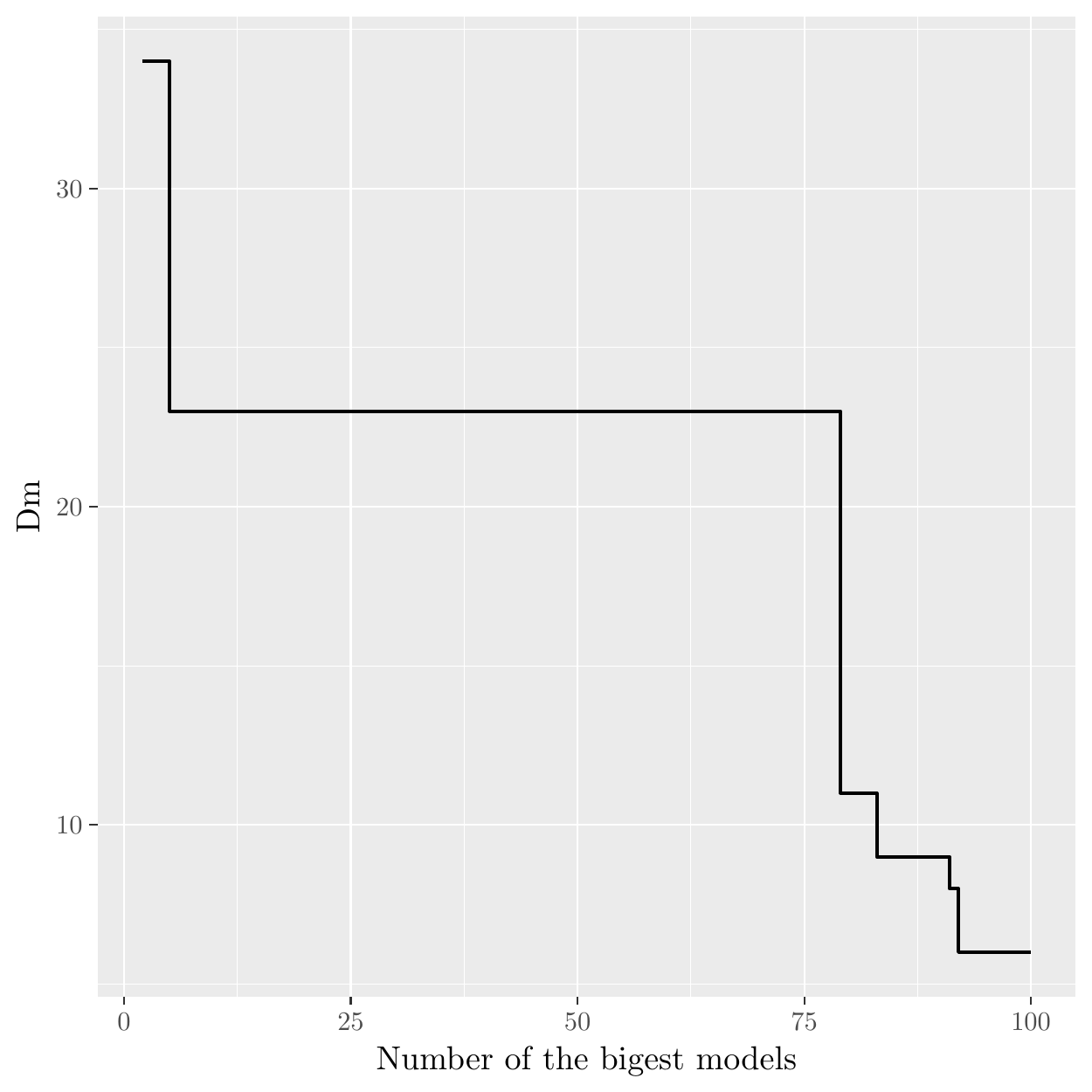}
\caption{Estimation of the over-penalization constant.}
\label{fig:proportion}
\end{figure}

\subsection{Experimental Setup}

We have compared the numerical performance of our procedure with the classic methods of penalisation of the literature on several densities. In particular, we consider the estimator of \cite{BirRozen:06} and AICc (\citep{Sugiura:78, HurTsai:89}). We also report on AIC's behaviour. In the following, we name the procedure of \cite{BirRozen:06} by BR, and our criterion AIC$_1$ when the constant $C=1$ in \eqref{pen_a} and AIC$_a$ for a fully adaptive procedure which will be detailed below. More specifically, the performance of the following four model selection methods were compared:
\begin{enumerate}
\item AIC:
\begin{equation*}
\widehat{m}_{\mathrm{AIC}}\in\arg\min_{m\in \mathcal{M}_{n}}\left\{\mathrm{crit_{\mathrm{AIC}}}(m)\right\},
\end{equation*}
with
\[
\mathrm{crit_{\mathrm{AIC}}}(m) = P_{n}(\gamma(\widehat{f}_{m}))+\frac{D_m}{n}.
\]
\item AICc:
\begin{equation*}
\widehat{m}_{\mathrm{AICc}}\in\arg\min_{m\in \mathcal{M}_{n}}\left\{P_{n}(\gamma(\widehat{f}_{m}))+\pen_{\mathrm{AICc}}(m)\right\},
\end{equation*}
with
\[
\pen_{\mathrm{AICc}}(m) = \frac{D_m}{n-D_m-1}.
\]
\item BR:
\begin{equation*}
\widehat{m}_{\mathrm{BR}}\in\arg\min_{m\in \mathcal{M}_{n}}\left\{P_{n}(\gamma(\widehat{f}_{m}))+\pen_{\mathrm{BR}}(m)\right\},
\end{equation*}
with
\[
\pen_{\mathrm{BR}}(m) = \frac{ (\log{D_m})^{2.5}}{n}
\]
\item AIC$_1$:
\begin{equation*}
\widehat{m}_{\mathrm{AIC}_1}\in\arg\min_{m\in \mathcal{M}_{n}}\left\{P_{n}(\gamma(\widehat{f}_{m}))+\pen_{\mathrm{AIC}_1}(m)\right\},
\end{equation*}
with
\[
\pen_{\mathrm{AIC}_1}(m) = 1\times\max \left\{ \sqrt{\frac{D_{m}\ln (n+1)}{n%
}};\sqrt{\frac{\ln (n+1)}{D_{m}}};\frac{\ln (n+1)}{D_{m}}\right\}
\frac{D_{m}}{n}\text{ ,}
\]
\item  AIC$_a$:
\begin{equation*}
\widehat{m}_{\mathrm{AIC}_a}\in\arg\min_{m\in \mathcal{M}_{n}}\left\{P_{n}(\gamma(\widehat{f}_{m}))+\pen_{\mathrm{AIC}_a}(m)\right\},
\end{equation*}
with
\[
\pen_{\mathrm{AIC}_a}(m) = \hat{C}\max \left\{ \sqrt{\frac{D_{m}\ln (n+1)}{n%
}};\sqrt{\frac{\ln (n+1)}{D_{m}}};\frac{\ln (n+1)}{D_{m}}\right\}
\frac{D_{m}}{n}\text{ ,}
\]
where $\hat{C}={\rm median}_{\alpha \in \mathcal{P}} \hat{C}_{\alpha}$, with $\hat{C}_{\alpha}={\rm median}_{m\in \mathcal{M}_{\alpha}} |\hat{C}_m|$, where
\begin{equation*}
\hat{C}_m=\frac{\Delta_m}{\max \left\{ \sqrt{\frac{D_m}{n}};\sqrt{\frac{1}{D_{m}}}\right\}\frac{D_m}{2}} \text{ ,}
\end{equation*}
$\Delta_m$ is the least-squares distance between the opposite of the empirical risk $-P_{n}(\gamma(\widehat{f}_{m}))$ and a fitted line of equation $y=xD_m/n+\hat{a}$ (see Figure \ref{fig:proportion}), $\mathcal{P}$ is the set of proportions $\alpha$ corresponding to the longest plateau of equal selected models when using penalty \eqref{pen_a} with constant $C=\hat{C}_{\alpha}$ and $\mathcal{M}_{\alpha}$ is the set of models in the collection associated to the proportion $\alpha$ of the largest dimensions.
 
\end{enumerate}
\medskip

Let us briefly explain the ideas underlying the design of the fully adaptive AIC$_a$ procedure. According to the definition of penalty $\rm{pen}_{\rm{opt},\beta}$ given in \eqref{pen_opt_quantile} the constant $C$ in our over-penalization procedure \eqref{pen_a} should ideally estimate some normalized deviations of the sum of the excess risk and the empirical excess risk on the models of the collection. Based on Theorem \ref{opt_bounds},we can also assume that the deviations of excess risk and excess empirical risk are of the same order. Moreover, considering the largest models in the collection neglects questions of bias and, therefore, the median of the normalized deviations of the empirical risk around its mean for the largest models should be a reasonable estimator of the $C$ constant. Now the problem is to give a tractable definition to the "largest models" in the collection. To do this, we choose a proportion $\alpha$ of the largest dimensions of the models at hand and calculate using these models an estimator $\hat{C}_{\alpha}$ of the constant $C$ in \eqref{pen_a}. We then proceed for each  $\alpha$ in a grid of values between $0$ and $1$ to a model selection step by over-penalization using the constant $C=\hat{C}_{\alpha}$. This gives us a graph of the selected dimensions with respect to the proportions (see Figure~\ref{fig:proportion}). Finally we define our over-penalization constant $\hat{C}$ as the median of the values of the constants $\hat{C}_{\alpha}$,  $\alpha \in \mathcal{P}$ where $\mathcal{P}$ is the largest plateau in the graph of the selected dimensions with respect to proportions $\alpha$. 

The models that we used along the experiments are made of histograms densities defined a on a regular partition of the interval $[0,1]$ (with the exception of the density Isosceles triangle which is supported on $[-1,1]$).  We show the performance of the proposed method for a set of four test distributions (see Figure~\ref{fig:orig}) and described in the \textit{benchden}\footnote{Available on the CRAN \url{http://cran.r-project.org}.} R-package \cite{Mildenberger-Weinert2012} which provides an implementation of the distributions introduced in \cite{berlinet1994comparison}.

\begin{figure}[t]
\centering
\subfigure[Isosceles triangle]{\includegraphics[width=0.24\textwidth]{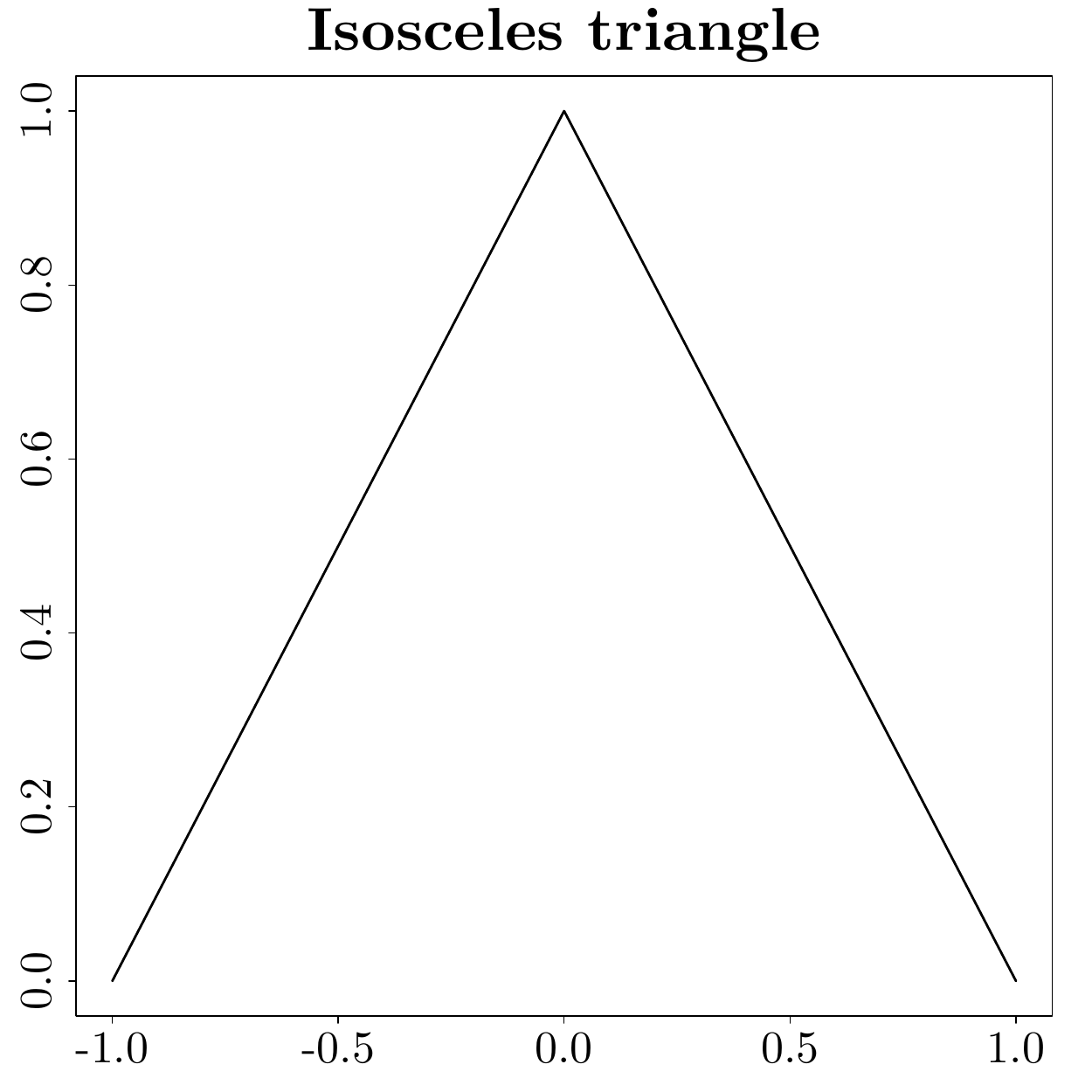}} %
\subfigure[Bilogarithmic peak]{\includegraphics[width=0.24\textwidth]{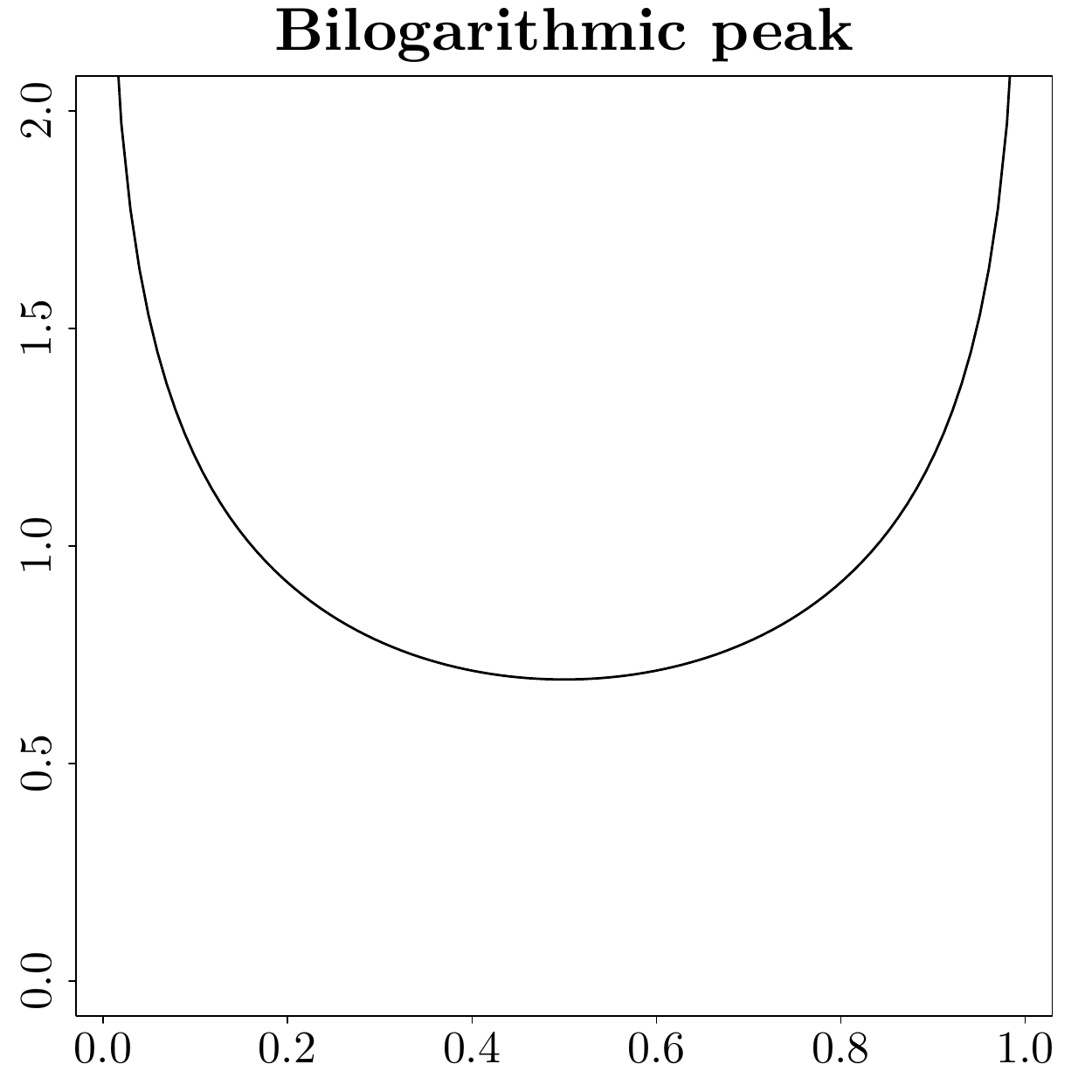}} %
\subfigure[Beta (2,2)]{\includegraphics[width=0.24\textwidth]{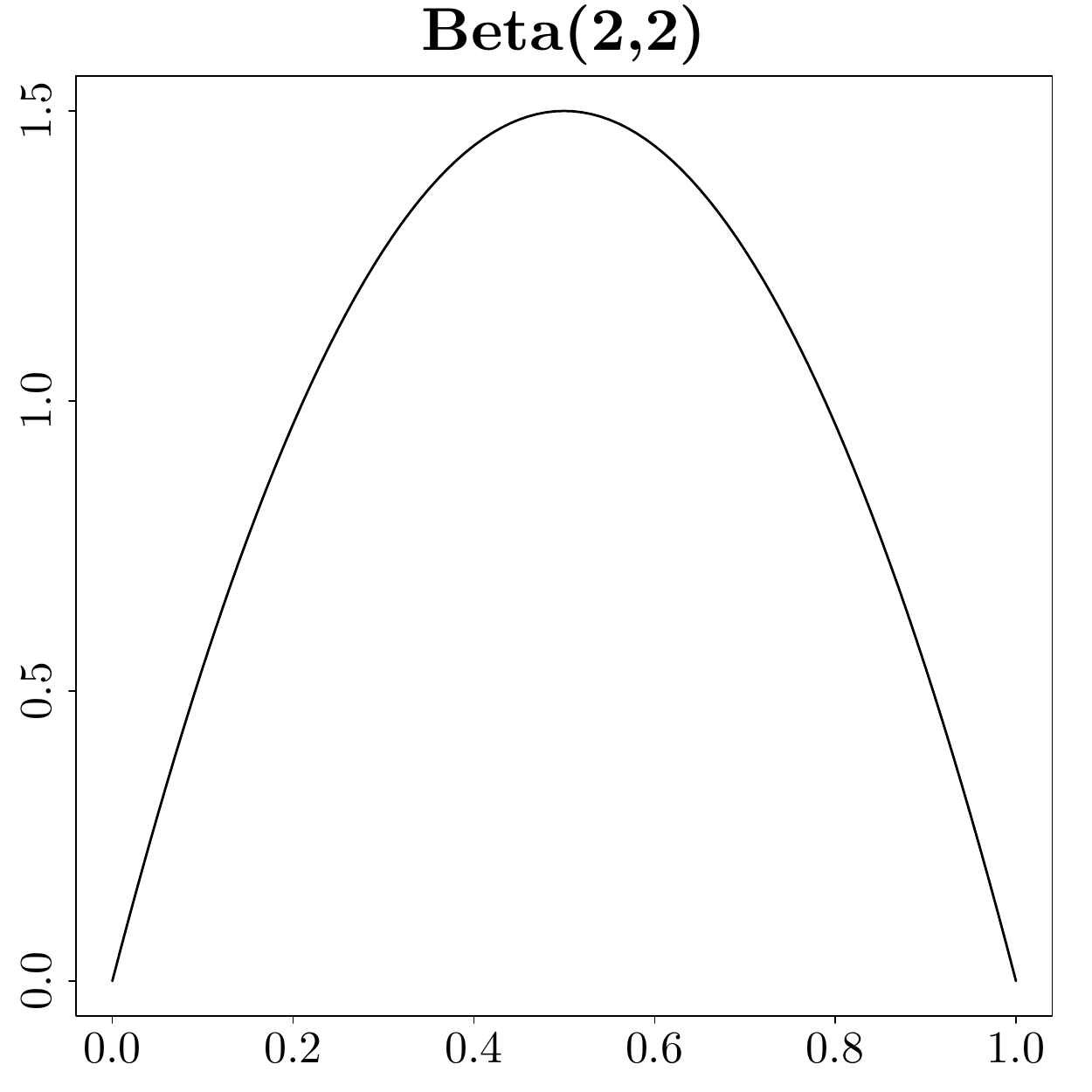}} %
\subfigure[Infinite peak]{\includegraphics[width=0.24\textwidth]{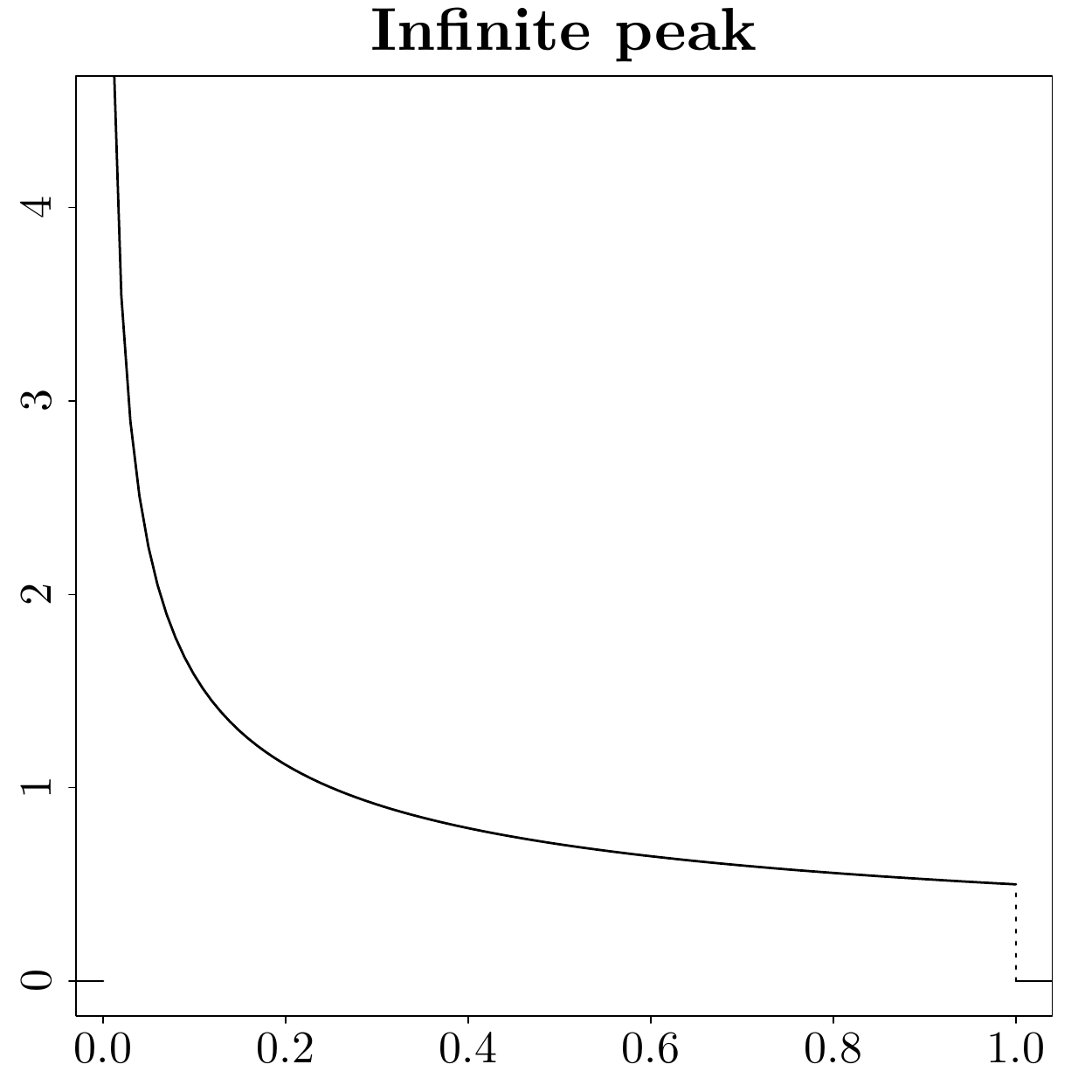}}
\caption{Test densities $f$.}
\label{fig:orig}
\end{figure}

\subsection{Results}
We compared procedures on $N=100$ independent data sets of size $n$ ranging from $50$ to $1000$. We estimate the quality of the model-selection strategies using the median KL divergence. Boxplots were made of the KL risk  over the $N$ trials. The horizontal lines of the boxplot indicate the $5\%, 25\%, 50\%, 75\%$, and $95\%$ quantiles of the error distribution. The median value of  AIC (horizontal black line) is also superimposed for visualization purposes.

It can be seen from Figure~\ref{fig:result} that, as expected, for each method and in all cases, the KL divergence decreases as the sample size increases. We also see clearly that there is generally a substantial advantage to modifying AIC for sample sizes smaller than $1000$. 

Moreover, none of the methods  is more effective than the others in all cases. However,  in almost all cases, AIC$_1$ always had a lower median KL than the others. We therefore recommend using  AIC$_1$ rather than AIC$_c$. Indeed, compared to AIC$_1$, it seems that AIC$_c$ is not penalizing enough,which translates into a worse performance for samples equal to $50$ and $100$. Moreover, both procedures have exactly the same calculation costs.

On the contrary, it seems that the BR criterion penalizes too much. As a result, its performance deteriorates relative to other methods as the sample size increases.

Finally, the behavior of AIC$_a$ is quite good in general, but it is often a little less efficient than AIC$_1$. We interpret this fact by postulating that the extra computations that we engage in AIC$_a$ to remove any hyper-parameter from the procedure actually induces extra variance in the estimation. 

\begin{figure}[t!]
\centering
\subfigure[Isosceles triangle]{
 \includegraphics[width=0.24\textwidth]{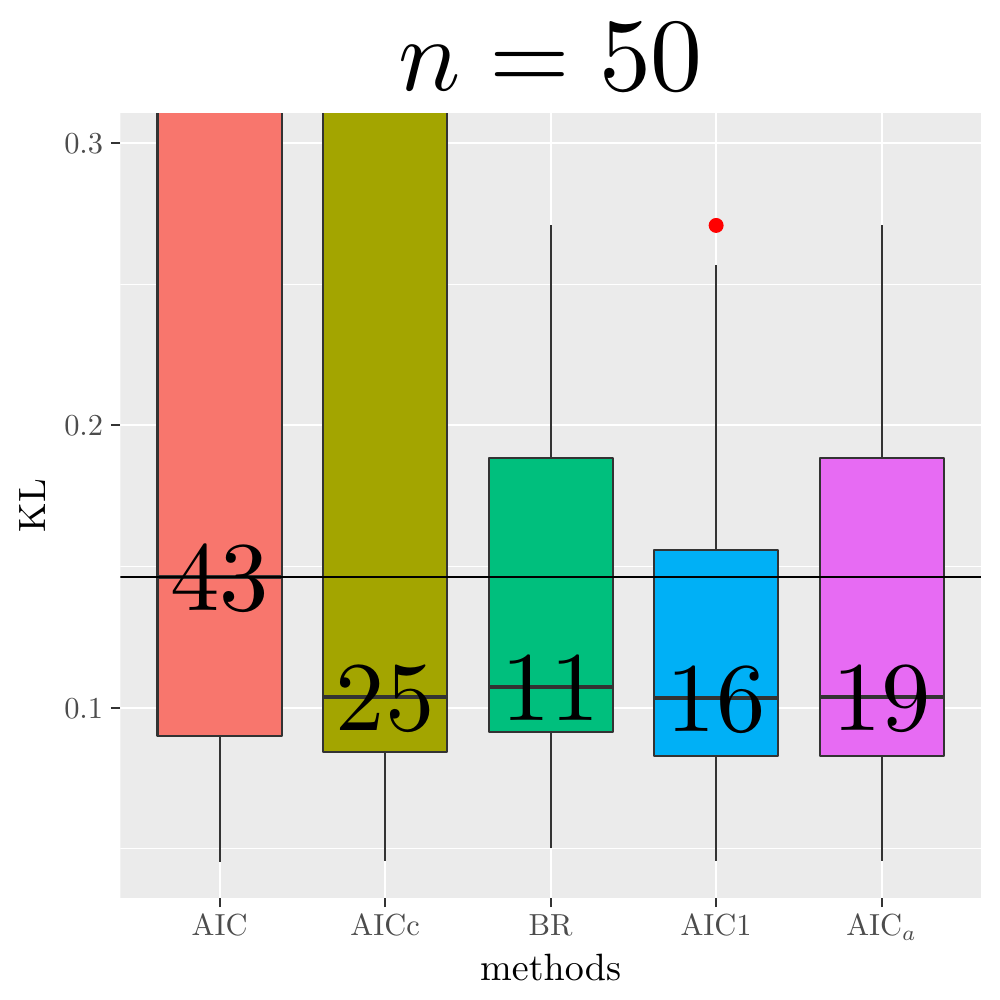}
 \includegraphics[width=0.24\textwidth]{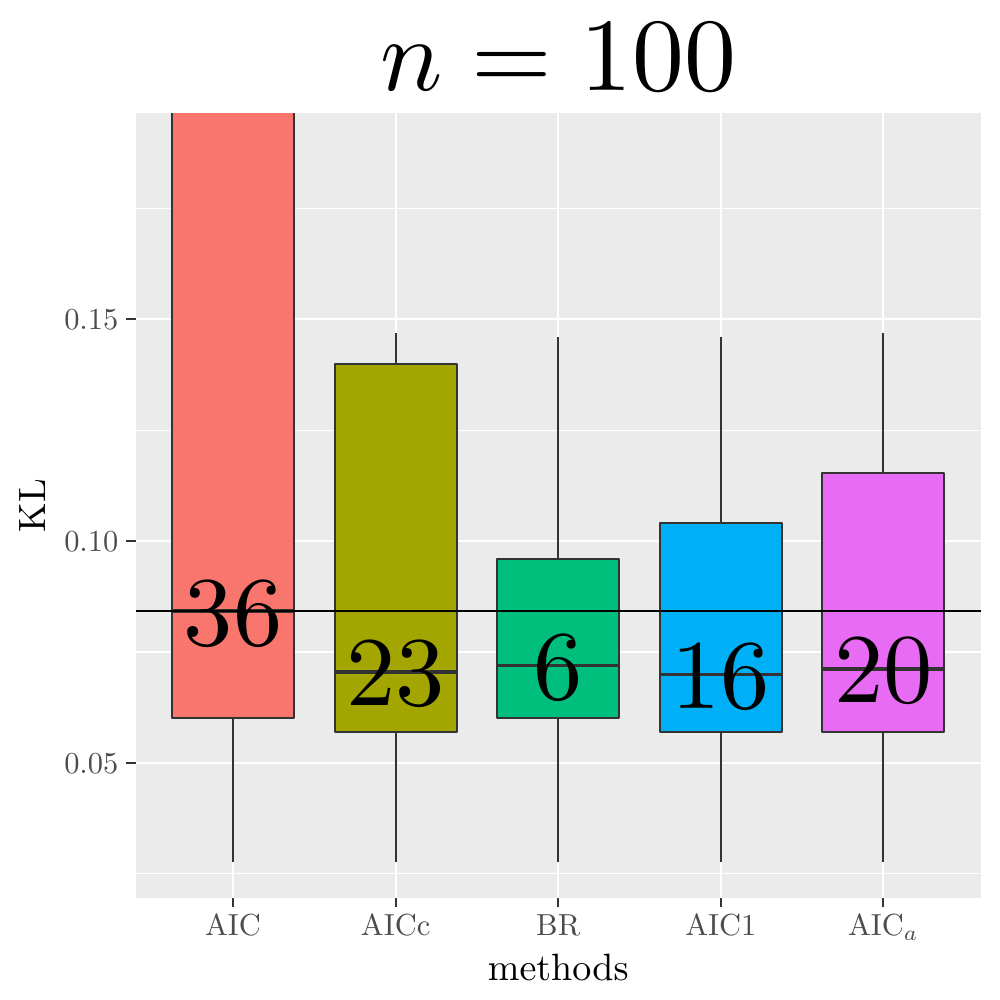}
 \includegraphics[width=0.24\textwidth]{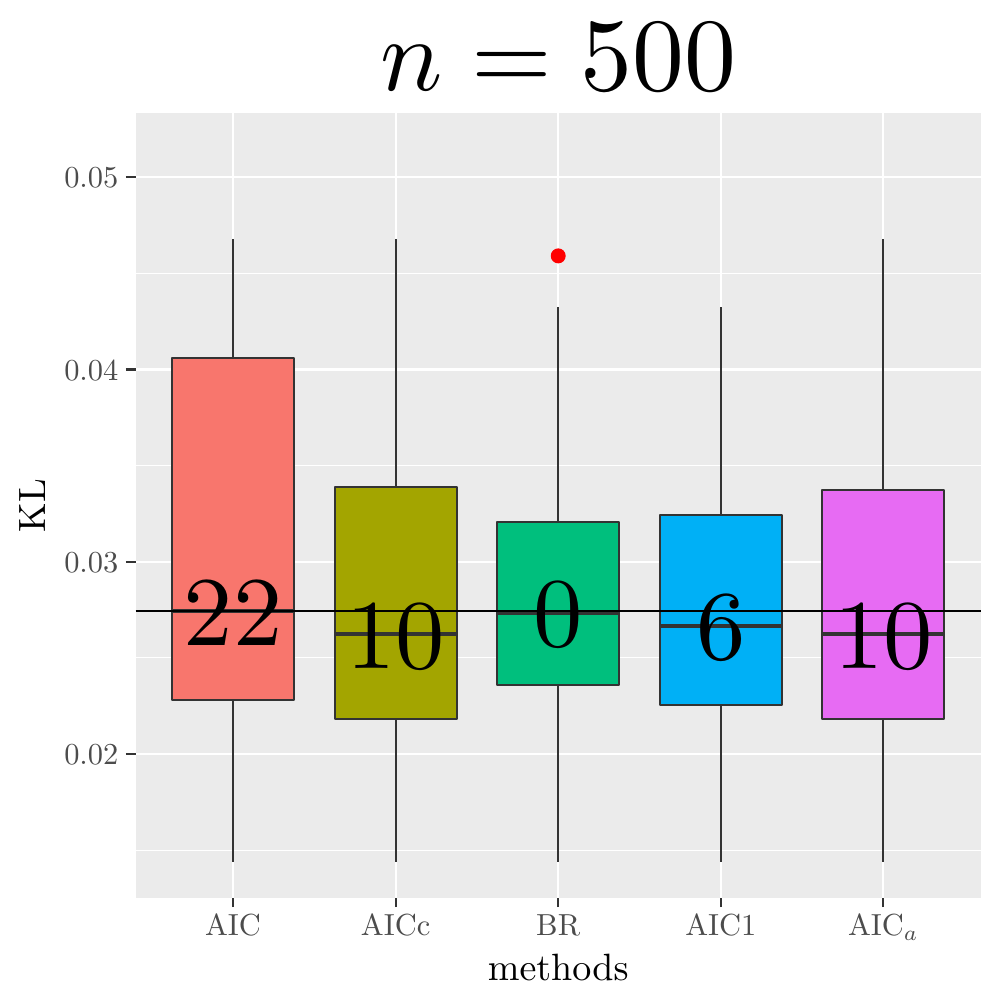}
 \includegraphics[width=0.24\textwidth]{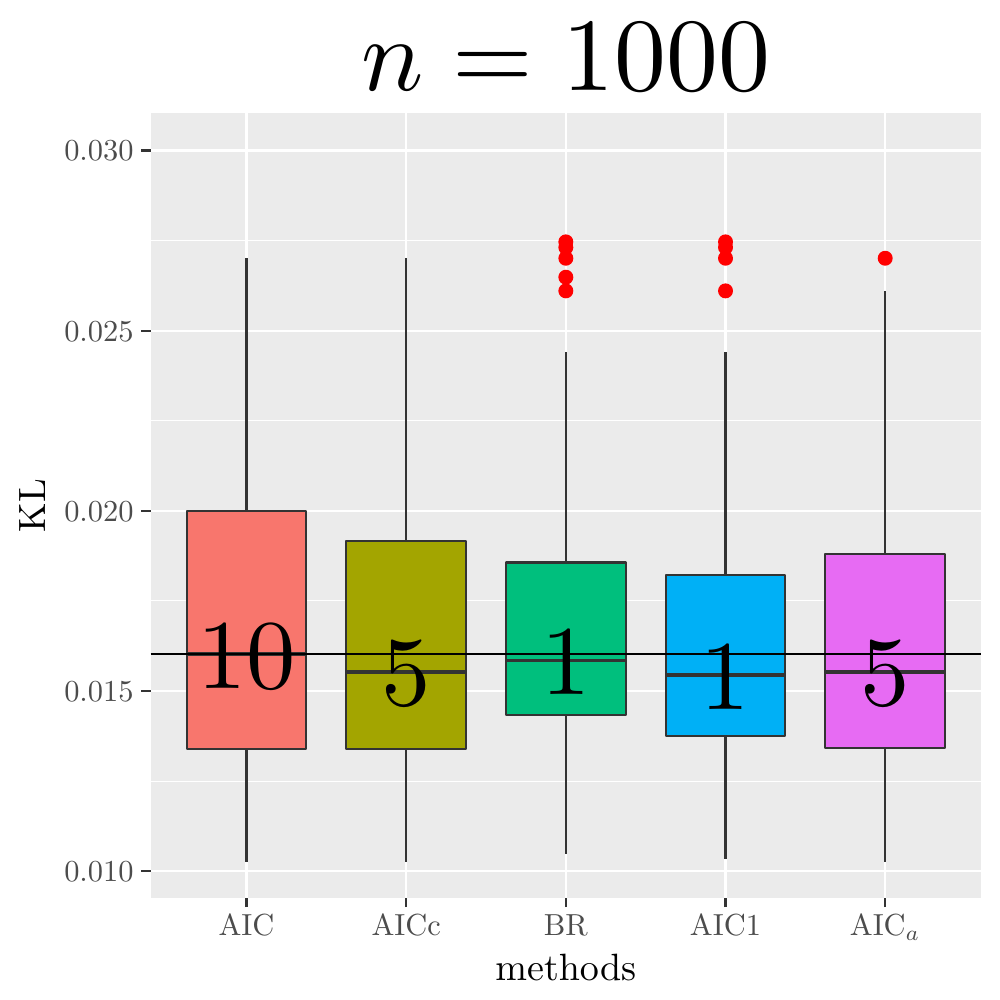}
 }
\subfigure[Bilogarithmic peak]{
 \includegraphics[width=0.24\textwidth]{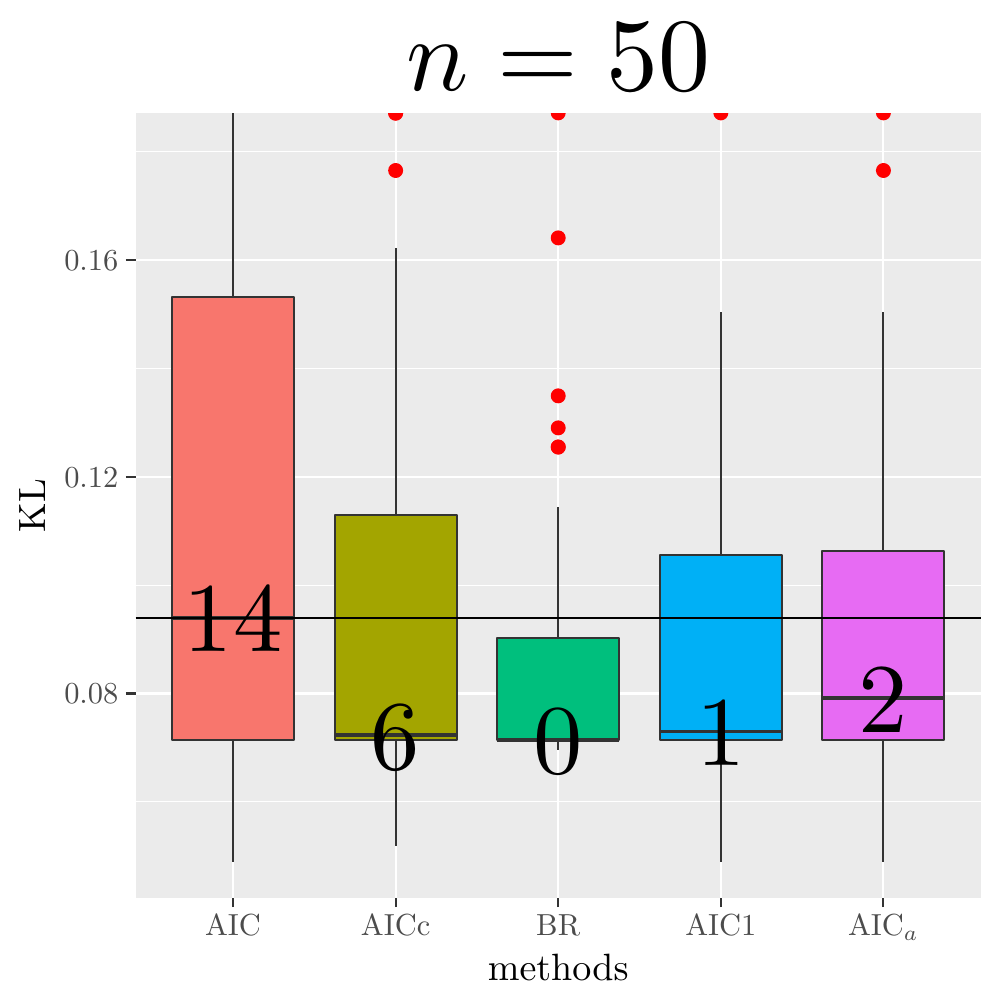}
 \includegraphics[width=0.24\textwidth]{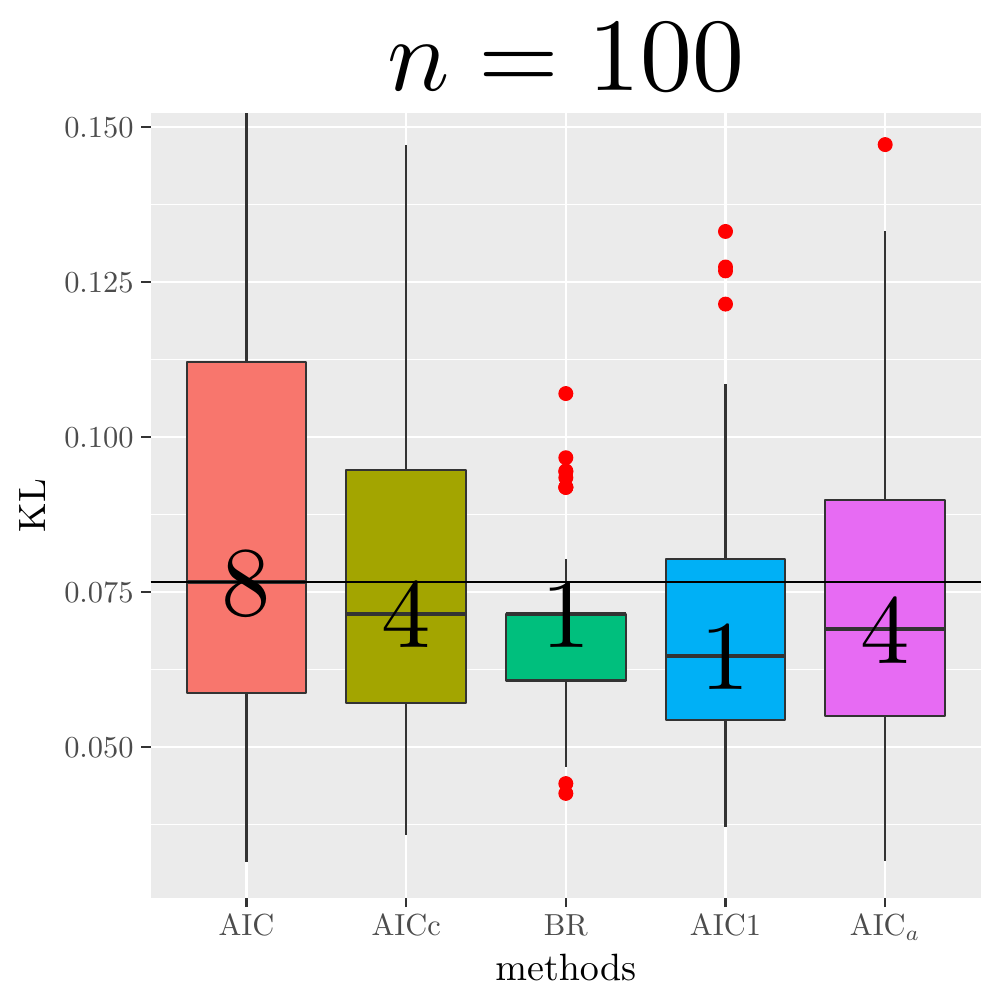}
 \includegraphics[width=0.24\textwidth]{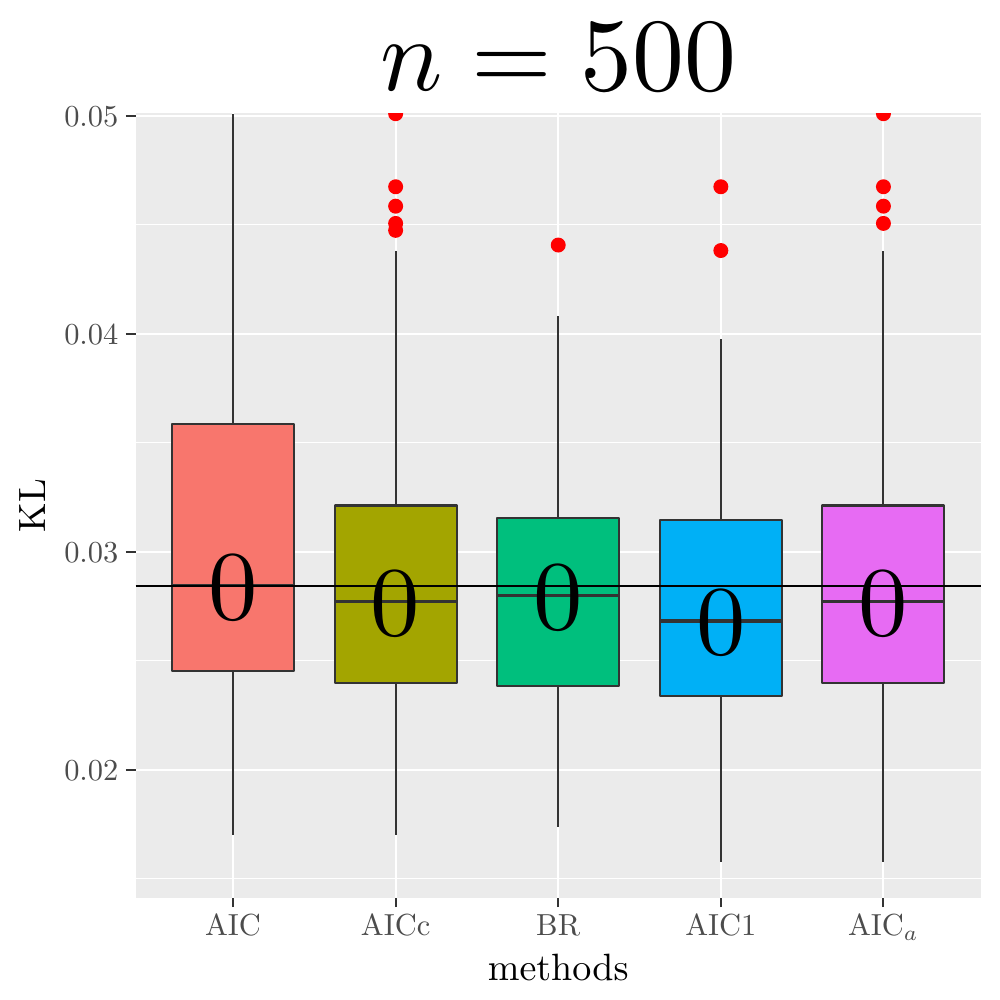}
 \includegraphics[width=0.24\textwidth]{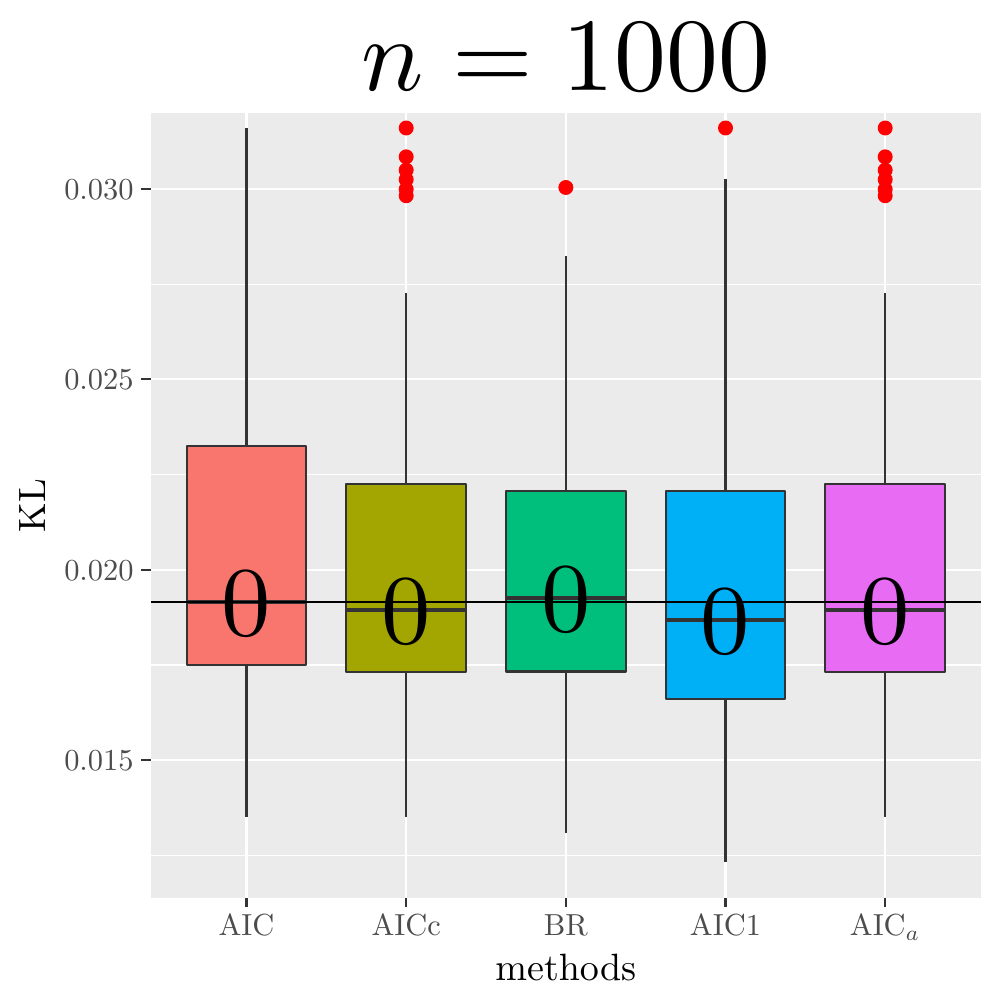}
 }
\subfigure[Beta (2,2)]{
 \includegraphics[width=0.24\textwidth]{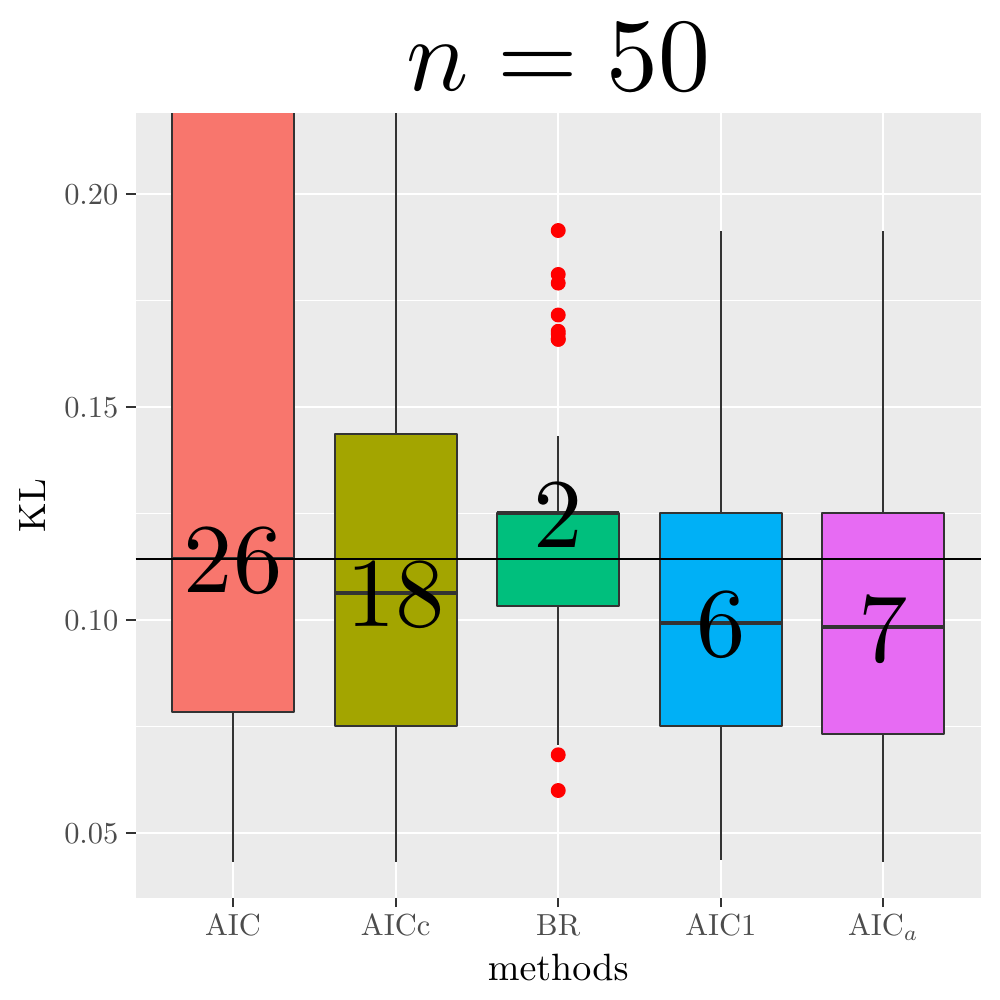}
 \includegraphics[width=0.24\textwidth]{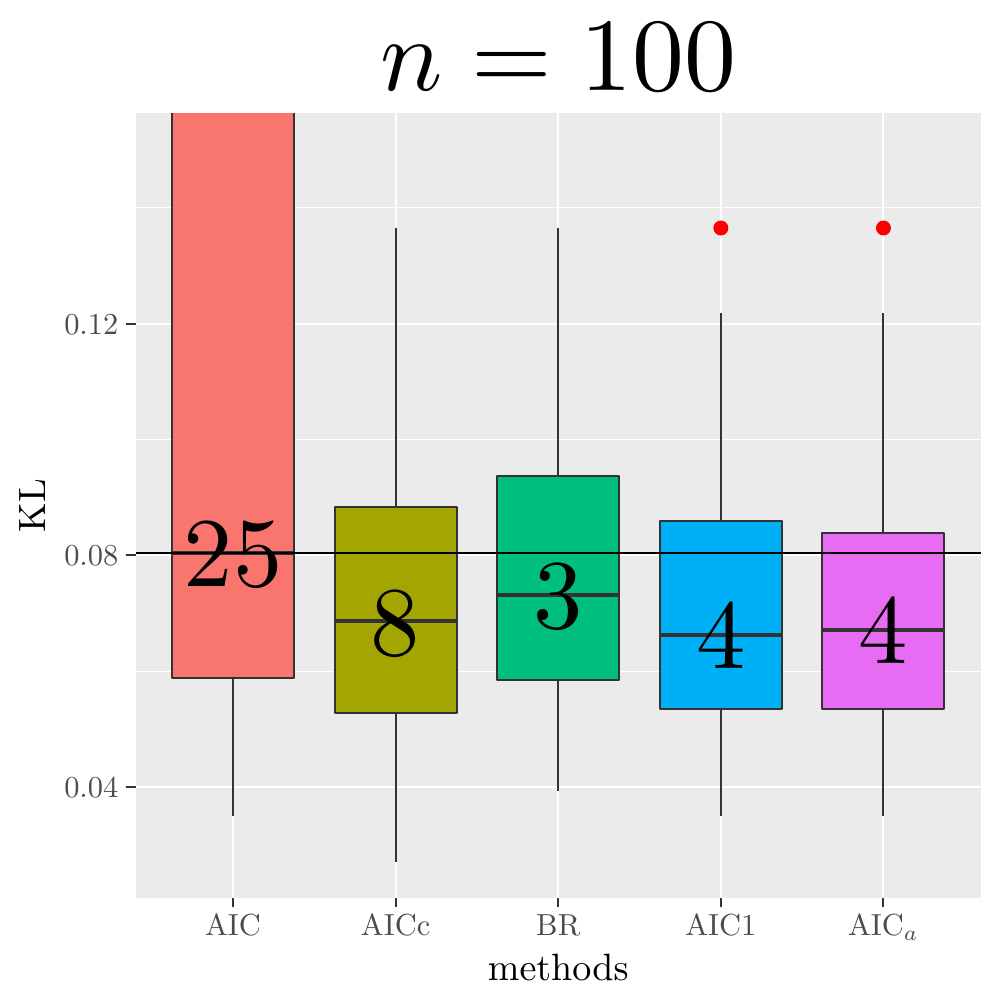}
 \includegraphics[width=0.24\textwidth]{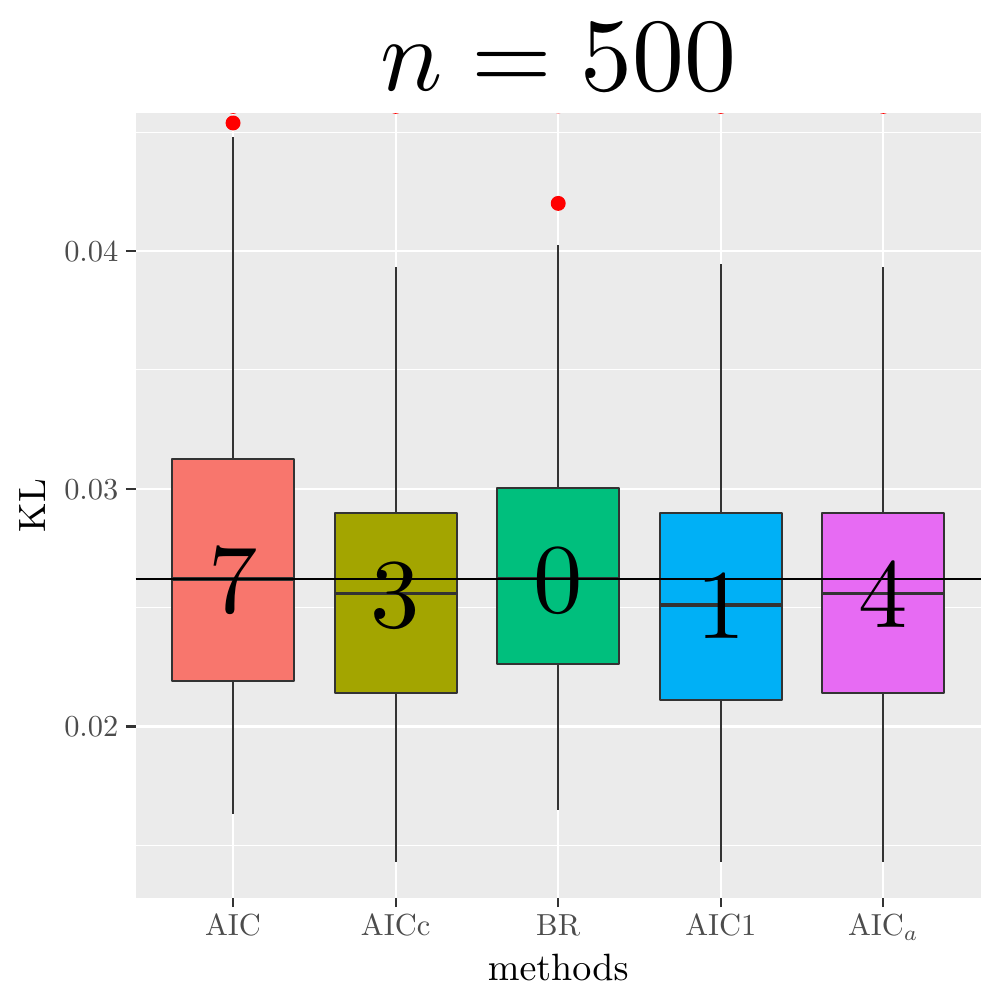}
 \includegraphics[width=0.24\textwidth]{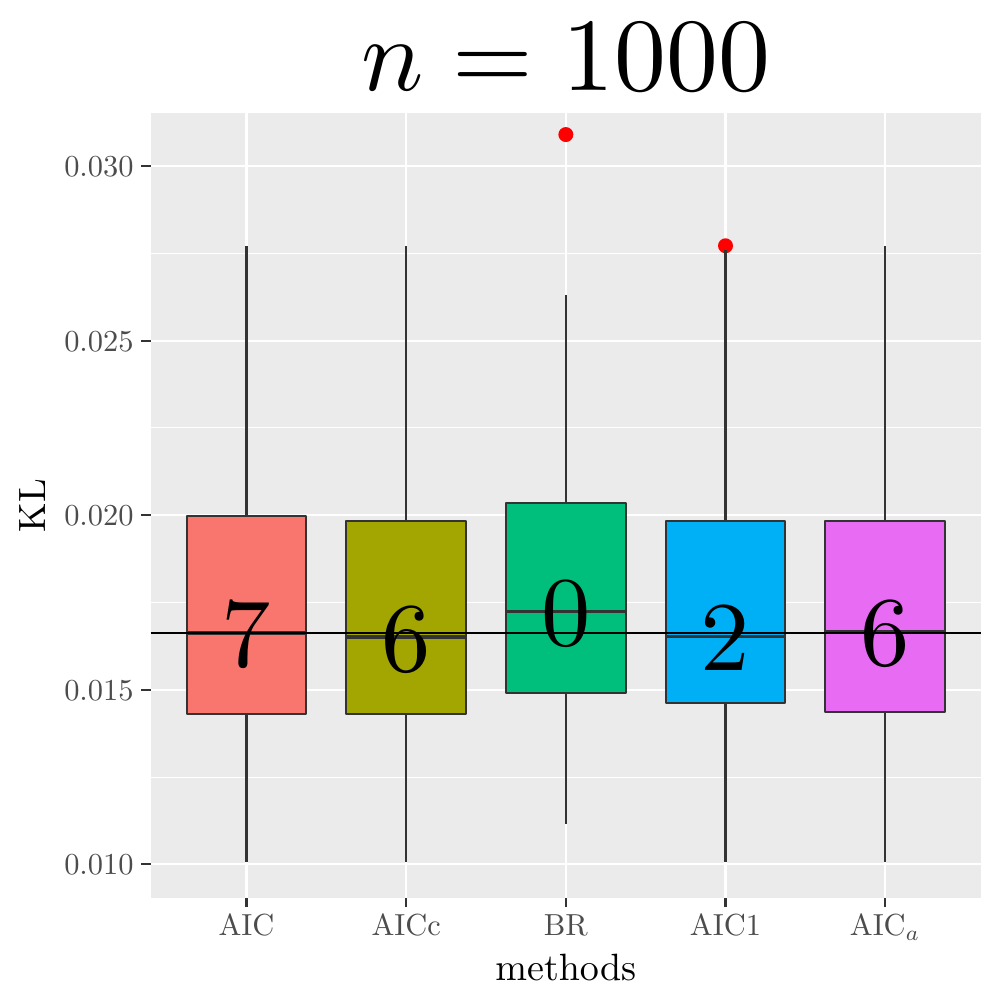}
 }
\subfigure[Infinite peak]{ 
 \includegraphics[width=0.24\textwidth]{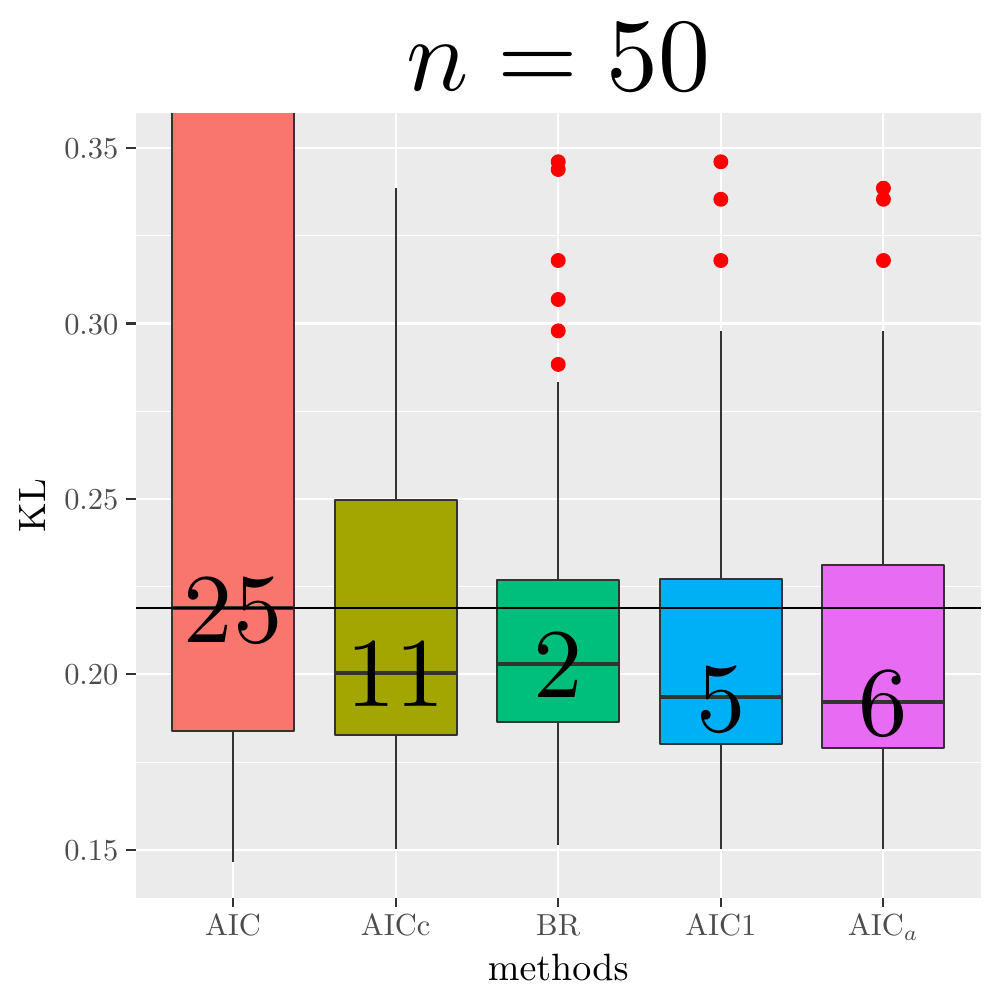}
 \includegraphics[width=0.24\textwidth]{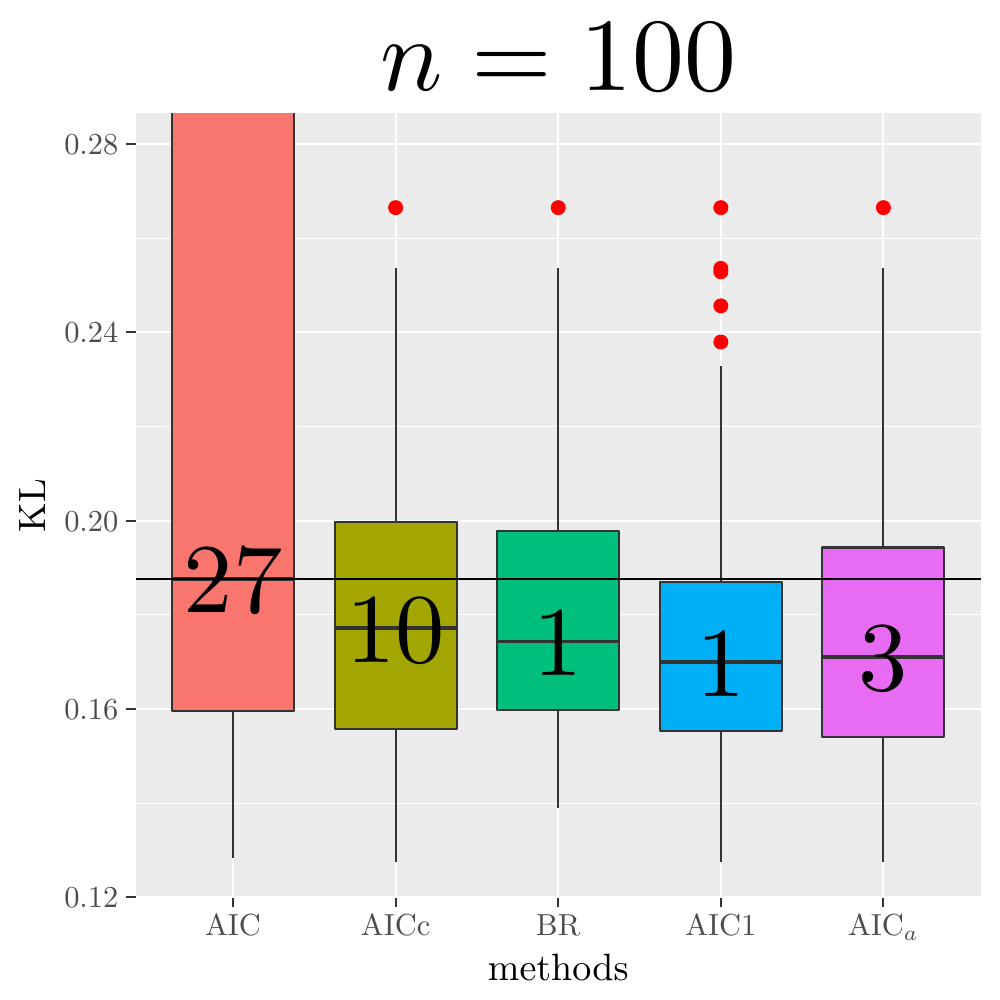}
 \includegraphics[width=0.24\textwidth]{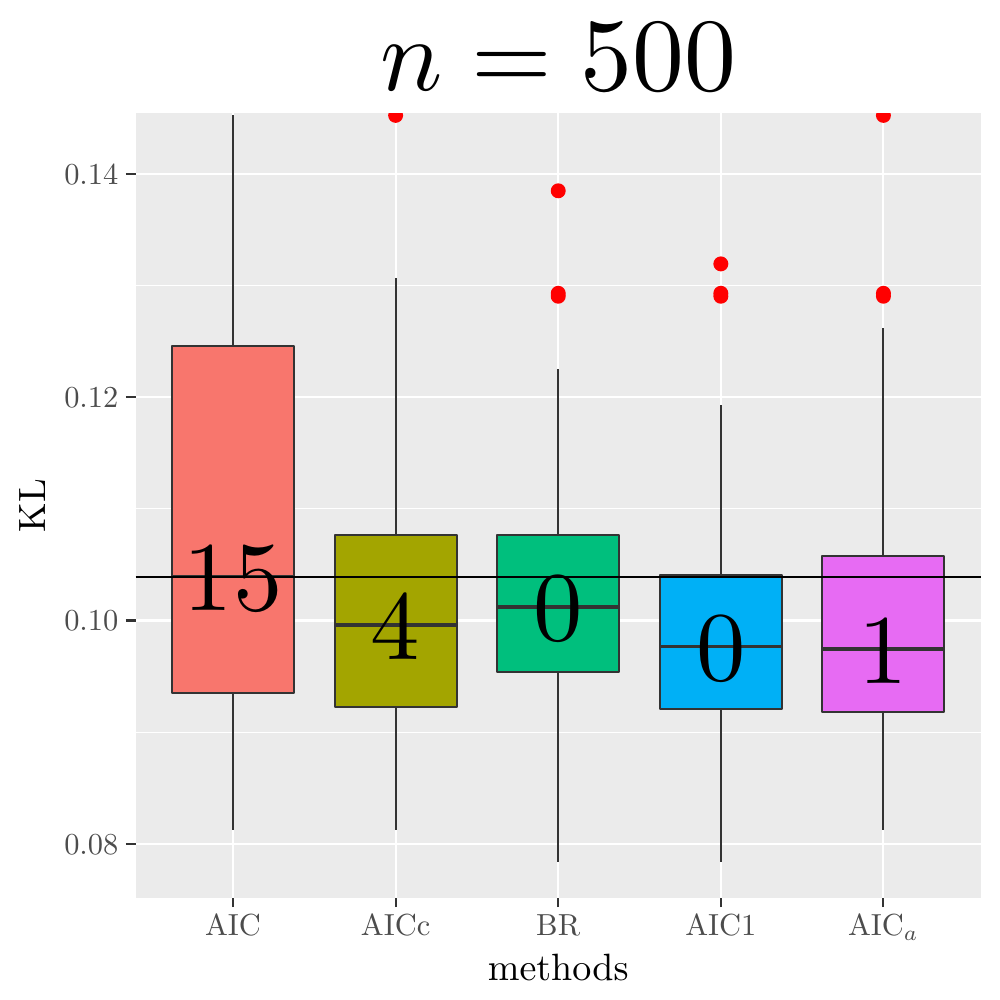}
 \includegraphics[width=0.24\textwidth]{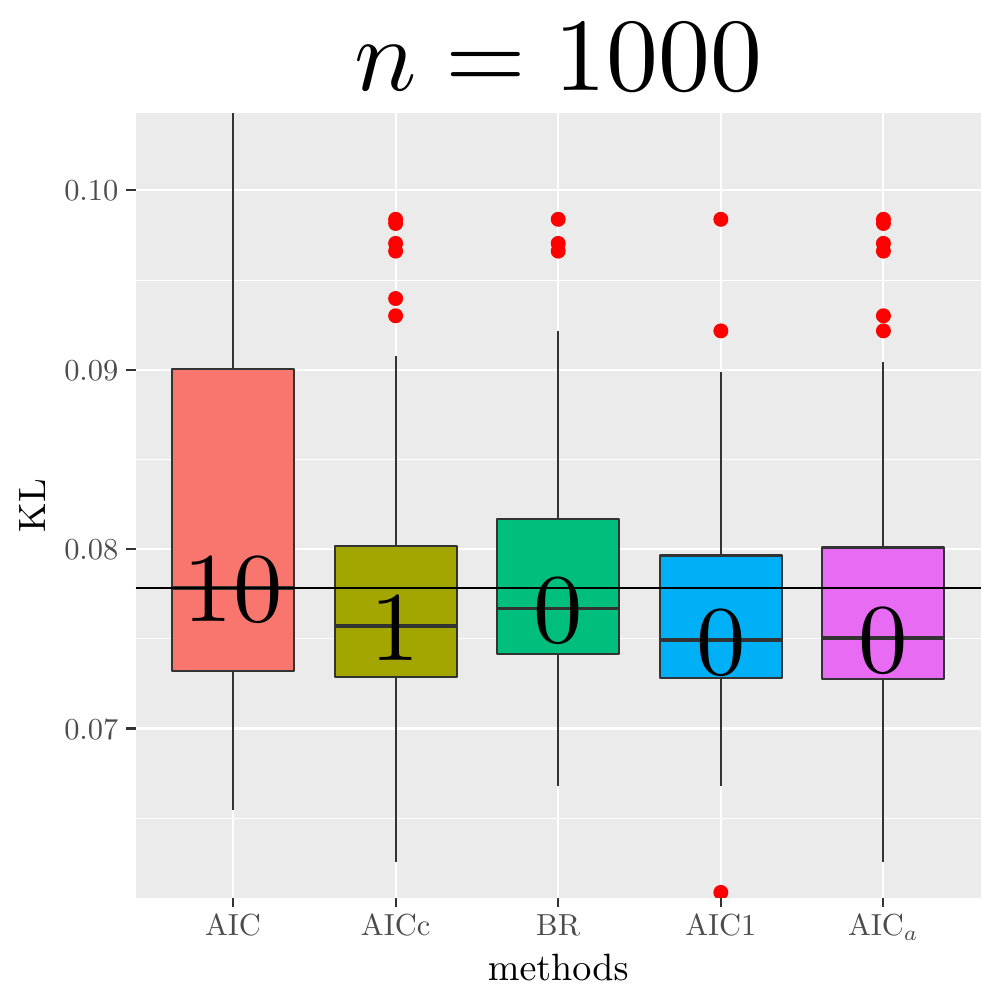}
} 
\includegraphics[width=0.45\textwidth]{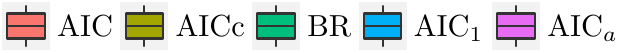}
 
\caption{KL divergence results. Box plots of the KL divergence to the true distribution for the estimated distribution. The solid black line corresponds to the AIC KL divergence median. The term inside the box is the number of time the KL divergence equals $\infty$ out of 100.}
\label{fig:result}
\end{figure}

\subsection*{Acknowledgements}

The first author warmly thanks Matthieu Lerasle for instructive discussions
on the topic of estimation by tests and Alain C{\'e}lisse for a nice discussion at a early stage of this work. He is also grateful to Pascal Massart
for having pushed him towards obtaining better oracle inequalities than in a
previous version of this work. Finally, we owe thanks to Sylvain Arlot and Amandine Dubois for a careful reading that helped to correct some mistakes and improve the presentation of the paper.

\section{Proofs}
\label{section_proofs_MLE}

\subsection{Proofs Related to Section \protect\ref{section_chi_square}\label%
{ssection_chisquare_proof}}

\begin{proof}[Proof of Theorem \protect\ref{prop_upper_dev_chi_2}]
We fix $x,\theta >0$ and we set $z>0$ to be chosen later. Let us set for any 
$I\in m$, $\varphi _{I}=\left( P\left( I\right) \right) ^{-1/2}\mathbbm{1}
_{I}$. The family of functions $\left( \varphi _{I}\right) _{I\in m}$ forms
an orthonormal basis of $\left( m,\left\Vert \cdot \right\Vert _{2}\right) $%
. By Cauchy-Schwarz inequality, we have%
\begin{equation*}
\chi _{n}\left( m\right) =\sup_{\left( a_{I}\right) _{I\in m}\in
B_{2}}\left\vert \left( P_{n}-P\right) \left( \sum_{I\in m}a_{I}\varphi
_{I}\right) \right\vert \text{ ,}
\end{equation*}%
where $B_{2}:=\left\{ \left( a_{I}\right) _{I\in m}\text{ };\text{ }%
\sum_{I\in m}a_{I}^{2}\leq 1\right\} $. Furthermore, the case of equality in
Cauchy-Schwarz inequality gives us%
\begin{equation*}
\chi _{n}\left( m\right) =\left( P_{n}-P\right) \left( \sum_{I\in
m}a_{I}^{\infty }\varphi _{I}\right) \text{ with }a_{I}^{\infty }=\frac{%
\left( P_{n}-P\right) \left( \varphi _{I}\right) }{\chi _{n}\left( m\right) }%
\text{ .}
\end{equation*}%
Hence, by setting $\mathcal{A}\left( s\right) :=B_{2}\bigcap \left\{ \left(
a_{I}\right) _{I\in m}\text{ };\text{ }\sup_{I\in m}\left\vert a_{I}\left(
P\left( I\right) \right) ^{-1/2}\right\vert \leq s\right\} $ for any $s\geq
0 $, we get 
\begin{equation}
\chi _{n}\left( m\right) \boldsymbol{1}_{\Omega _{m}\left( \theta \right) }%
\boldsymbol{1}_{\left\{ \chi _{n}\left( m\right) \geq z\right\} }\leq
\sup_{\left( a_{I}\right) _{I\in m}\in \mathcal{A}\left( \theta /z\right)
}\left\vert \left( P_{n}-P\right) \left( \sum_{I\in m}a_{I}\varphi
_{I}\right) \right\vert \text{ .}  \label{upper_sup}
\end{equation}%
By applying Bousquet's inequality (\cite{Bousquet:02}) to the supremum in
the right-hand side of (\ref{upper_sup}) then gives for any $\delta >0$,%
\begin{align}
&\mathbb{P}\left( \chi _{n}\left( m\right) \boldsymbol{1}_{\Omega _{m}\left(
\theta \right) }\boldsymbol{1}_{\left\{ \chi _{n}\left( m\right) \geq
z\right\} }\geq \left( 1+\delta \right) E_{m}+\sqrt{\frac{2\sigma _{m}^{2}x}{%
n}}+\left( \frac{1}{\delta }+\frac{1}{3}\right) \frac{b_{m}x}{n}\right)
\label{bousquet_z_1} \\
\leq &\mathbb{P}\left( \sup_{\left( a_{I}\right) _{I\in m}\in \mathcal{A}%
\left( \theta /z\right) }\left\vert \left( P_{n}-P\right) \left( \sum_{I\in
m}a_{I}\varphi _{I}\right) \right\vert \geq \left( 1+\delta \right) E_{m}+%
\sqrt{\frac{2\sigma _{m}^{2}x}{n}}+\left( \frac{1}{\delta }+\frac{1}{3}%
\right) \frac{b_{m}x}{n}\right)  \notag \\
\leq &\exp \left( -x\right) \text{ ,}  \notag
\end{align}%
with%
\begin{equation*}
E_{m}\leq \mathbb{E}\left[ \chi _{n}\left( m\right) \right] \leq \mathbb{E}%
^{1/2}\left[ \chi _{n}^{2}\left( m\right) \right] =\sqrt{\frac{D_{m}}{n}}%
\text{ \ \ };\text{ \ \ }\sigma _{m}^{2}\leq \sup_{\left( a_{I}\right)
_{I\in m}\in B_{2}}\mathbb{V}\left( \sum_{I\in m}a_{I}\varphi _{I}\left( \xi
_{1}\right) \right) \leq 1
\end{equation*}%
and%
\begin{equation*}
b_{m}=\sup_{\left( a_{I}\right) _{I\in m}\in \mathcal{A}\left( \theta
/z\right) }\left\Vert \sum_{I\in m}a_{I}\varphi _{I}\right\Vert _{\infty
}\leq \frac{\theta }{z}\text{ .}
\end{equation*}%
Injecting the latter bounds in (\ref{bousquet_z_1}) and taking $z=\sqrt{%
\left( D_{m}\right) /n}+\sqrt{2x/n}$, we obtain that with probability at
least $1-\exp \left( -x\right) $,%
\begin{equation}
\chi _{n}\left( m\right) \boldsymbol{1}_{\Omega _{m}\left( \theta \right)
}<\left( 1+\delta \right) \sqrt{\frac{D_{m}}{n}}+\sqrt{\frac{2x}{n}}+\left( 
\frac{1}{\delta }+\frac{1}{3}\right) \frac{\theta x}{\left( \sqrt{D_{m}}+%
\sqrt{2x}\right) \sqrt{n}}\text{ .}  \label{upper_chi}
\end{equation}%
By taking $\delta =\sqrt{\theta x}\left( D_{m}+\sqrt{2xD_{m}}\right) ^{-1/2}$%
, the right-hand side of (\ref{upper_chi}) becomes%
\begin{align*}
&\sqrt{\frac{D_{m}}{n}}+\sqrt{\frac{2x}{n}}+2\sqrt{\frac{\theta x}{n}}\frac{%
\sqrt{D_{m}}}{\sqrt{D_{m}+\sqrt{2xD_{m}}}}+\frac{\theta x}{3\left( \sqrt{%
D_{m}}+\sqrt{2x}\right) \sqrt{n}} \\
\leq &\sqrt{\frac{D_{m}}{n}}+\sqrt{\frac{2x}{n}}+2\sqrt{\frac{\theta }{n}}%
\left( \sqrt{x}\wedge \left( \frac{xD_{m}}{2}\right) ^{1/4}\right) +\frac{%
\theta }{3}\sqrt{\frac{x}{n}}\left( \sqrt{\frac{x}{D_{m}}}\wedge \frac{1}{%
\sqrt{2}}\right) \text{ ,}
\end{align*}%
which gives (\ref{upper_chi_gene}). Inequality (\ref{upper_chi_part}) is a
direct consequence of (\ref{upper_chi_gene}).
\end{proof}

\begin{proof}[Proof of Theorem \protect\ref{prop_dev_gauche_chi}]
Let us set for any $I\in m$, $\varphi _{I}=\left( P\left( I\right) \right)
^{-1/2}\mathbbm{1}_{I}$. The family of functions $\left( \varphi _{I}\right)
_{I\in m}$ forms an orthonormal basis of $\left( m,\left\Vert \cdot
\right\Vert _{2}\right) $. By Cauchy-Schwarz inequality, we have%
\begin{equation*}
\chi _{n}\left( m\right) =\sup_{\left( a_{I}\right) _{I\in m}\in
B_{2}}\left\vert \left( P_{n}-P\right) \left( \sum_{I\in m}a_{I}\varphi
_{I}\right) \right\vert \text{ ,}
\end{equation*}%
where $B_{2}=\left\{ \left( a_{I}\right) _{I\in m};\sum_{I\in
m}a_{I}^{2}\leq 1\right\} $. As 
\begin{equation*}
\sigma _{m}^{2}:=\sup_{\left( a_{I}\right) _{I\in m}\in B_{2}}\mathbb{V}%
\left( \sum_{I\in m}a_{I}\varphi _{I}\left( \xi _{1}\right) \right) \leq 1
\end{equation*}%
and 
\begin{equation*}
b_{m}=\sup_{\left( a_{I}\right) _{I\in m}\in B_{2}}\left\Vert \sum_{I\in
m}a_{I}\varphi _{I}\right\Vert _{\infty }\leq \sup_{I\in m}\left\Vert
\varphi _{I}\right\Vert _{\infty }=\sup_{I\in m}\frac{1}{\sqrt{P\left(
I\right) }}\text{ ,}
\end{equation*}%
we get by Klein-Rio's inequality (see \cite{Klein_Rio:05}), for every $%
x,\delta >0$, 
\begin{equation}
\mathbb{P}\left( \chi _{n}\left( m\right) \leq \left( 1-\delta \right) 
\mathbb{E}\left[ \chi _{n}\left( m\right) \right] -\sqrt{\frac{2x}{n}}%
-\left( \frac{1}{\delta }+1\right) \frac{b_{m}x}{n}\right) \leq \exp \left(
-x\right) \text{ }.  \label{dev_left_chi}
\end{equation}%
Note that we have $b_{m}\leq \sqrt{D_{m}A_{\Lambda }^{-1}}$. Now, we bound $%
\mathbb{E}\left[ \chi _{n}\left( m\right) \right] $ by below. By Theorem
11.10 in \cite{BouLugMas:13} applied to $\chi _{n}\left( m\right) $, we get,
for any $\zeta >0$,%
\begin{align}
\mathbb{E}\left[ \chi _{n}^{2}\left( m\right) \right] -\mathbb{E}^{2}\left[
\chi _{n}\left( m\right) \right] \leq & \frac{\sigma _{m}^{2}}{n}+4\frac{%
b_{m}}{n}\mathbb{E}\left[ \chi _{n}\left( m\right) \right]  \notag \\
\leq & \frac{1}{n}+4\frac{\sqrt{D_{m}A_{\Lambda }^{-1}}}{n}\mathbb{E}\left[
\chi _{n}\left( m\right) \right]  \notag \\
\leq & \frac{1}{n}+4\zeta ^{-1}\frac{D_{m}A_{\Lambda }^{-1}}{n^{2}}+\zeta 
\mathbb{E}^{2}\left[ \chi _{n}\left( m\right) \right] \text{ .}
\label{variance_bound}
\end{align}%
The latter inequality results from $2ab\leq \zeta a^{2}+\zeta ^{-1}b^{2}$
applied with $a=\mathbb{E}\left[ \chi _{n}\left( m\right) \right] $ and $b=2%
\sqrt{D_{m}A_{\Lambda }^{-1}}/n$. As $\mathbb{E}\left[ \chi _{n}^{2}\left(
m\right) \right] =D_{m}/n$, (\ref{variance_bound}) applied with $\zeta
=n^{-1/2}$ gives%
\begin{align}
\mathbb{E}\left[ \chi _{n}\left( m\right) \right] \geq & \sqrt{\frac{1}{%
1+n^{-1/2}}\left( \frac{D_{m}}{n}-4A_{\Lambda }^{-1}\frac{D_{m}}{n^{3/2}}-%
\frac{1}{n}\right) _{+}}  \notag \\
\geq & \sqrt{\frac{D_{m}}{n}}\left( 1-L_{A_{\Lambda }}D_{m}^{-1/2}\vee
n^{-1/4}\right) \text{ .}  \label{lower_mean}
\end{align}%
Hence, by using (\ref{lower_mean}) and taking $x=\alpha \ln (n+1)$ and $%
\delta =n^{-1/4}\sqrt{\ln (n+1)}$ in (\ref{dev_left_chi}), we obtain with
probability at least $1-(n+1)^{-\alpha }$,%
\begin{align*}
& \chi _{n}\left( m\right) \\
\geq & \left( 1-\frac{\sqrt{\ln (n+1)}}{n^{1/4}}\right) \sqrt{\frac{D_{m}}{n}%
}\left( 1-L_{A_{\Lambda }}\frac{1}{\sqrt{D_{m}}}\vee \frac{1}{n^{1/4}}\right)
\\
{}& -\sqrt{\frac{2\alpha \ln (n+1)}{n}}-\left( \frac{n^{1/4}}{\sqrt{\ln (n+1)%
}}+1\right) \frac{\alpha \sqrt{D_{m}}\ln (n+1)}{n} \\
\geq & \sqrt{\frac{D_{m}}{n}}\left( 1-L_{A_{\Lambda }}\frac{\sqrt{\ln (n+1)}%
}{n^{1/4}}\vee \frac{1}{\sqrt{D_{m}}}-\sqrt{\frac{2\alpha \ln (n+1)}{D_{m}}}%
-L_{\alpha }\frac{\sqrt{\ln (n+1)}}{n^{1/4}}\right) \\
\geq & \sqrt{\frac{D_{m}}{n}}\left( 1-L_{A_{\Lambda },\alpha }\frac{\sqrt{%
\ln (n+1)}}{n^{1/4}}\vee \sqrt{\frac{\ln (n+1)}{D_{m}}}\right) \text{ ,}
\end{align*}%
which concludes the proof.
\end{proof}

\subsection{Proofs related to Section \protect\ref%
{section_concentration_inequalities}}


In this section, our aim is to control for any histogram model $m$, the
deviations of the empirical bias from its mean, which is the true bias.
Concentration inequalities for the centered empirical bias are provided in
Section~\ref{section_concentration_inequalities}. In order to prove oracle
inequalities for the KL divergence, we relate the magnitude of these
deviations to the true bias. This is done in Section~\ref%
{section_margin_like_relations}. Since we believe that the results in this section are of interest in themselves, we have stated them, when possible, for a general density in $\mathcal{S}$ rather for the projections $f_{m}$.

The results in this section are based on the Cram\`{e}r-Chernoff method (see 
\cite{BoucheronLugosiMassart:2013} for instance). Let us recall that if we
set $S:=\sum_{i=1}^{n}X_{i}-\mathbb{E}\left[ X_{i}\right] $, where $\left(
X_{i}\right) _{i=1}^{n}$ are $n$ i.i.d. real random variables, and for any $%
\lambda \geq 0$ and $y\in \mathbb{R}_{+}$, 
\begin{equation*}
\varphi _{S}\left( \lambda \right) :=\ln \left( \mathbb{E}\left[ \exp \left(
\lambda S\right) \right] \right) =n\left( \ln \left( \mathbb{E}\left[ \exp
\left( \lambda X_{1}\right) \right] \right) -\lambda \mathbb{E}\left[ X_{1}%
\right] \right)
\end{equation*}%
and 
\begin{equation*}
\varphi _{S}^{\ast }\left( y\right) :=\sup_{\lambda \in \mathbb{R}%
_{+}}\left\{ \mathbb{\lambda }y\mathbb{-}\varphi _{S}\left( \lambda \right)
\right\} \text{ ,}
\end{equation*}%
then 
\begin{equation}
\mathbb{P}\left( S\geq y\right) \leq \exp \left( -\varphi _{S}^{\ast }\left(
y\right) \right) \text{ .}  \label{markov}
\end{equation}

\begin{proof}[Proof of Proposition \protect\ref{delta_bar_dev_droite_MLE}]
We first prove concentration inequality (\ref{dev_droite_gene_2}). We set $%
X_{i}:=\ln \left( \left. f\right/ f_{\ast }\right) \left( \xi _{i}\right) $
and use Inequality (\ref{markov}). For $\lambda \in \left[ 0,1\right] $, as $%
\mathbb{E}\left[ X_{1}\right] =-\mathcal{K}\left( f_{\ast },f\right) $, we
have%
\begin{align*}
\left. \varphi _{S}\left( \lambda \right) \right/ n=& \ln \left( P\left[
\left( \left. f\right/ f_{\ast }\right) ^{\lambda }\right] \right) +\lambda 
\mathcal{K}\left( f_{\ast },f\right)   \notag \\
\leq & \lambda \ln \left( P\left[ \left. f\right/ f_{\ast }\right] \right)
+\lambda \mathcal{K}\left( f_{\ast },f\right) =\lambda \mathcal{K}\left(
f_{\ast },f\right) \text{ ,}  
\end{align*}%
where the inequality derives from the concavity of the function $x\mapsto
x^{\lambda }$. By setting $\mathcal{K=K}\left( f_{\ast },f\right) $, we thus
get%
\begin{equation*}
\varphi _{S}^{\ast }\left( y\right) \geq \sup_{\lambda \in \left[ 0,1\right]
}\left\{ \mathbb{\lambda }y\mathbb{-}\varphi _{S}\left( \lambda \right)
\right\} \geq \left( y-n\mathcal{K}\right) _{+}\text{ ,}
\end{equation*}%
which implies,%
\begin{equation}
\mathbb{P}\left( (P_{n}-P)\left( \ln \left( f/f_{\ast }\right) \right) \geq
x\right) \leq \exp \left( -n(x-\mathcal{K})_{+}\right) \text{ }.
\label{ineq_right_zero}
\end{equation}%
Inequality (\ref{dev_droite_gene_2}) is a direct consequence of (\ref%
{ineq_right_zero}). Moreover, we notice that for any $u\in \mathbb{R}$, $%
\exp \left( u\right) \leq 1+u+\frac{u^{2}}{2}\exp \left( u_{+}\right) $ and $%
\ln \left( 1+u\right) \leq u$, where $u_{+}=u\vee 0$. By consequence, for $%
\lambda \in \left[ 0,1\right] $, it holds%
\begin{gather*}
\left. \varphi _{S}\left( \lambda \right) \right/ n=\ln \left( \mathbb{E}%
\left[ \exp \left( \lambda X_{1}\right) \right] \right) -\lambda \mathbb{E}%
\left[ X_{1}\right]  \\
\leq \ln \left( 1+\lambda \mathbb{E}\left[ X_{1}\right] +\frac{\lambda ^{2}}{%
2}\mathbb{E}\left[ X_{1}^{2}\exp \left( \lambda \left( X_{1}\right)
_{+}\right) \right] \right) -\lambda \mathbb{E}\left[ X_{1}\right]  \\
\leq \frac{\lambda ^{2}}{2}\mathbb{E}\left[ X_{1}^{2}\exp \left( \lambda
\left( X_{1}\right) _{+}\right) \right] \leq \frac{\lambda ^{2}}{2}\mathbb{E}%
\left[ X_{1}^{2}\exp \left( \left( X_{1}\right) _{+}\right) \right] \leq 
\frac{\lambda ^{2}}{2}v\text{ .}
\end{gather*}%
Now, we get, for any $y\geq 0$,%
\begin{equation}
\varphi _{S}^{\ast }\left( y\right) \geq \sup_{\lambda \in \left[ 0,1\right]
}\left\{ \mathbb{\lambda }y\mathbb{-}\varphi _{S}\left( \lambda \right)
\right\} \geq \sup_{\lambda \in \left[ 0,1\right] }\left\{ \mathbb{\lambda }y%
\mathbb{-}n\frac{\lambda ^{2}v}{2}\right\} =\left( \frac{y^{2}}{2nv}\mathbf{1%
}_{y\leq nv}+\left( y-\frac{nv}{2}\right) \mathbbm{1}_{y>nv}\right) \text{ .}
\label{lower_phi_star_droite}
\end{equation}%
So, by using (\ref{markov}) with (\ref{lower_phi_star_droite}) taken with $%
x=y/n$, it holds%
\begin{equation}
\mathbb{P}\left( \left( P_{n}-P\right) \left( \ln \left( \left. f\right/
f_{\ast }\right) \right) \geq x\right) \leq \exp \left( -n\left( \frac{x^{2}%
}{2v}\mathbbm{1}_{x\leq v}+\left( x-\frac{v}{2}\right) \mathbbm{1}%
_{x>v}\right) \right) \text{ .}  \label{dev_droite_prelim}
\end{equation}%
To obtain (\ref{bernstein_dev_droite_2}), we notice that Inequality (\ref%
{dev_droite_prelim}) implies by simple calculations, for any $z\geq 0$,%
\begin{equation*}
\mathbb{P}\left( \left( P_{n}-P\right) \left( \ln \left( \left. f\right/
f_{\ast }\right) \right) \geq \sqrt{\frac{2vz}{n}}\mathbbm{1}_{z\leq
nv/2}+\left( \frac{z}{n}+\frac{v}{2}\right) \mathbbm{1}_{z>nv/2}\right) \leq
\exp \left( -z\right) \text{ .}  
\end{equation*}%
To conclude the proof, it suffices to remark that $\sqrt{2vz/n}\mathbbm{1}%
_{z\leq nv/2}+\left( z/n+v/2\right) \mathbbm{1}_{z>nv/2}\leq \sqrt{2vz/n}%
+2z/n$.
\end{proof}

\begin{proof}[Proof of Proposition~\ref{delta_bar_dev_gauche_MLE}]
Let us first prove the inequality of concentration (\ref{delta_dev_gauche_2}). We
set $X_{i}:=\ln \left( \left. f_{\ast }\right/ f\right) \left( \xi
_{i}\right) $ and use (\ref{markov}). For $\lambda \in \left[ 0,r\right] $,
we have by H\"{o}lder's inequality, $P\left[ \left( \left. f_{\ast }\right/
f\right) ^{\lambda }\right] \leq P\left[ \left( \left. f_{\ast }\right/
f\right) ^{r}\right] ^{\lambda /r}$. Hence,%
\begin{align*}
\left. \varphi _{S}\left( \lambda \right) \right/ n=& \ln \left( P\left[
\left( \left. f_{\ast }\right/ f\right) ^{\lambda }\right] \right) -\lambda 
\mathcal{K}\left( f_{\ast },f\right)  \\
\leq & \lambda \left( \frac{1}{r}\ln \left( P\left[ \left( \left. f_{\ast
}\right/ f\right) ^{r}\right] \right) -\mathcal{K}\left( f_{\ast },f\right)
\right) \text{ .}  
\end{align*}%
Let us notice that by concavity of $\ln $, we have $\frac{1}{r}\ln \left( P%
\left[ \left( \left. f_{\ast }\right/ f\right) ^{r}\right] \right) -\mathcal{%
K}\left( f_{\ast },f\right) \geq 0$. Now we get, for any $y\geq 0$,%
\begin{align*}
\varphi _{S}^{\ast }\left( y\right) \geq & \sup_{\lambda \in \left[ 0,r%
\right] }\left\{ \mathbb{\lambda }y\mathbb{-}\varphi _{S}\left( \lambda
\right) \right\}  \\
\geq & \sup_{\lambda \in \left[ 0,r\right] }\left\{ \mathbb{\lambda }\left( y%
\mathbb{-}n\left( \frac{1}{r}\ln \left( P\left[ \left( \left. f_{\ast
}\right/ f\right) ^{r}\right] \right) -\mathcal{K}\left( f_{\ast },f\right)
\right) \right) \right\}  \\
=& \left( ry\mathbb{-}n\left( \ln \left( P\left[ \left( \left. f_{\ast
}\right/ f\right) ^{r}\right] \right) -r\mathcal{K}\left( f_{\ast },f\right)
\right) \right) _{+}\text{ .}
\end{align*}%
Using (\ref{markov}), we obtain, for any $x\geq 0$,%
\begin{equation}
\mathbb{P}\left( (P_{n}-P)(\ln \left( f/f_{\ast }\right) )\leq -x\right)
\leq \exp \left( -n(rx-\ln \left[ \left( \left. f_{\ast }\right/ f\right)
^{r}\right] +r\mathcal{K}\left( f_{\ast },f\right) )_{+}\right) \text{ .}
\label{delta_dev_gauche_0}
\end{equation}%
Inequality (\ref{delta_dev_gauche_2}) is a straightforward consequence of (%
\ref{delta_dev_gauche_0}). As in the proof of Theorem \ref%
{delta_bar_dev_droite_MLE}, we notice that for any $u\in \mathbb{R}$, $\exp
\left( u\right) \leq 1+u+\frac{u^{2}}{2}\exp \left( u_{+}\right) $ and $\ln
\left( 1+u\right) \leq u$, where $u_{+}=u\vee 0$. By consequence, for $%
\lambda \in \left[ 0,r\right] $, it holds%
\begin{gather*}
\left. \varphi _{S}\left( \lambda \right) \right/ n=\lambda \mathbb{E}\left[
X_{1}\right] +\ln \left( \mathbb{E}\left[ \exp \left( -\lambda X_{1}\right) %
\right] \right)  \\
\leq \ln \left( 1-\lambda \mathbb{E}\left[ X_{1}\right] +\frac{\lambda ^{2}}{%
2}\mathbb{E}\left[ X_{1}^{2}\exp \left( \lambda \left( -X_{1}\right)
_{+}\right) \right] \right) +\lambda \mathbb{E}\left[ X_{1}\right]  \\
\leq \frac{\lambda ^{2}}{2}\mathbb{E}\left[ X_{1}^{2}\exp \left( \lambda
\left( -X_{1}\right) _{+}\right) \right] \leq \frac{\lambda ^{2}}{2}\mathbb{E%
}\left[ X_{1}^{2}\exp \left( r\left( -X_{1}\right) _{+}\right) \right] \leq 
\frac{\lambda ^{2}}{2}w_{r}\text{ .}
\end{gather*}%
Now we get, for any $y\geq 0$,%
\begin{gather*}
\varphi _{S}^{\ast }\left( y\right) \geq \sup_{\lambda \in \left[ 0,r\right]
}\left\{ \mathbb{\lambda }y\mathbb{-}\varphi _{S}\left( \lambda \right)
\right\}  \\
\geq \sup_{\lambda \in \left[ 0,r\right] }\left\{ \mathbb{\lambda }y\mathbb{-%
}\frac{n\lambda ^{2}w_{r}}{2}\right\} =\frac{y^{2}}{2nw_{r}}\mathbbm{1}%
_{y\leq rnw_{r}}+r\left( y-\frac{rnw_{r}}{2}\right) \mathbbm{1}_{y>rnw_{r}}%
\text{ ,}
\end{gather*}%
which gives 
\begin{equation}
\mathbb{P}\left( (P_{n}-P)(\ln \left( f/f_{\ast }\right) )\leq -x\right)
\leq \exp \left( -n\left( \frac{x^{2}}{2nw_{r}}\mathbbm{1}_{x\leq
rnw_{r}}+r\left( x-\frac{rnw_{r}}{2}\right) \mathbbm{1}_{x>rnw_{r}}\right)
\right)   \label{delat_dev_gauche_gauss_0}
\end{equation}%
Inequality (\ref{delta_dev_gauche_gauss_2}) is again a consequence of (\ref%
{delat_dev_gauche_gauss_0}), by the same kind of arguments as those involved
in the proof of \ref{bernstein_dev_droite_2}\ in Lemma \ref%
{delta_bar_dev_droite_MLE}.
\end{proof}

\subsection{Proofs related to Section \protect\ref%
{section_margin_like_relations}}

\begin{proof}[Proof of Proposition \protect\ref{lem:margin_like_unbounded}]
Let us take $q>1$ such that $1/p+1/q=1$. It holds%
\begin{align*}
P\left[ \left( \frac{f}{f_{\ast }}\vee 1\right) \left( \ln \left( \frac{f}{%
f_{\ast }}\right) \right) ^{2}\right] =& \int \left( f\vee f_{\ast }\right)
\left( \ln \left( \frac{f}{f_{\ast }}\right) \right) ^{2}d\mu \\
=& \int \frac{f_{\ast }\vee f}{f_{\ast }\wedge f}\left( \left( f_{\ast
}\wedge f\right) \left( \ln \left( \frac{f}{f_{\ast }}\right) \right)
^{2}\right) ^{\frac{1}{p}+\frac{1}{q}}d\mu \\
=& \int \left( \frac{f_{\ast }\vee f}{\left( f_{\ast }\wedge f\right) ^{%
\frac{1}{q}}}\left\vert \ln \left( \frac{f}{f_{\ast }}\right) \right\vert ^{%
\frac{2}{p}}\right) \left( \left( f_{\ast }\wedge f\right) \left( \ln \left( 
\frac{f}{f_{\ast }}\right) \right) ^{2}\right) ^{\frac{1}{q}}d\mu \\
\leq & \left( \int \left( f_{\ast }\wedge f\right) \left( \ln \left( \frac{f%
}{f_{\ast }}\right) \right) ^{2}d\mu \right) ^{\frac{1}{q}}\left( \underset{%
:=I}{\underbrace{\int \frac{\left( f_{\ast }\vee f\right) ^{p}}{\left(
f_{\ast }\wedge f\right) ^{p-1}}\left( \ln \left( \frac{f}{f_{\ast }}\right)
\right) ^{2}d\mu }}\right) ^{\frac{1}{p}}
\end{align*}%
where in the last step we used H\"{o}lder's inequality. Now, by Lemma 7.24
of Massart \cite{Massart:07}, it also holds%
\begin{equation*}
\frac{1}{2}\int \left( f_{\ast }\wedge f\right) \left( \ln \left( \frac{f}{%
f_{\ast }}\right) \right) ^{2}d\mu \leq \mathcal{K}\left( f_{\ast },f\right) 
\text{ .}
\end{equation*}%
In order to prove (\ref{margin_like_soft_droite_gene}), it thus remains to
bound $I$ in terms of $p,c_{+}$ and $c_{-}$ only. First, we decompose $I$
into two parts,%
\begin{equation}
\int \frac{\left( f_{\ast }\vee f\right) ^{p}}{\left( f_{\ast }\wedge
f\right) ^{p-1}}\left( \ln \left( \frac{f}{f_{\ast }}\right) \right)
^{2}d\mu =\int \frac{f_{\ast }^{p}}{f^{p-1}}\left( \ln \left( \frac{f_{\ast }%
}{f}\right) \right) ^{2}\mathbbm{1}_{f_{\ast }\geq f}d\mu +\int \frac{f^{p}}{%
f_{\ast }^{p-1}}\left( \ln \left( \frac{f}{f_{\ast }}\right) \right) ^{2}%
\mathbbm{1}_{f\geq f_{\ast }}d\mu\text{.}  \label{ineq_0_gene}
\end{equation}%
For the first term in the right-hand side of (\ref{ineq_0_gene}), we get%
\begin{gather}
\int \frac{f_{\ast }^{p}}{f^{p-1}}\left( \ln \left( \frac{f_{\ast }}{f}%
\right) \right) ^{2}\mathbbm{1}_{f_{\ast }\geq f}d\mu \leq c_{-}^{1-p}\int
f_{\ast }^{p}\left( \ln \left( \frac{f_{\ast }}{c_{-}}\right) \right)
^{2}d\mu  \notag \\
\leq 4c_{-}^{1-p}\left( \left( \ln c_{-}\right) ^{2}\vee 1\right) \int
f_{\ast }^{p}\left( \left( \ln f_{\ast }\right) ^{2}\vee 1\right) d\mu
<+\infty \text{ ,}  \label{ineq_1_gene}
\end{gather}%
where in the second inequality we used the following fact: $\left(
a+b\right) ^{2}\leq 4\left( a^{2}\vee 1\right) \left( b^{2}\vee 1\right) $,
for any real numbers $a$ and $b$. The finiteness of the upper bound is
guaranteed by assumption (\ref{moment_p_bis}). For the second term in the
right-hand side of (\ref{ineq_0_gene}), it holds by same kind of arguments
that lead to (\ref{ineq_1_gene}),%
\begin{gather}
\int \frac{f^{p}}{f_{\ast }^{p-1}}\left( \ln \left( \frac{f}{f_{\ast }}%
\right) \right) ^{2}\mathbbm{1}_{f\geq f_{\ast }}d\mu \leq c_{+}^{p}\int 
\frac{1}{f_{\ast }^{p-1}}\left( \ln \left( \frac{c_{+}}{f_{\ast }}\right)
\right) ^{2}d\mu  \notag \\
\leq 4c_{+}^{p}\left( \left( \ln c_{+}\right) ^{2}\vee 1\right) \int f_{\ast
}^{1-p}\left( \left( \ln f_{\ast }\right) ^{2}\vee 1\right) d\mu <+\infty 
\text{ ,}  \label{ineq_2_gene}
\end{gather}%
where in the last inequality we used the following fact: $\left( a-b\right)
^{2}\leq 4\left( a^{2}\vee 1\right) \left( b^{2}\vee 1\right) $. Again, the
finiteness of the upper bound is guaranteed by (\ref{moment_p_bis}).
Inequality (\ref{margin_like_soft_droite_gene}) then follows from combining (%
\ref{ineq_0_gene}), (\ref{ineq_1_gene}) and (\ref{ineq_2_gene}).

Inequality (\ref{margin_like_soft_gauche_gene}) follows from the same kind
of computations. Indeed, we have by the use of H\"{o}lder's inequality,%
\begin{gather*}
P\left[ \left( \frac{f_{\ast }}{f}\vee 1\right) ^{r}\left( \ln \left( \frac{f%
}{f_{\ast }}\right) \right) ^{2}\right] =\int \left( \frac{f_{\ast }^{r+1}}{%
f^{r}}\vee f_{\ast }\right) \left( \ln \left( \frac{f}{f_{\ast }}\right)
\right) ^{2}d\mu \\
=\int \left( \frac{f_{\ast }^{r+1}\vee f_{\ast }f^{r}}{f^{r}\left( f_{\ast
}\wedge f\right) ^{1-\frac{r+1}{p}}}\left\vert \ln \left( \frac{f}{f_{\ast }}%
\right) \right\vert ^{\frac{2\left( r+1\right) }{p}}\right) \left( \left(
f_{\ast }\wedge f\right) \left( \ln \left( \frac{f}{f_{\ast }}\right)
\right) ^{2}\right) ^{1-\frac{r+1}{p}}d\mu \\
\leq \left( \int \left( f_{\ast }\wedge f\right) \left( \ln \left( \frac{f}{%
f_{\ast }}\right) \right) ^{2}d\mu \right) ^{1-\frac{r+1}{p}}\left( \underset%
{:=I_{r}}{\underbrace{\int \frac{f_{\ast }^{p}\vee f_{\ast }^{\frac{p}{r+1}%
}f^{\frac{rp}{r+1}}}{f^{\frac{rp}{r+1}}f_{\ast }^{\frac{p}{r+1}-1}\wedge
f^{p-1}}\left( \ln \left( \frac{f}{f_{\ast }}\right) \right) ^{2}d\mu }}%
\right) ^{\frac{r+1}{p}}\text{ .}
\end{gather*}%
In order to prove (\ref{margin_like_soft_gauche_gene}), it thus remains to
bound $I_{r}$ in terms of $p,c_{+}$ and $c_{-}$ only. Again, we split $I_{r}$
into two parts,%
\begin{equation}
\int \frac{f_{\ast }^{p}\vee f_{\ast }^{\frac{p}{r+1}}f^{\frac{rp}{r+1}}}{f^{%
\frac{rp}{r+1}}f_{\ast }^{\frac{p}{r+1}-1}\wedge f^{p-1}}\left( \ln \left( 
\frac{f}{f_{\ast }}\right) \right) ^{2}d\mu =\int \frac{f_{\ast }^{p}}{%
f^{p-1}}\left( \ln \left( \frac{f_{\ast }}{f}\right) \right) ^{2}\mathbbm{1}%
_{f_{\ast }\geq f}d\mu +\int f_{\ast }\left( \ln \left( \frac{f}{f_{\ast }}%
\right) \right) ^{2}\mathbbm{1}_{f\geq f_{\ast }}d\mu \text{ .}
\label{ineq_0_gauche_gene}
\end{equation}%
The first term in the right-hand side of (\ref{ineq_0_gauche_gene}) is given
by (\ref{ineq_1_gene}) above. For the second term in the right-hand side of (%
\ref{ineq_0_gauche_gene}), it holds%
\begin{align}
\int f_{\ast }\left( \ln \left( \frac{f}{f_{\ast }}\right) \right) ^{2}%
\mathbbm{1}_{f\geq f_{\ast }}d\mu \leq & \int f_{\ast }\left( \ln \left( 
\frac{c_{+}}{f_{\ast }}\right) \right) ^{2}d\mu  \notag \\
\leq & 2\left( \ln \left( c_{+}\right) \right) ^{2}+2P\left( \left( \ln
f_{\ast }\right) ^{2}\right) \text{ .}  \label{ineq_2_gene_bis}
\end{align}%
Furthermore we have $f_{\ast }^{p-1}+f_{\ast }^{-p}\geq 1$, so $P\left(
\left( \ln f_{\ast }\right) ^{2}\right) \leq P\left( f_{\ast }^{p-1}\left(
\ln f_{\ast }\right) ^{2}\right) +P\left( f_{\ast }^{-p}\left( \ln f_{\ast
}\right) ^{2}\right) \leq J+Q$ and by (\ref{ineq_2_gene_bis}), 
\begin{equation*}
\int f_{\ast }\left( \ln \left( \frac{f}{f_{\ast }}\right) \right) ^{2}%
\mathbbm{1}_{f\geq f_{\ast }}d\mu \leq 2\left( \left( \ln \left(
c_{+}\right) \right) ^{2}+J+Q\right) <+\infty \text{ ,}
\end{equation*}%
where the finiteness of the upper bound comes from (\ref{moment_p_bis}).
Inequality (\ref{margin_like_soft_gauche_gene}) then easily follows.
\end{proof}

\bigskip

\begin{proof}[Proof of Proposition \protect\ref%
{lemma_margin_like_bounded_below}]
Let us first prove Inequality (\ref{margin_like_soft}). Considering the
proof of Inequality (\ref{margin_like_soft_droite_gene}) of Lemma \ref%
{lem:margin_like_unbounded} given above, we see that it is sufficient to
bound the second term in the right-hand side of (\ref{ineq_0_gene}), applied
with $f=f_{m}$, in terms of $A_{\min },J$ and $p$ only. It holds%
\begin{equation}
\int \frac{f_{m}}{f_{\ast }^{p-1}}\left( \ln \left( \frac{f_{m}}{f_{\ast }}%
\right) \right) ^{2}\mathbbm{1}_{f_{m}\geq f_{\ast }}d\mu \leq A_{\min }\int
\left( \frac{f_{m}}{A_{\min }}\right) ^{p}\left( \ln \left( \frac{f_{m}}{%
A_{\min }}\right) \right) ^{2}d\mu \text{ .}  \label{ineq_2_lem}
\end{equation}%
Now, we set $h$ an auxilliary function, defined by $h\left( x\right)
=x^{p}\left( \ln x\right) ^{2}$ for any $x\geq 1$. It is easily seen that $h$
is convex. So it holds, for any $I\in \Lambda _{m}$,%
\begin{equation*}
h\left( \frac{f_{m}\left( I\right) }{A_{\min }}\right) =h\left( \int_{I}%
\frac{f_{\ast }}{A_{\min }}\frac{d\mu }{\mu \left( I\right) }\right) \leq
\int_{I}h\left( \frac{f_{\ast }}{A_{\min }}\right) \frac{d\mu }{\mu \left(
I\right) }\text{ .}
\end{equation*}%
From the latter inequality and from (\ref{ineq_2_lem}), we deduce%
\begin{equation*}
\int \frac{f_{m}}{f_{\ast }^{p-1}}\left( \ln \left( \frac{f_{m}}{f_{\ast }}%
\right) \right) ^{2}\mathbbm{1}_{f_{m}\geq f_{\ast }}d\mu \leq A_{\min }\int
h\left( \frac{f_{m}}{A_{\min }}\right) d\mu \leq A_{\min }\int h\left( \frac{%
f_{\ast }}{A_{\min }}\right) d\mu \leq 4A_{\min }^{1-p}\left( \left( \ln
A_{\min }\right) ^{2}\vee 1\right) J\text{ .}  
\end{equation*}%
Inequality (\ref{margin_like_soft}) is thus proved.

In the same manner, to establish Inequality (\ref{margin_like_gauche_soft})
it suffices to adapt the proof of inequality (\ref%
{margin_like_soft_gauche_gene}) given above by controlling the second term
in the right-hand side of (\ref{ineq_0_gauche_gene}), applied with $f=f_{m}$%
, in terms of $A_{\min },p$ and $P\left( \ln f_{\ast }\right) ^{2}$ only.
Let us notice that the function $f$ defined on $\left[ 1,+\infty \right) $
by $f\left( x\right) =x\left( \ln x\right) ^{2}$ is convex. We have%
\begin{equation*}
\int f_{\ast }\left( \ln \left( \frac{f_{m}}{f_{\ast }}\right) \right) ^{2}%
\mathbbm{1}_{f_{m}\geq f_{\ast }}d\mu \leq \int f_{m}\left( \ln \left( \frac{%
f_{m}}{A_{\min }}\right) \right) ^{2}d\mu =A_{\min }\int f\left( \frac{f_{m}%
}{A_{\min }}\right) d\mu \text{ .}
\end{equation*}%
Now, for any $I\in \Lambda _{m}$, it holds $f\left( \frac{f_{m}\left(
I\right) }{A_{\min }}\right) =f\left( \int_{I}\frac{f_{\ast }}{A_{\min }}%
\frac{d\mu }{\mu \left( I\right) }\right) \leq \int_{I}f\left( \frac{f_{\ast
}}{A_{\min }}\right) \frac{d\mu }{\mu \left( I\right) }$. Hence,%
\begin{equation*}
\int f_{\ast }\left( \ln \left( \frac{f_{m}}{f_{\ast }}\right) \right) ^{2}%
\mathbbm{1}_{f_{m}\geq f_{\ast }}d\mu \leq A_{\min }\int f\left( \frac{%
f_{\ast }}{A_{\min }}\right) d\mu \leq 2P\left( \ln f_{\ast }\right)
^{2}+2\left( \ln A_{\min }\right) ^{2}\text{ ,}
\end{equation*}%
which gives the desired upper-bound and proves (\ref{margin_like_gauche_soft}%
). In the event that $f_{\ast }\in L_{\infty }\left( \mu \right) $, we have to
prove (\ref{margin_like_opt}). We have ${\inf_{z\in \mathcal{Z}}}$\textup{$%
f_{\ast }$}${\left( z\right) \leq \inf_{z\in \mathcal{Z}}}$\textup{$f_{m}$}${%
\left( z\right) \leq \left\Vert f_{m}\right\Vert _{\infty }\leq \left\Vert
f_{\ast }\right\Vert _{\infty }}$, so it holds%
\begin{equation*}
P\left[ \left( \frac{f_{m}}{f_{\ast }}\vee 1\right) \left( \ln \left( \frac{%
f_{m}}{f_{\ast }}\right) \right) ^{2}\right] \vee P\left[ \left( \frac{%
f_{\ast }}{f_{m}}\vee 1\right) ^{r}\left( \ln \left( \frac{f_{m}}{f_{\ast }}%
\right) \right) ^{2}\right] \leq \left( \frac{{\left\Vert f_{\ast
}\right\Vert _{\infty }}}{A_{\min }}\right) ^{r\vee 1}P\left[ \left( \ln
\left( \frac{f_{m}}{f_{\ast }}\right) \right) ^{2}\right] \text{ .}
\end{equation*}%
Now, Inequality (\ref{margin_like_opt}) is a direct consequence of Lemma 1
of Barron and Sheu \cite{BarSheu:91}, which contains the following
inequality, 
\begin{equation*}
P\left[ \left( \ln \left( \frac{f_{m}}{f_{\ast }}\right) \right) ^{2}\right]
\leq 2\exp \left( \left\Vert \ln \left( \frac{f_{m}}{f_{\ast }}\right)
\right\Vert _{\infty }\right) \mathcal{K}\left( f_{\ast },f_{m}\right) \text{
}.
\end{equation*}%
This finishes the proof of Lemma \ref{lemma_margin_like_bounded_below}.
\end{proof}

\bibliographystyle{abbrv}
\bibliography{Slope_heuristics_regression_13}

\newcommand{\noop}[1]{}
\begin{thebibliography}{10}

\bibitem{MR2281879}
F.~Abramovich, Y.~Benjamini, D.~L. Donoho, and I.~M. Johnstone.
\newblock Adapting to unknown sparsity by controlling the false discovery rate.
\newblock {\em Ann. Statist.}, 34(2):584--653, 2006.

\bibitem{Akaike:73}
H.~Akaike.
\newblock Information theory and an extension of the maximum likelihood
  principle.
\newblock In {\em Second {I}nternational {S}ymposium on {I}nformation {T}heory
  ({T}sahkadsor, 1971)}, pages 267--281. Akad\'emiai Kiad\'o, Budapest, 1973.

\bibitem{Arlot:07}
S.~Arlot.
\newblock {\em Resampling and Model Selection}.
\newblock PhD thesis, University Paris-Sud 11, Dec. 2007.
\newblock oai:tel.archives-ouvertes.fr:tel-00198803\_v1.

\bibitem{Arl:2008a}
S.~Arlot.
\newblock {$V$}-fold cross-validation improved: {$V$}-fold penalization, Feb.
  2008.
\newblock arXiv:0802.0566v2.

\bibitem{Arl:09}
S.~Arlot.
\newblock Model selection by resampling penalization.
\newblock {\em Electron. J. Stat.}, 3:557--624, 2009.

\bibitem{Arl_Bar:2008}
S.~Arlot and P.~L. Bartlett.
\newblock Margin-adaptive model selection in statistical learning.
\newblock {\em Bernoulli}, 17(2):687--713, 05 2011.

\bibitem{ArlotCelisse:10}
S.~Arlot and A.~C\'elisse.
\newblock A survey of cross-validation procedures for model selection.
\newblock {\em Stat. Surv.}, 4:40--79, 2010.

\bibitem{ArlotMassart:09}
S.~Arlot and P.~Massart.
\newblock Data-driven calibration of penalties for least-squares regression.
\newblock {\em J. Mach. Learn. Res.}, 10:245--279 (electronic), 2009.

\bibitem{MR2449129}
Y.~Baraud and L.~Birg\'e.
\newblock Estimating the intensity of a random measure by histogram type
  estimators.
\newblock {\em Probab. Theory Related Fields}, 143(1-2):239--284, 2009.

\bibitem{BarBirgSart:17}
Y.~Baraud, L.~Birg\'e, and M.~Sart.
\newblock A new method for estimation and model selection: {$\rho$}-estimation.
\newblock {\em Invent. Math.}, 207(2):425--517, 2017.

\bibitem{BarBirMassart:99}
A.~Barron, L.~Birg{\'e}, and P.~Massart.
\newblock Risk bounds for model selection via penalization.
\newblock {\em Probab. Theory Related Fields}, 113(3):301--413, 1999.

\bibitem{BarSheu:91}
A.~Barron and C.~Sheu.
\newblock Approximation of density functions by sequences of exponential
  families.
\newblock {\em Ann. Statist.}, 19(3):1347--1369, 1991.

\bibitem{MR921623}
P.~Bauer, B.~M. P\"otscher, and P.~Hackl.
\newblock Model selection by multiple test procedures.
\newblock {\em Statistics}, 19(1):39--44, 1988.

\bibitem{bellec2016bounds}
P.~Bellec and A.~Tsybakov.
\newblock Bounds on the prediction error of penalized least squares estimators
  with convex penalty.
\newblock In V.~Panov, editor, {\em Modern Problems of Stochastic Analysis and
  Statistics}, pages 315--333, Cham, 2017. Springer International Publishing.

\bibitem{bellec2017towards}
P.~C. Bellec, G.~Lecu{\'e}, and A.~B. Tsybakov.
\newblock Towards the study of least squares estimators with convex penalty.
\newblock {\em arXiv preprint arXiv:1701.09120}, 2017.

\bibitem{MR2668704}
Y.~Benjamini and Y.~Gavrilov.
\newblock A simple forward selection procedure based on false discovery rate
  control.
\newblock {\em Ann. Appl. Stat.}, 3(1):179--198, 2009.

\bibitem{MR2156820}
Y.~Benjamini and D.~Yekutieli.
\newblock False discovery rate-adjusted multiple confidence intervals for
  selected parameters.
\newblock {\em J. Amer. Statist. Assoc.}, 100(469):71--93, 2005.
\newblock With comments and a rejoinder by the authors.

\bibitem{berlinet1994comparison}
A.~Berlinet and L.~Devroye.
\newblock A comparison of kernel density estimates.
\newblock {\em Publications de l'Institut de Statistique de l'Universit{\'e} de
  Paris}, 38(3):3--59, 1994.

\bibitem{MR3210971}
P.~Bühlmann, L.~Meier, and S.~van~de Geer.
\newblock Discussion: ``a significance test for the lasso''.
\newblock {\em Ann. Statist.}, 42(2):469--477, 04 2014.

\bibitem{MR2219712}
L.~Birg\'e.
\newblock Model selection via testing: an alternative to (penalized) maximum
  likelihood estimators.
\newblock {\em Ann. Inst. H. Poincar\'e Probab. Statist.}, 42(3):273--325,
  2006.

\bibitem{MR3186748}
L.~Birg\'e.
\newblock Robust tests for model selection.
\newblock In {\em From probability to statistics and back: high-dimensional
  models and processes}, volume~9 of {\em Inst. Math. Stat. (IMS) Collect.},
  pages 47--64. Inst. Math. Statist., Beachwood, OH, 2013.

\bibitem{BirRozen:06}
L.~Birg{\'e} and Y.~Rozenholc.
\newblock How many bins should be put in a regular histogram.
\newblock {\em ESAIM Probab. Stat.}, 10:24--45 (electronic), 2006.

\bibitem{BoucheronLugosiMassart:2013}
S.~Boucheron, G.~Lugosi, and P.~Massart.
\newblock {\em Concentration Inequalities}.
\newblock Oxford University Press, 2013.

\bibitem{BouLugMas:13}
S.~Boucheron, G.~Lugosi, and P.~Massart.
\newblock {\em Concentration Inequalities: A Nonasymptotic Theory of
  Independence}.
\newblock Oxford University Press, Oxford, 2013.

\bibitem{BouMas:10}
S.~Boucheron and P.~Massart.
\newblock A high-dimensional {W}ilks phenomenon.
\newblock {\em Probab. Theory Related Fields}, 150(3-4):405--433, 2011.

\bibitem{Bousquet:02}
O.~Bousquet.
\newblock A {B}ennett concentration inequality and its application to suprema
  of empirical processes.
\newblock {\em C. R. Math. Acad. Sci. Paris}, 334(6):495--500, 2002.

\bibitem{MR2280619}
F.~Bunea, A.~B. Tsybakov, and M.~H. Wegkamp.
\newblock Aggregation and sparsity via {$l_1$} penalized least squares.
\newblock In {\em Learning theory}, volume 4005 of {\em Lecture Notes in
  Comput. Sci.}, pages 379--391. Springer, Berlin, 2006.

\bibitem{Burman:02}
P.~Burman.
\newblock Estimation of equifrequency histogram.
\newblock {\em Statist. Probab. Lett.}, 56(3):227--238, 2002.

\bibitem{Castellan:99}
G.~Castellan.
\newblock Modified {A}kaike's criterion for histogram density estimation.
\newblock {\em Technical report $\sharp$99.61, Universit\'e Paris-Sud}, 1999.

\bibitem{Castellan:03}
G.~Castellan.
\newblock Density estimation via exponential model selection.
\newblock {\em IEEE Trans. Inform. Theory}, 49(8):2052--2060, 2003.

\bibitem{chatterjee2014}
S.~Chatterjee.
\newblock A new perspective on least squares under convex constraint.
\newblock {\em Ann. Statist.}, 42(6):2340--2381, 12 2014.

\bibitem{ClaeskensHjort:08}
G.~Claeskens and N.~L. Hjort.
\newblock {\em Model selection and model averaging}.
\newblock Cambridge Series in Statistical and Probabilistic Mathematics.
  Cambridge University Press, Cambridge, 2008.

\bibitem{Csiszar:75}
I.~Csisz{\'a}r.
\newblock {$I$}-divergence geometry of probability distributions and
  minimization problems.
\newblock {\em Ann. Probab.}, 3(1):146--158, 1975.

\bibitem{MR3184277}
T.~Dickhaus.
\newblock {\em Simultaneous statistical inference}.
\newblock Springer, Heidelberg, 2014.
\newblock With applications in the life sciences.

\bibitem{grunwald2007minimum}
P.~D. Gr{\"u}nwald.
\newblock {\em The minimum description length principle}.
\newblock MIT press, 2007.

\bibitem{HurTsai:89}
C.~M. Hurvich and C.-L. Tsai.
\newblock Model selection for least absolute deviations regression in small
  samples.
\newblock {\em Statist. Probab. Lett.}, 9(3):259--265, 1990.

\bibitem{MR3174645}
J.~T.~G. Hwang and Z.~Zhao.
\newblock Empirical {B}ayes confidence intervals for selected parameters in
  high-dimensional data.
\newblock {\em J. Amer. Statist. Assoc.}, 108(502):607--618, 2013.

\bibitem{Klein_Rio:05}
T.~Klein and E.~Rio.
\newblock Concentration around the mean for maxima of empirical processes.
\newblock {\em Ann. Probab.}, 33(3):1060--1077, 2005.

\bibitem{lecue2017learning}
G.~Lecu{\'e} and M.~Lerasle.
\newblock Learning from {MOM}'s principles: {L}e {C}am's approach.
\newblock 2017.

\bibitem{MR3210977}
R.~Lockhart, J.~Taylor, R.~J. Tibshirani, and R.~Tibshirani.
\newblock Rejoinder: ``a significance test for the lasso''.
\newblock {\em Ann. Statist.}, 42(2):518--531, 04 2014.

\bibitem{MR3210970}
R.~Lockhart, J.~Taylor, R.~J. Tibshirani, and R.~Tibshirani.
\newblock A significance test for the lasso.
\newblock {\em Ann. Statist.}, 42(2):413--468, 04 2014.

\bibitem{MamTsy:99}
E.~Mammen and A.~Tsybakov.
\newblock Smooth discrimination analysis.
\newblock {\em Ann.Stat.}, 27:1808--1829, 1999.

\bibitem{Massart:07}
P.~Massart.
\newblock {\em Concentration inequalities and model selection}, volume 1896 of
  {\em Lecture Notes in Mathematics}.
\newblock Springer, Berlin, 2007.
\newblock Lectures from the 33rd Summer School on Probability Theory held in
  Saint-Flour, July 6--23, 2003, With a foreword by Jean Picard.

\bibitem{MassartNedelec:06}
P.~Massart and E.~N\'{e}d\'{e}lec.
\newblock Risks bounds for statistical learning.
\newblock {\em Ann.Stat.}, 34(5):2326--2366, 2006.

\bibitem{Mildenberger-Weinert2012}
T.~Mildenberger and H.~Weinert.
\newblock The benchden package: Benchmark densities for nonparametric density
  estimation.
\newblock {\em Journal of Statistical Software, Articles}, 46(14):1--14, 2012.

\bibitem{MurovandeGeer:15}
A.~Muro and S.~van~de Geer.
\newblock Concentration behavior of the penalized least squares estimator.
\newblock {\em arXiv preprint arXiv:1511.08698}, 2015.

\bibitem{MR2395714}
J.~Qiu and J.~T.~G. Hwang.
\newblock Sharp simultaneous confidence intervals for the means of selected
  populations with application to microarray data analysis.
\newblock {\em Biometrics}, 63(3):767--776, 2007.

\bibitem{saum:12}
A.~Saumard.
\newblock Optimal upper and lower bounds for the true and empirical excess
  risks in heteroscedastic least-squares regression.
\newblock {\em Electron. J. Statist.}, 6(1-2):579--655, 2012.

\bibitem{MR2183942}
J.~Sch\"afer and K.~Strimmer.
\newblock A shrinkage approach to large-scale covariance matrix estimation and
  implications for functional genomics.
\newblock {\em Stat. Appl. Genet. Mol. Biol.}, 4:Art. 32, 28, 2005.

\bibitem{MR630098}
C.~M. Stein.
\newblock Estimation of the mean of a multivariate normal distribution.
\newblock {\em Ann. Statist.}, 9(6):1135--1151, 1981.

\bibitem{Stone:85}
C.~Stone.
\newblock An asymptotically optimal histogram selection rule.
\newblock In {\em Proceedings of the {B}erkeley conference in honor of {J}erzy
  {N}eyman and {J}ack {K}iefer, {V}ol.\ {II} ({B}erkeley, {C}alif., 1983)},
  Wadsworth Statist./Probab. Ser., pages 513--520, Belmont, CA, 1985.
  Wadsworth.

\bibitem{Sugiura:78}
N.~Sugiura.
\newblock Further analysts of the data by {A}kaike's information criterion and
  the finite corrections.
\newblock {\em Communications in Statistics - Theory and Methods}, 7(1):13--26,
  1978.

\bibitem{Tsy:04a}
A.~Tsybakov.
\newblock Optimal aggregation of classifiers in statistical learning.
\newblock {\em Ann.Stat.}, 32:135--166, 2004.

\bibitem{vandeGeerWain:16}
S.~van~de Geer and M.~J. Wainwright.
\newblock On concentration for (regularized) empirical risk minimization.
\newblock {\em Sankhya A}, 79(2):159--200, Aug 2017.

\bibitem{ZouHastieTib:07}
H.~Zou, T.~Hastie, and R.~Tibshirani.
\newblock On the ``degrees of freedom'' of the lasso.
\newblock {\em Ann. Statist.}, 35(5):2173--2192, 2007.

\bibitem{MR2823520}
V.~Zuber and K.~Strimmer.
\newblock High-dimensional regression and variable selection using {CAR}
  scores.
\newblock {\em Stat. Appl. Genet. Mol. Biol.}, 10:Art. 34, 29, 2011.

\end{thebibliography}

\appendix{}
\newpage

\section{Further Proofs and Theoretical Results}
\label{sec:appendix}

\subsection{Proofs Related to Section \protect\ref{section_risks_bounds_MLE} 
\label{section_deviation_bounds}}

Most of the arguments given in the proofs of this section are borrowed from
Castellan \cite{Castellan:99}. We essentially rearrange these arguments in a
more efficient way, thus obtaining better concentration bounds than in \cite%
{Castellan:99} (or \cite{Massart:07}, see also \cite%
{BoucheronLugosiMassart:2013}).

\noindent We also set, for any $\varepsilon >0$, the event $\Omega
_{m}\left( \varepsilon \right) $ where some control of $\hat{f}_{m}$ in
sup-norm is achieved,%
\begin{equation*}
\Omega _{m}\left( \varepsilon \right) =\left\{ \left\Vert \frac{\hat{f}%
_{m}-f_{m}}{f_{m}}\right\Vert _{\infty }\leq \varepsilon \right\} \text{ .}
\end{equation*}%
As we have the following formulas for the estimators and the projections of
the target,%
\begin{equation*}
\hat{f}_{m}=\sum_{I\in \Lambda _{m}}\frac{P_{n}\left( I\right) }{\mu \left(
I\right) }\mathbbm{1} _{I}\text{ \ \ },\text{ \ \ }f_{m}=\sum_{I\in \Lambda
_{m}}\frac{P\left( I\right) }{\mu \left( I\right) }\mathbbm{1} _{I}\text{ },
\end{equation*}%
we deduce that,

\begin{equation}
\left\Vert \frac{\hat{f}_{m}-f_{m}}{f_{m}}\right\Vert _{\infty }=\sup_{I\in
\Lambda _{m}}\frac{\left\vert \left( P_{n}-P\right) \left( I\right)
\right\vert }{P\left( I\right) }\text{ .}  \label{1-1}
\end{equation}%
Hence, it holds $\Omega _{m}\left( \varepsilon \right) =\bigcap_{I\in
m}\left\{ \left\vert P_{n}\left( I\right) -P\left( I\right) \right\vert \leq
\varepsilon P\left( I\right) \right\} $.

Let us state the main result of this section, concerning upper and lower
bounds for the true and empirical excess risks of the histograms on each
model. These bounds are optimal to the first order.

Before giving the proof of Theorem \ref{opt_bounds}, the following lemma
will be useful. It describes the consistency in sup-norm of the histogram
estimators, suitably normalized by the projections of the target on each
model.

\begin{lemma}
\label{prop_consistency_normalized_MLE}Let $\alpha ,A_{+}$ and $A_{\Lambda }$
be positive constants. Consider a finite partition $m$ of $\mathcal{Z}$,
with cardinality $D_{m}$. Assume%
\begin{equation}
0<A_{\Lambda }\leq D_{m}\inf_{I\in m}\left\{ P\left( I\right) \right\} \text{%
\ \ \ and \ \ }0<D_{m}\leq A_{+}\frac{n}{\ln (n+1)}\leq n\text{ .}
\label{lower_mu_2}
\end{equation}%
Then by setting%
\begin{equation}
R_{n}^{\infty }\left( m\right) =\sqrt{\frac{2\left( \alpha +1\right)
D_{m}\ln (n+1)}{A_{\Lambda }n}}+\frac{\left( \alpha +1\right) D_{m}\ln (n+1)%
}{A_{\Lambda }n}\text{ ,}  \label{def_R_n_infinity}
\end{equation}%
we get%
\begin{equation}
\mathbb{P}\left( \left\Vert \frac{\hat{f}_{m}-f_{m}}{f_{m}}\right\Vert
_{\infty }\leq R_{n}^{\infty }\left( m\right) \right) \geq 1-2(n+1)^{-\alpha
}\text{ .}  \label{rate_sup_norm_histo_MLE_2}
\end{equation}%
In other words, $\mathbb{P}\left( \Omega _{m}\left( R_{n}^{\infty }\left(
m\right) \right) \right) \geq 1-2(n+1)^{-\alpha }$ . In addition, there
exists a positive constant $A_{c}$, only depending on $\alpha ,A_{+}$ and $%
A_{\Lambda }$, such that $R_{n}^{\infty }\left( m\right) \leq A_{c}\sqrt{%
\frac{D_{m}\ln (n+1)}{n}}$. Furthermore, if%
\begin{equation*}
\frac{\left( \alpha +1\right) A_{+}}{A_{\Lambda }}\leq \tau =\sqrt{3}-\sqrt{2}<0.32%
\text{ ,}
\end{equation*}%
or if $n\geq n_{0}\left( \alpha ,A_{+},A_{\Lambda }\right) $, then $%
R_{n}^{\infty }\left( m\right) \leq1/2$ .
\end{lemma}

\begin{proof}[Proof of Lemma \protect\ref{prop_consistency_normalized_MLE}]
Let $\beta >0$ to be fixed later. Recall that, by (\ref{1-1}) we have 
\begin{equation}
\left\Vert \frac{\hat{f}_{m}-f_{m}}{f_{m}}\right\Vert _{\infty }=\sup_{I\in
\Lambda _{m}}\frac{\left\vert \left( P_{n}-P\right) \left( I\right)
\right\vert }{P\left( I\right) }\text{ }.  \label{normalized}
\end{equation}%
By Bernstein's inequality (see Proposition 2.9 in \cite{Massart:07}) applied
to the random variables $\mathbbm{1}_{\xi _{i}\in I}$ we get, for all $x>0$, 
\begin{equation*}
\mathbb{P}\left[ \left\vert \left( P_{n}-P\right) \left( I\right)
\right\vert \geq \sqrt{\frac{2P\left( I\right) x}{n}}+\frac{x}{3n}\right]
\leq 2\exp \left( -x\right) \text{ }.
\end{equation*}%
Taking $x=\beta \ln (n+1)$ and normalizing by the quantity $P\left( I\right)
>0$\ we get 
\begin{equation}
\mathbb{P}\left[ \frac{\left\vert \left( P_{n}-P\right) \left( I\right)
\right\vert }{P\left( I\right) }\geq \sqrt{\frac{2\beta \ln (n+1)}{P\left(
I\right) n}}+\frac{\beta \ln (n+1)}{P\left( I\right) 3n}\right] \leq
2(n+1)^{-\beta }\text{ }.  \label{dev_normalized}
\end{equation}%
Now, by the first inequality in (\ref{lower_mu_2}), we have $0<P\left(
I\right) ^{-1}\leq A_{\Lambda }^{-1}D_{m}$. Hence, using (\ref%
{dev_normalized}) we get 
\begin{equation}
\mathbb{P}\left[ \frac{\left\vert \left( P_{n}-P\right) \left( I\right)
\right\vert }{P\left( I\right) }\geq \sqrt{\frac{2\beta D_{m}\ln (n+1)}{%
A_{\Lambda }n}}+\frac{\beta D_{m}\ln (n+1)}{A_{\Lambda }n}\right] \leq
2(n+1)^{-\beta }\text{ ,}  \label{dev_norm_2}
\end{equation}%
We then deduce from (\ref{normalized}) and (\ref{dev_norm_2}) that 
\begin{equation*}
\mathbb{P}\left[ \left\Vert \frac{\hat{f}_{m}-f_{m}}{f_{m}}\right\Vert
_{\infty }\geq R_{n}^{\infty }\left( m\right) \right] \leq \frac{2D_{m}}{%
(n+1)^{\beta }}
\end{equation*}%
and, since $D_{m}\leq n$, taking $\beta =\alpha +1$ yields Inequality (\ref%
{rate_sup_norm_histo_MLE_2}). The other facts of Lemma \ref%
{prop_consistency_normalized_MLE}\ then follow from simple computations.
\end{proof}

\begin{proof}[Proof of Theorem \protect\ref{opt_bounds}]
Recall that $\alpha $ is fixed. By Inequality (\ref{upper_chi_part}) in
Proposition \ref{prop_upper_dev_chi_2} applied with $\theta =\mathnormal{R}%
_{n}^{\infty }\left( m\right) $---where $\mathnormal{R}_{n}^{\infty }\left(
m\right) $ is defined in (\ref{def_R_n_infinity}) with our fixed value of $%
\alpha $---and $x=\alpha \ln (n+1)$, it holds with probability at least $%
1-(n+1)^{-\alpha }$,%
\begin{equation}
\chi _{n}\left( m\right) \boldsymbol{1}_{\Omega _{m}\left( \mathnormal{R}%
_{n}^{\infty }\left( m\right) \right) }\leq \sqrt{\frac{D_{m}}{n}}+\left( 1+%
\sqrt{2\mathnormal{R}_{n}^{\infty }\left( m\right) }+\frac{\mathnormal{R}%
_{n}^{\infty }\left( m\right) }{6}\right) \sqrt{\frac{2\alpha \ln (n+1)}{n}}%
\text{ .}  \label{upper_chi_1}
\end{equation}%
As $\mathnormal{R}_{n}^{\infty }\left( m\right) \leq L_{\alpha
,A_{+},A_{\Lambda }}\sqrt{D_{m}\ln (n+1)/n}\leq L_{\alpha ,A_{+},A_{\Lambda
}}$, (\ref{upper_chi_1}) gives%
\begin{align}
\chi _{n}\left( m\right) \boldsymbol{1}_{\Omega _{m}\left( \mathnormal{R}%
_{n}^{\infty }\left( m\right) \right) }\leq & \sqrt{\frac{D_{m}}{n}}%
+L_{\alpha ,A_{+},A_{\Lambda }}\sqrt{\frac{\ln (n+1)}{n}}  \notag \\
=& \sqrt{\frac{D_{m}}{n}}\left( 1+L_{\alpha ,A_{+},A_{\Lambda }}\sqrt{\frac{%
\ln (n+1)}{D_{m}}}\right) \text{ .}  \label{upper_chi_2}
\end{align}%
We set the event $\Omega _{0}$ on which we have%
\begin{equation*}
\left\Vert \frac{\hat{f}_{m}-f_{m}}{f_{m}}\right\Vert _{\infty }\leq
R_{n}^{\infty }\left( m\right) \text{ ,} 
\end{equation*}%
\begin{equation}
\chi _{n}\left( m\right) \leq \sqrt{\frac{D_{m}}{n}}\left( 1+L_{\alpha
,A_{+},A_{\Lambda }}\sqrt{\frac{\ln (n+1)}{D_{m}}}\right) \text{ ,}
\label{assump_upper_chi}
\end{equation}%
and%
\begin{equation*}
\chi _{n}\left( m\right) \geq \left( 1-A_{g}\left( \sqrt{\frac{\ln (n+1)}{%
D_{m}}}\vee \frac{\sqrt{\ln (n+1)}}{n^{1/4}}\right) \right) \sqrt{\frac{D_{m}%
}{n}}\text{ .}  
\end{equation*}%
In particular, $\Omega _{0}\subset \Omega _{m}\left( R_{n}^{\infty }\left(
m\right) \right) $. By (\ref{upper_chi_2}), Lemma \ref%
{prop_consistency_normalized_MLE} and Proposition \ref{prop_dev_gauche_chi},
it holds $\mathbb{P}\left( \Omega _{0}\right) \geq 1-4(n+1)^{-\alpha }$. It
suffices to prove the inequalities of Theorem \ref{opt_bounds}\ on $\Omega
_{0}$. The following inequalities, between the excess risk on $m$ and the
chi-square statistics, are shown in \cite{Castellan:99} (Inequalities
(2.13)). For any $\varepsilon \in \left( 0,1\right) $, on $\Omega _{m}\left(
\varepsilon \right) $, 
\begin{equation}
\frac{1-\varepsilon }{2\left( 1+\varepsilon \right) ^{2}}\chi _{n}^{2}\left(
m\right) \leq \mathcal{K}\left( f_{m},\hat{f}_{m}\right) \leq \frac{%
1+\varepsilon }{2\left( 1-\varepsilon \right) ^{2}}\chi _{n}^{2}\left(
m\right) \text{ }.  \label{castel_2_13}
\end{equation}%
Under the condition $\left( \alpha +1\right) A_{+}A_{\Lambda }^{-1}<\tau $,
Proposition \ref{prop_consistency_normalized_MLE}\ gives $\mathnormal{R}%
_{n}^{\infty }\left( m\right) <1$. Hence, by applying the right-hand side of
(\ref{castel_2_13}) with $\varepsilon =\mathnormal{R}_{n}^{\infty }\left(
m\right) \leq 1/2$, using (\ref{assump_upper_chi}) and the fact that $(1-\epsilon)^{-1}\leq 1+2\epsilon$ for $\epsilon \leq 1/2$, we get on $\Omega _{0}$,%
\begin{align*}
\mathcal{K}\left( f_{m},\hat{f}_{m}\right) \leq & \frac{1+\mathnormal{R}%
_{n}^{\infty }\left( m\right) }{2\left( 1-\mathnormal{R}_{n}^{\infty }\left(
m\right) \right) ^{2}}\chi _{n}^{2}\left( m\right) \\
\leq & \left( \frac{1}{2}+L_{\alpha ,A_{+},A_{\Lambda }}\sqrt{\frac{D_{m}\ln
(n+1)}{n}}\right) \frac{D_{m}}{n}\left( 1+L_{\alpha ,A_{+},A_{\Lambda }}%
\sqrt{\frac{\ln (n+1)}{D_{m}}}\right) ^{2}\text{ .}
\end{align*}%
Then simple computations allow to get the right-hand side inequality in (\ref%
{upper_true_risk}).

By applying the left-hand side of (\ref{castel_2_13}) with $\varepsilon
=R_{n}^{\infty }\left( m\right)\leq 1/2 $ and using the fact that $(1+\epsilon)^{-1}\geq 1-\epsilon$, we also get on $\Omega _{0}$,%
\begin{align*}
\mathcal{K}\left( f_{m},\hat{f}_{m}\right) \geq & \frac{1-R_{n}^{\infty
}\left( m\right) }{2\left( 1+R_{n}^{\infty }\left( m\right) \right) ^{2}}%
\chi _{n}^{2}\left( m\right) \\
\geq & (1-R_{n}^{\infty }\left( m\right))^3\left( 1-A_{g}\left( \sqrt{\frac{\ln (n+1)}{%
D_{m}}}\vee \frac{\sqrt{\ln (n+1)}}{n^{1/4}}\right) \right) ^{2}\frac{D_{m}}{2n} \\
\geq & (1-(3R_{n}^{\infty }\left( m\right)\wedge 1))\left( 1-2A_{g}\left( \sqrt{\frac{\ln (n+1)}{%
D_{m}}}\vee \frac{\sqrt{\ln (n+1)}}{n^{1/4}}\right) \right) \frac{D_{m}}{2n}
\end{align*}%
The left-hand side inequality in (\ref{upper_true_risk}) then follows by
simple computations, noticing in particular that $n^{-1/4}\sqrt{\ln (n+1)}\leq \sqrt{\ln
(n+1)/D_{m}}\vee \sqrt{D_{m}\ln (n+1)/n}$.

Inequalities in (\ref{upper_emp_risk}) follow from the same kind of
arguments as those involved in the proofs of inequalities in (\ref%
{upper_true_risk}). Indeed, from \cite[Lemma 7.24]{Massart:07}---or \cite[
Lemma 2.3]{Castellan:99} ---, it holds 
\begin{equation*}
\frac{1}{2}\int \left( \hat{f}_{m}\wedge f_{m}\right) \left( \ln \frac{\hat{f%
}_{m}}{f_{m}}\right) ^{2}d\mu \leq \mathcal{K}\left( \hat{f}%
_{m},f_{m}\right) \leq \frac{1}{2}\int \left( \hat{f}_{m}\vee f_{m}\right)
\left( \ln \frac{\hat{f}_{m}}{f_{m}}\right) ^{2}d\mu \text{ }.
\end{equation*}%
We deduce that for $\varepsilon \in \left( 0,1\right) $, we have on $\Omega
_{m}\left( \varepsilon \right) $, 
\begin{equation}
\frac{1-\varepsilon }{2}\int f_{m}\left( \ln \frac{\hat{f}_{m}}{f_{m}}%
\right) ^{2}d\mu \leq \mathcal{K}\left( \hat{f}_{m},f_{m}\right) \leq \frac{%
1+\varepsilon }{2}\int f_{m}\left( \ln \frac{\hat{f}_{m}}{f_{m}}\right)
^{2}d\mu \text{ .}  \label{bound_K_ln}
\end{equation}%
As for every $x>0$, we have $\left( 1\vee x\right) ^{-1}\leq \left(
x-1\right) ^{-1}\ln x\leq \left( 1\wedge x\right) ^{-1}$, (\ref{bound_K_ln})
leads by simple computations to the following inequalities 
\begin{equation*}
\frac{1-\varepsilon }{2\left( 1+\varepsilon \right) ^{2}}\chi _{n}^{2}\left(
m\right) \leq \mathcal{K}\left( \hat{f}_{m},f_{m}\right) \leq \frac{%
1+\varepsilon }{2\left( 1-\varepsilon \right) ^{2}}\chi _{n}^{2}\left(
m\right) \text{ }.
\end{equation*}%
We thus have the same upper and lower bounds, in terms of the chi-square
statistic $\chi _{n}^{2}\left( m\right) $, for the empirical excess risk as
for the true excess risk.
\end{proof}

\subsection{Oracle inequalities and dimension guarantees\label%
{section_oracle_proofs}}

Using the notations of Section \ref{section_results_bounded_setting}, we
define the set of assumptions (\textbf{SA}$_{0}$) to be the conjunction of
assumptions (\textbf{P1}), (\textbf{P2}), (\textbf{Asm}) and (\textbf{Alr}).
The set of assumptions (\textbf{SA}) of Section \ref%
{section_results_bounded_setting} thus consists on assuming (\textbf{SA}$%
_{0} $) together with (\textbf{Ap}).

For some of the following results, we will also need the following
assumptions.

\begin{description}
\item[(P3)] Richness of $\mathcal{M}_{n}$: there exist $m_{0},m_{1}\in 
\mathcal{M}_{n}$ such that $D_{m_{0}}\in \left[ \sqrt{n},c_{rich}\sqrt{n}%
\right] $ and $D_{m_{1}}\geq A_{rich}n\left( \ln (n+1)\right) ^{-2}.$

\item[(Ap$_{u}$)] The bias decreases as a power of $D_{m}$: there exist $%
\beta _{+}>0$ and $C_{+}>0$ such that 
\begin{equation*}
\mathcal{K}\left( f_{\ast },f_{m}\right) \leq C_{+}D_{m}^{-\beta _{+}}\text{ 
}.
\end{equation*}

\item[(\textbf{Ap})] The bias decreases like a power of $D_{m}$: there exist 
$\beta _{-}\geq \beta _{+}>0$ and $C_{+},C_{-}>0$ such that%
\begin{equation*}
C_{-}D_{m}^{-\beta _{-}}\leq \mathcal{K}\left( f_{\ast },f_{m}\right) \leq
C_{+}D_{m}^{-\beta _{+}}\text{ }.
\end{equation*}
\end{description}

Theorem \ref{theorem_opt_pen_MLE} is a direct corollary of the following
theorem.

\begin{theorem}
\label{theorem_opt_pen_MLE_gene}Take $n\geq 1$ and $r\in \left( 0,p-1\right) 
$. Assume that the set of assumptions (\textbf{SA}$_{0}$) holds and that for
some $\theta \in \left( 1/2,+\infty \right) $ and $\Delta >0$,%
\begin{equation*}
\pen%
\left( m\right) =\left( \theta +\Delta \varepsilon _{n}^{+}\left( m\right)
\right) \frac{D_{m}}{n}\text{ ,}  
\end{equation*}%
for every model $m\in \mathcal{M}_{n}$. Then there exists an event $\Omega
_{n}$ of probability $1-(n+1)^{-2}$ and some positive constants $A_{1},$ $%
A_{2}$ depending only on the constants defined in (\textbf{SA}$_{0}$) such
that, if $\Delta \geq \left( \theta -1\right) _{-}A_{1}+A_{2}>0$ then we
have on $\Omega _{n}$,%
\begin{equation}
\mathcal{K}\left( f_{\ast },\hat{f}_{\widehat{m}}\right) \leq \frac{%
1+2\left( \theta -1\right) _{+}+L_{\text{(\textbf{SA})},\theta ,r}\left( \ln
(n+1)\right) ^{-1/2}}{1-2\left( \theta -1\right) _{-}}\inf_{m\in \mathcal{M}%
_{n}}\left\{ \mathcal{K}\left( f_{\ast },\hat{f}_{m}\right) \right\} +L_{%
\text{(\textbf{SA})},\theta ,r}\frac{\left( \ln (n+1)\right) ^{\frac{3p-1-r}{%
2\left( p+1+r\right) }}}{n^{\frac{p}{p+1+r}}}\text{ }.  \label{oracle_gene_1}
\end{equation}%
Assume furthermore that Assumption (\textbf{Ap}$_{u}$) holds. Then it holds
on $\Omega _{n}$,%
\begin{equation*}
D_{\widehat{m}}\leq L_{\text{(\textbf{SA)}},\Delta ,\theta ,r}n^{^{\frac{1}{%
2+\beta _{+}\left( 1-\frac{r+1}{p}\right) }}}\sqrt{\ln (n+1)}\text{ \ \ , \
\ }D_{m_{\ast }}\leq L_{\text{(\textbf{SA})}}n^{\frac{1}{1+\beta _{+}}}\text{
.}
\end{equation*}%
In particular, if we are in the case where $p<\beta _{+}$\ then Inequality (%
\ref{oracle_gene_1}) reduces to 
\begin{equation}
\mathcal{K}\left( f_{\ast },\hat{f}_{\widehat{m}}\right) \leq L_{\text{(%
\textbf{SA})},\theta ,r}\frac{\left( \ln (n+1)\right) ^{\frac{3p-1-r}{%
2\left( p+1+r\right) }}}{n^{\frac{p}{p+1+r}}}\text{ .}
\label{bound_unfavorable}
\end{equation}%
Grant Assumption (\textbf{Ap}). Then it holds on $\Omega _{n}$,%
\begin{equation*}
L_{\Delta ,\theta ,\text{(\textbf{SA})}}^{(1)}\frac{n^{\frac{\beta _{+}}{%
\beta _{-}\left( 1+\beta _{+}\right) }}}{\left( \ln (n+1)\right) ^{\frac{1}{%
\beta _{-}}}}\leq D_{\widehat{m}}\leq L_{\Delta ,\theta ,\text{(\textbf{SA})}%
}^{(2)}n^{^{\frac{1}{2+\beta _{+}\left( 1-\frac{r+1}{p}\right) }}}\sqrt{\ln
(n+1)}\text{ ,}
\end{equation*}%
\begin{equation*}
L_{\text{(\textbf{SA})}}^{(1)}n^{\frac{\beta _{+}}{\left( 1+\beta
_{+}\right) \beta _{-}}}\leq D_{m_{\ast }}\leq L_{\text{(\textbf{SA})}%
}^{(2)}n^{\frac{1}{1+\beta _{+}}}
\end{equation*}%
and%
\begin{equation}
\mathcal{K}\left( f_{\ast },\hat{f}_{\widehat{m}}\right) \leq \frac{%
1+2\left( \theta -1\right) _{+}+L_{\text{(\textbf{SA})},\theta ,r}n^{-\frac{%
\beta _{+}}{\left( 1+\beta _{+}\right) \beta _{-}}}\sqrt{\ln (n+1)}}{%
1-2\left( \theta -1\right) _{-}}\inf_{m\in \mathcal{M}_{n}}\left\{ \mathcal{K%
}\left( f_{\ast },\hat{f}_{m}\right) \right\} +L_{\text{(\textbf{SA})}%
,\theta ,r}\frac{\left( \ln (n+1)\right) ^{\frac{3p-1-r}{2\left(
p+1+r\right) }}}{n^{\frac{p}{p+1+r}}}\text{ .}  \label{oracle_um_ap}
\end{equation}%
Furthermore, if $\beta _{-}<p\left( 1+\beta _{+}\right) /(1+p+r)$ or $%
p/(1+r)>\beta _{-}+\beta _{-}/\beta _{+}-1$, then we have on $\Omega _{n}$,%
\begin{equation}
\mathcal{K}\left( f_{\ast },\hat{f}_{\widehat{m}}\right) \leq \frac{%
1+2\left( \theta -1\right) _{+}+L_{\text{(\textbf{SA})},\theta ,r}\left( \ln
(n+1)\right) ^{-1/2}}{1-2\left( \theta -1\right) _{-}}\inf_{m\in \mathcal{M}%
_{n}}\left\{ \mathcal{K}\left( f_{\ast },\hat{f}_{m}\right) \right\} \text{ }%
.  \label{oracle_um_ap_opt}
\end{equation}
\end{theorem}

We obtain in Theorem \ref{theorem_opt_pen_MLE_gene} oracle inequalities and
dimension bounds for the oracle and selected models.

In order to avoid cumbersome notations in the proofs of Theorem \ref%
{theorem_opt_pen_MLE_gene}, when generic constants $L$ and $n_{0}$ depend on
constants defined in the sets of assumptions (\textbf{SA}$_{0}$) or (\textbf{%
SA}), we will note $L_{\text{(\textbf{SA})}}$ and $n_{0}\left( \text{(%
\textbf{SA})}\right) $. The values of the these constants may change from
line to line, or even whithin one line.

\noindent \textbf{Proof of Theorem \ref{theorem_opt_pen_MLE}}

\begin{itemize}
\item \underline{Proof of oracle Inequality (\ref{oracle_gene_1}):}
\end{itemize}

From the definition of the selected model $\widehat{m}$ given in (\ref%
{def_proc_2_MLE}), $\widehat{m}$ minimizes 
\begin{equation*}
\crit\left( m\right) :=P_{n}\left( \gamma (\hat{f}_{m})\right) +\pen\left(
m\right) \text{ ,}  
\end{equation*}%
over the models $m\in \mathcal{M}_{n}$. Hence, $\widehat{m}$ also minimizes 
\begin{equation}
\crit^{\prime }\left( m\right) :=\crit\left( m\right) -P_{n}\left( \gamma
\left( f_{\ast }\right) \right)  \label{def_criterion_prime_1}
\end{equation}%
over the collection $\mathcal{M}_{n}$. Let us write 
\begin{align*}
\mathcal{K}\left( f_{\ast },\hat{f}_{m}\right) & =P\left( \gamma (\hat{f}%
_{m})-\gamma \left( f_{\ast }\right) \right) \\
& =P_{n}\left( \gamma (\hat{f}_{m})\right) +P_{n}\left( \gamma
(f_{m})-\gamma (\hat{f}_{m})\right) +\left( P_{n}-P\right) \left( \gamma
\left( f_{\ast }\right) -\gamma (f_{m})\right) \\
& +P\left( \gamma (\hat{f}_{m})-\gamma (f_{m})\right) -P_{n}\left( \gamma
\left( f_{\ast }\right) \right) \text{ }.
\end{align*}%
By setting 
\begin{equation*}
p_{1}\left( m\right) =P\left( \gamma (\hat{f}_{m})-\gamma (f_{m})\right) =%
\mathcal{K}\left( f_{m},\hat{f}_{m}\right) \text{ ,}
\end{equation*}%
\begin{equation*}
p_{2}\left( m\right) =P_{n}\left( \gamma (f_{m})-\gamma (\hat{f}_{m})\right)
=\mathcal{K}\left( \hat{f}_{m},f_{m}\right) \text{ ,}
\end{equation*}%
\begin{equation*}
\bar{\delta}\left( m\right) =\left( P_{n}-P\right) \left( \gamma \left(
f_{\ast }\right) -\gamma (f_{m})\right) =\left( P_{n}-P\right) \left( \ln
\left( \left. f_{m}\right/ f_{\ast }\right) \right)
\end{equation*}%
and 
\begin{equation*}
\pen_{\mathrm{id}}^{\prime }\left( m\right) =p_{1}\left( m\right)
+p_{2}\left( m\right) +\bar{\delta}\left( m\right) \text{ ,}
\end{equation*}%
we have 
\begin{equation*}
\mathcal{K}\left( f_{\ast },\hat{f}_{m}\right) =P_{n}\left( \gamma (\hat{f}%
_{m})\right) +p_{1}\left( m\right) +p_{2}\left( m\right) +\bar{\delta}\left(
m\right) -P_{n}\left( \gamma \left( f_{\ast }\right) \right)
\end{equation*}%
and by (\ref{def_criterion_prime_1}), 
\begin{equation}
\crit^{\prime }\left( m\right) =\mathcal{K}\left( f_{\ast },\hat{f}%
_{m}\right) +\left( \pen\left( m\right) -\pen_{\mathrm{id}}^{\prime }\left(
m\right) \right) \text{ .}  \label{formule_crit_prime}
\end{equation}%
As $\widehat{m}$ minimizes $\crit^{\prime }$ over $\mathcal{M}_{n}$, it is
therefore sufficient by (\ref{formule_crit_prime}) to control $\pen\left(
m\right) -\pen_{\mathrm{id}}^{\prime }\left( m\right) $ in terms of the
excess risk $\mathcal{K}\left( f_{\ast },\hat{f}_{m}\right) $, for every $%
m\in \mathcal{M}_{n}$, in order to derive oracle inequalities. We further set%
\begin{equation*}
\mathcal{K}_{m}=\mathcal{K}\left( f_{\ast },f_{m}\right) \text{ \ \ , \ \ }%
v_{m}=P\left[ \left( \frac{f_{m}}{f_{\ast }}\vee 1\right) \left( \ln \left( 
\frac{f_{m}}{f_{\ast }}\right) \right) ^{2}\right] \text{ \ \ and \ \ }%
w_{m}=P\left[ \left( \frac{f_{\ast }}{f_{m}}\vee 1\right) ^{r}\left( \ln
\left( \frac{f_{m}}{f_{\ast }}\right) \right) ^{2}\right] \text{ .}
\end{equation*}%
Let $\Omega _{n}$ be the event on which:

\begin{itemize}
\item For all models $m\in \mathcal{M}_{n}$, we set $z_{n}=\left( 2+\alpha _{%
\mathcal{M}}\right) \ln (n+1)+2\ln 2$ and we have,%
\begin{align}
\bar{\delta}\left( m\right) & \leq \sqrt{\frac{2v_{m}z_{n}}{n}}+\frac{2z_{n}%
}{n}  \label{line_3_bb} \\
-\bar{\delta}\left( m\right) & \leq \sqrt{\frac{2w_{m}z_{n}}{n}}+\frac{2z_{n}%
}{nr}  \label{line_4_bis}
\end{align}%
\begin{align}
-L_{\text{(\textbf{SA})},\alpha }\varepsilon _{n}^{-}\left( m\right) \frac{%
D_{m}}{2n}& \leq p_{1}\left( m\right) -\frac{D_{m}}{2n}\leq L_{\text{(%
\textbf{SA})},\alpha }\varepsilon _{n}^{+}\left( m\right) \frac{D_{m}}{2n}
\label{line_1_bb} \\
-L_{\text{(\textbf{SA})},\alpha }\varepsilon _{n}^{-}\left( m\right) \frac{%
D_{m}}{2n}& \leq p_{2}\left( m\right) -\frac{D_{m}}{2n}\leq L_{\text{(%
\textbf{SA})},\alpha }\varepsilon _{n}^{+}\left( m\right) \frac{D_{m}}{2n}
\label{line_2_bb}
\end{align}
\end{itemize}

\noindent By Theorem \ref{opt_bounds} applied with $\alpha =5+\alpha _{%
\mathcal{M}}$ and Propositions \ref{delta_bar_dev_droite_MLE} and \ref%
{delta_bar_dev_gauche_MLE} applied with $z=z_{n}$, we get 
\begin{eqnarray*}
\mathbb{P}\left( \Omega _{n}\right) &\geq &1-\sum_{m\in \mathcal{M}_{n}}%
\left[ 4(n+1)^{-5-\alpha _{\mathcal{M}}}+\frac{(n+1)^{-2-\alpha _{\mathcal{M}%
}}}{2}\right] \\
&=&1-\sum_{m\in \mathcal{M}_{n}}(n+1)^{-2-\alpha
_{\mathcal{M}}}\geq 1-(n+1)^{-2}\text{ }.
\end{eqnarray*}%
The following simple remark will be used along the proof: for any $m\in 
\mathcal{M}_{n}$, $z_{n}/n\leq L_{\text{(\textbf{SA})}}\varepsilon
_{n}^{+}\left( m\right) \frac{D_{m}}{n}$. Notice also that $\varepsilon
_{n}^{-}\left( \widehat{m}\right) \leq \varepsilon _{n}^{+}\left( \widehat{m}%
\right) $ (see Theorem \ref{opt_bounds}).

By using (\ref{margin_like_soft}), (\ref{formule_crit_prime}), (\ref%
{line_3_bb}), (\ref{line_1_bb}) and (\ref{line_2_bb}), we get that on $%
\Omega _{n}$, for $\Delta $ of the form $\left( \theta -1\right) _{-}L_{%
\text{(\textbf{SA})}}^{(1)}+L_{\text{(\textbf{SA})}}^{(2)}$ with $L_{\text{(%
\textbf{SA})}}^{(1)}$ and $L_{\text{(\textbf{SA})}}^{(2)}$ sufficiently
large,

\begin{align}
& \crit^{\prime }\left( \widehat{m}\right)  \notag \\
\geq & \mathcal{K}\left( f_{\ast },\hat{f}_{\widehat{m}}\right) +%
\pen%
\left( \widehat{m}\right) -p_{1}\left( \widehat{m}\right) -p_{2}\left( 
\widehat{m}\right) -\sqrt{\frac{2v_{\widehat{m}}z_{n}}{n}}-\frac{2z_{n}}{n} 
\notag \\
\geq & \mathcal{K}\left( f_{\ast },\hat{f}_{\widehat{m}}\right) +\left(
\theta -1\right) \frac{D_{\widehat{m}}-1}{n}+\left( \Delta -L_{\text{(%
\textbf{SA})}}\right) \varepsilon _{n}^{+}\left( \widehat{m}\right) \frac{D_{%
\widehat{m}}}{n}-\sqrt{\frac{2A_{MR,-}\mathcal{K}_{\widehat{m}}^{1-1/p}z_{n}%
}{n}}  \notag \\
\geq & \mathcal{K}\left( f_{\ast },\hat{f}_{\widehat{m}}\right) -2\left(
\theta -1\right) _{-}\left( \frac{D_{\widehat{m}}-1}{2n}-L_{\text{(\textbf{SA%
})}}\varepsilon _{n}^{-}\left( \widehat{m}\right) \frac{D_{\widehat{m}}}{n}%
\right)  \notag \\
& +\left( \Delta -\left( L_{\text{(\textbf{SA})}}^{(1)}\left( \theta
-1\right) _{-}+L_{\text{(\textbf{SA})}}^{(2)}\right) \right) \varepsilon
_{n}^{+}\left(\widehat{m}\right) \frac{D_{\widehat{m}}}{n}-\sqrt{\frac{2A_{MR,-}%
\mathcal{K}_{\widehat{m}}^{1-1/p}z_{n}}{n}}  \notag \\
\geq & \left( 1-2\left( \theta -1\right) _{-}\right) \mathcal{K}\left(
f_{\ast },\hat{f}_{\widehat{m}}\right) -\sqrt{\frac{2A_{MR,-}\mathcal{K}_{%
\widehat{m}}^{1-1/p}z_{n}}{n}}\text{ .}  \label{lower_crit_M_hat}
\end{align}%
Note that $1-2\left( \theta -1\right) _{-}>0$. Let us take $\eta \in \left(
0,1/2-\left( \theta -1\right) _{-}\right) $, so that 
\begin{equation*}
1-2\left( \theta -1\right) _{-}-\eta >1/2-\left( \theta -1\right) _{-}>0%
\text{ .}
\end{equation*}%
By Lemma \ref{prop_ineq_tech} applied with $a=\left( \eta \mathcal{K}_{%
\widehat{m}}\right) ^{\frac{p-1}{2p}}$, $b=\eta ^{-\frac{p-1}{2p}}\sqrt{%
2A_{MR,-}z_{n}/n}$, $u=\frac{2p}{p-1}$ and $v=\frac{2p}{p+1}$, we have 
\begin{equation*}
\sqrt{\frac{2A_{MR,-}\mathcal{K}_{\widehat{m}}^{1-1/p}z_{n}}{n}}\leq \eta 
\mathcal{K}_{\widehat{m}}+L_{\text{(\textbf{SA})},\alpha }\left( \frac{\ln
(n+1)}{n}\right) ^{\frac{p}{p+1}}\left( \frac{1}{\eta }\right) ^{\frac{p-1}{%
p+1}}\text{ .}
\end{equation*}%
By using the latter inequality in (\ref{lower_crit_M_hat}) we obtain,%
\begin{equation}
\crit^{\prime }\left( \widehat{m}\right) \geq \left( 1-2\left( \theta
-1\right) _{-}-\eta \right) \mathcal{K}\left( f_{\ast },\hat{f}_{\widehat{m}%
}\right) -L_{\text{(\textbf{SA})},\alpha }\left( \frac{\ln (n+1)}{n}\right)
^{\frac{p}{p+1}}\left( \frac{1}{\eta }\right) ^{\frac{p-1}{p+1}}\text{ .}
\label{lower_crit_M_hat_2}
\end{equation}%
We compute now an upper bound on $\crit^{\prime }$ for each model $m$. By
Lemma \ref{prop_ineq_tech} applied with $a=\left( \eta \mathcal{K}%
_{m}\right) ^{\frac{p-r-1}{2p}}$, $b=\eta ^{-\frac{p-1-r}{2p}}\sqrt{%
2A_{MR,-}z_{n}/n}$, $u=\frac{2p}{p-1-r}$ and $v=\frac{2p}{p+1+r}$, we have 
\begin{equation*}
\sqrt{\frac{2A_{MR,-}\mathcal{K}_{m}^{1-\frac{r+1}{p}}z_{n}}{n}}\leq \eta 
\mathcal{K}_{m}+L_{\text{(\textbf{SA})},r}\left( \frac{\ln (n+1)}{n}\right)
^{\frac{p}{p+1+r}}\left( \frac{1}{\eta }\right) ^{\frac{p-1-r}{p+1+r}}\text{
.}
\end{equation*}%
By (\ref{line_4_bis}), (\ref{formule_crit_prime}), (\ref{line_1_bb}), (\ref%
{line_2_bb}), (\ref{margin_like_gauche_soft}) and by using Lemma \ref%
{prop_ineq_tech} we have on $\Omega _{n}$,%
\begin{align*}
\crit^{\prime }& \left( m\right) =\mathcal{K}\left( f_{\ast },\hat{f}%
_{m}\right) +%
\pen%
\left( m\right) -p_{1}\left( m\right) -p_{2}\left( m\right) -\bar{\delta}%
\left( m\right) \\
\leq & \mathcal{K}\left( f_{\ast },\hat{f}_{m}\right) +2\left( \theta
-1\right) _{+}\left( \frac{D_{m}}{2n}-L_{\text{(\textbf{SA})}}\varepsilon
_{n}^{-}\left( m\right) \frac{D_{m}}{n}\right) \\
& +\left( \Delta +\left( \theta -1\right) _{+}L_{\text{(\textbf{SA})}%
}^{(1)}+L_{\text{(\textbf{SA})}}^{(2)}\right) \varepsilon _{n}^{+}\left(
m\right) \frac{D_{m}}{n}+\sqrt{\frac{2A_{MR,-}\mathcal{K}_{m}^{1-\frac{r+1}{p%
}}z_{n}}{n}}+\frac{2z_{n}}{nr} \\
\leq & \left( 1+2\left( \theta -1\right) _{+}+\eta \right) \mathcal{K}\left(
f_{\ast },\hat{f}_{m}\right) +\left( \Delta +\left( \theta -1\right) _{+}L_{%
\text{(\textbf{SA})}}^{(1)}+L_{\text{(\textbf{SA})},r}^{(2)}\right)
\varepsilon _{n}^{+}\left( m\right) \frac{D_{m}}{n} \\
& +L_{\text{(\textbf{SA})},r}\left( \frac{\ln (n+1)}{n}\right) ^{\frac{p}{%
p+1+r}}\left( \frac{1}{\eta }\right) ^{\frac{p-1-r}{p+1+r}}\text{ .}
\end{align*}%
Recall that we took $\Delta =L_{\text{(\textbf{SA})}}^{(1)}\left( \theta
-1\right) _{-}+L_{\text{(\textbf{SA})}}^{(2)}$ for some positive constants
sufficiently large, so%
\begin{equation*}
\Delta +\left( \theta -1\right) _{+}L_{\text{(\textbf{SA})}}^{(1)}+L_{\text{(%
\textbf{SA})},r}^{(2)}\leq \left\vert \theta -1\right\vert L_{\text{(\textbf{%
SA})}}^{(1)}+L_{\text{(\textbf{SA})},r}^{(2)}
\end{equation*}%
and we finally get,%
\begin{align}
\crit^{\prime }\left( m\right) \leq & \left( 1+2\left( \theta -1\right)
_{+}+\eta \right) \mathcal{K}\left( f_{\ast },\hat{f}_{m}\right) +\left(
\left\vert \theta -1\right\vert L_{\text{(\textbf{SA})}}^{(1)}+L_{\text{(%
\textbf{SA})},r}^{(2)}\right) \varepsilon _{n}^{+}\left( m\right) \frac{D_{m}%
}{n}  \label{upper_crit_prime_1} \\
& +L_{\text{(\textbf{SA})},r}\left( \frac{\ln (n+1)}{n}\right) ^{\frac{p}{%
p+1+r}}\left( \frac{1}{\eta }\right) ^{\frac{p-1-r}{p+1+r}}\text{ .}  \notag
\end{align}%
Now, as $\widehat{m}$ minimizes $%
\crit%
^{\prime }$ we get from (\ref{lower_crit_M_hat_2}) and (\ref%
{upper_crit_prime_1}), on $\Omega _{n}$,%
\begin{gather}
\mathcal{K}\left( f_{\ast },\hat{f}_{\widehat{m}}\right) \leq \frac{%
1+2\left( \theta -1\right) _{+}+\eta }{1-2\left( \theta -1\right) _{-}-\eta }%
\left( \mathcal{K}\left( f_{\ast },\hat{f}_{m_{\ast }}\right) +L_{\text{(%
\textbf{SA})},\theta ,r}\varepsilon _{n}^{+}\left( m_{\ast }\right) \frac{%
D_{m_{\ast }}}{n}\right)  \label{oracle_gene_0} \\
+L_{\text{(\textbf{SA})},\theta ,r}\left( \frac{\ln (n+1)}{n}\right) ^{\frac{%
p}{p+1+r}}\left( \frac{1}{\eta }\right) ^{\frac{p-1-r}{p+1+r}}\text{ .} 
\notag
\end{gather}%
We distinguish two cases. If $D_{m_{\ast }}\geq L_{\text{(\textbf{SA})}%
}\left( \ln (n+1)\right) ^{2}$ with a constant $L_{\text{(\textbf{SA})}}$
chosen such that 
\begin{equation*}
A_{0}\varepsilon _{n}^{+}\left( m_{\ast }\right) \leq \frac{1}{2\sqrt{\ln(n+1)}} \leq \frac{1}{2\sqrt{\ln
2}} < 1 \text{ ,}
\end{equation*}%
where $A_{0}$ and $\varepsilon _{n}^{+}\left( m\right) $ are defined in
Theorem \ref{opt_bounds}, then by Theorem \ref{opt_bounds} it holds on $%
\Omega _{n}$,%
\begin{equation*}
\mathcal{K}\left( f_{\ast },\hat{f}_{m_{\ast }}\right) \geq \left( 1-\frac{1}{2\sqrt{\ln(n+1)}}\right) \frac{D_{m_{\ast }}}{2n}\text{ .}
\end{equation*}%
On the other hand, if $D_{m_{\ast }}\leq L_{\text{(\textbf{SA})}}\left( \ln
(n+1)\right) ^{2}$ then by Theorem \ref{opt_bounds} we have on $\Omega _{n}$,%
\begin{equation*}
\mathcal{K}\left( f_{\ast },\hat{f}_{m_{\ast }}\right) +L_{\text{(\textbf{SA}%
)},\theta ,r}\varepsilon _{n}^{+}\left( m_{\ast }\right) \frac{D_{m_{\ast }}%
}{n}\leq L_{\text{(\textbf{SA})},\theta ,r}\frac{\left( \ln (n+1)\right) ^{3}%
}{n}\text{ .}
\end{equation*}%
Hence, in any case we always have on $\Omega _{n}$,%
\begin{align*}
& \mathcal{K}\left( f_{\ast },\hat{f}_{m_{\ast }}\right) +L_{\text{(\textbf{%
SA})},\theta ,r}\varepsilon _{n}^{+}\left( m_{\ast }\right) \frac{D_{m_{\ast
}}}{n}  \notag \\
\leq & L_{\text{(\textbf{SA})},\theta ,r}\frac{\left( \ln (n+1)\right) ^{3}}{%
n}+\left( 1+L_{\text{(\textbf{SA})},\theta ,r}\left( \ln (n+1)\right)
^{-1/2}\right) \mathcal{K}\left( f_{\ast },\hat{f}_{m_{\ast }}\right)\text{ .}
\end{align*}%
By taking $\eta=\left( \ln (n+1)\right) ^{-1/2}\left( 1/2-\left(\theta -1\right)_{-}\right) $ and using the fact that in this case,
\begin{equation*}
\frac{%
1+2\left( \theta -1\right) _{+}+\eta }{1-2\left( \theta -1\right) _{-}-\eta } \leq \frac{%
1+2\left( \theta -1\right) _{+}+L_{%
\theta}\eta }{1-2\left( \theta -1\right) _{-} }\text{ ,}
\end{equation*}
we deduce that Inequality (\ref{oracle_gene_0}) gives, %
\begin{equation}
\mathcal{K}\left( f_{\ast },\hat{f}_{\widehat{m}}\right) \leq \frac{%
1+2\left( \theta -1\right) _{+}+L_{\text{(\textbf{SA})},\theta ,r}\left( \ln
(n+1)\right) ^{-1/2}}{1-2\left( \theta -1\right) _{-}}\inf_{m\in \mathcal{M}%
_{n}}\left\{ \mathcal{K}\left( f_{\ast },\hat{f}_{m}\right) \right\} +L_{%
\text{(\textbf{SA})},\theta ,r}\frac{\left( \ln (n+1)\right) ^{\frac{3p-1-r}{%
2\left( p+1+r\right) }}}{n^{\frac{p}{p+1+r}}}  \notag
\end{equation}%
which is Inequality (\ref{oracle_gene_1}).

\begin{itemize}
\item \underline{Proof of Inequality (\ref{bound_unfavorable}):}
\end{itemize}

By Lemma \ref{lemma_oracle_model} we know that $D_{m_{\ast }}\leq L_{\text{(%
\textbf{SA})}}n^{\frac{1}{1+\beta _{+}}}$ on $\Omega _{n}$. Furthermore, we
have on $\Omega _{n}$, by simple computations,%
\begin{align*}
\mathcal{K}\left( f_{\ast },\hat{f}_{m}\right) =& \mathcal{K}\left( f_{\ast
},f_{m}\right) +\mathcal{K}\left( f_{m},\hat{f}_{m}\right) \\
\leq & C_{+}D_{m}^{-\beta _{+}}+\left( 1+L_{\text{(\textbf{SA})}}\varepsilon
_{n}^{+}\left( m\right) \right) \frac{D_{m}}{2n} \\
\leq & L_{\text{(\textbf{SA})}}\left( D_{m}^{-\beta _{+}}+\frac{D_{m}}{n}+%
\frac{\ln (n+1)}{n}\right) \text{ .}
\end{align*}%
This yields%
\begin{align*}
\inf_{m\in \mathcal{M}_{n}}\left\{ \mathcal{K}\left( f_{\ast },\hat{f}%
_{m}\right) \right\} \leq & L_{\text{(\textbf{SA})}}\inf \left\{
D_{m}^{-\beta _{+}}+\frac{D_{m}}{n}+\frac{\ln (n+1)}{n}\text{ };\text{ }m\in 
\mathcal{M}_{n},\text{ }D_{m}\leq L_{\text{(\textbf{SA})}}n^{\frac{1}{%
1+\beta _{+}}}\right\} \\
\leq & L_{\text{(\textbf{SA})}}n^{-\frac{\beta _{+}}{1+\beta _{+}}}\text{ .}
\end{align*}%
To conclude, it suffices to notice that if $\beta _{+}>p$ then%
\begin{equation*}
n^{-\frac{\beta _{+}}{1+\beta _{+}}}\leq n^{-\frac{p}{p+1}}\leq n^{-\frac{p}{%
p+1+r}}\text{ ,}
\end{equation*}%
which finally gives%
\begin{equation*}
\frac{1+2\left( \theta -1\right) _{+}+L_{\text{(\textbf{SA})},\theta
,r}\left( \ln (n+1)\right) ^{-1/2}}{1-2\left( \theta -1\right) _{-}}%
\inf_{M\in \mathcal{M}_{n}}\left\{ \mathcal{K}\left( f_{\ast },\hat{f}%
_{m}\right) \right\} \leq L_{\text{(\textbf{SA})},\theta ,r}\frac{\left( \ln
(n+1)\right) ^{\frac{3p-1-r}{2\left( p+1+r\right) }}}{n^{\frac{p}{p+1+r}}}%
\text{ .}
\end{equation*}

\begin{itemize}
\item \underline{Proof of Inequality (\ref{oracle_um_ap}):}
\end{itemize}

From \textbf{(Ap), }we know by Lemma \ref{lemma_oracle_model} that there
exist $L_{\text{(\textbf{SA})}}^{(1)},L_{\text{(\textbf{SA})}}^{(2)}>0$ such
that 
\begin{equation*}
L_{\text{(\textbf{SA})}}^{(1)}n^{\frac{\beta _{+}}{\left( 1+\beta
_{+}\right) \beta _{-}}}\leq D_{M_{\ast }}\leq L_{\text{(\textbf{SA})}%
}^{(2)}n^{\frac{1}{1+\beta _{+}}}
\end{equation*}%
and so%
\begin{align}
\varepsilon _{n}^{+}\left( m_{\ast }\right) \leq & \max \left\{ \sqrt{\frac{%
D_{m_{\ast }}\ln (n+1)}{n}};\sqrt{\frac{\ln (n+1)}{D_{m_{\ast }}}};\frac{\ln
(n+1)}{D_{m_{\ast }}}\right\}  \notag \\
\leq & L_{\text{(\textbf{SA})}}\max \left\{ n^{-\frac{\beta _{+}}{2\left(
1+\beta _{+}\right) }};n^{-\frac{\beta _{+}}{2\left( 1+\beta _{+}\right)
\beta _{-}}}\right\} \sqrt{\ln (n+1)}  \notag \\
=& L_{\text{(\textbf{SA})}}n^{-\frac{\beta _{+}}{\left( 1+\beta _{+}\right)
\beta _{-}}}\sqrt{\ln (n+1)}\text{ .}  \label{upper_bound_epsilon}
\end{align}%
Assume for now that we also have 
\begin{equation}
A_{0}\varepsilon _{n}^{+}\left( m_{\ast }\right) \leq \left( \ln
(n+1)\right) ^{-1/2}\text{ ,}  \label{borne_epsilon_n_0}
\end{equation}%
where $A_{0}$ and $\varepsilon _{n}^{+}\left( m\right) $ are defined in
Theorem \ref{opt_bounds}. Then by Theorem \ref{opt_bounds} it holds on $%
\Omega _{n}$,%
\begin{equation*}
\mathcal{K}\left( f_{\ast },\hat{f}_{m_{\ast }}\right) \geq \left( 1-\left(
\ln (n+1)\right) ^{-1/2}\right) \frac{D_{m_{\ast }}}{2n}\text{ .}
\end{equation*}%
In this case, we deduce that%
\begin{align}
& \mathcal{K}\left( f_{\ast },\hat{f}_{m_{\ast }}\right) +L_{\text{(\textbf{%
SA})},\theta ,r}\varepsilon _{n}^{+}\left( m_{\ast }\right) \frac{D_{m_{\ast
}}}{n}  \notag \\
\leq & \left( 1+L_{\text{(\textbf{SA})},\theta ,r}n^{-\frac{\beta _{+}}{%
\left( 1+\beta _{+}\right) \beta _{-}}}\sqrt{\ln (n+1)}\right) \mathcal{K}%
\left( f_{\ast },\hat{f}_{m_{\ast }}\right) \text{ ,}
\label{borne_sup_K_plus_dev}
\end{align}%
and Inequality (\ref{oracle_um_ap}) simply follows from using Inequality (%
\ref{oracle_gene_0}).

If Inequality (\ref{borne_epsilon_n_0}) is not satisfied, that is%
\begin{equation}
A_{0}\varepsilon _{n}^{+}\left( m_{\ast }\right) >\left( \ln (n+1)\right)
^{-1/2}\text{ ,}  \label{negation_epsilon}
\end{equation}%
then by (\ref{upper_bound_epsilon}), this means that there exists a positive
constant $L_{\text{(\textbf{SA})}}$ such that%
\begin{equation*}
L_{\text{(\textbf{SA})}}n^{-\frac{\beta _{+}}{\left( 1+\beta _{+}\right)
\beta _{-}}}\sqrt{\ln (n+1)}>\left( \ln (n+1)\right) ^{-1/2}\text{ .}
\end{equation*}%
Consequently, this ensures that in the case where (\ref{negation_epsilon})
is true, we also have $n\leq n_{0}\left( \text{(\textbf{SA})}\right) $.
Hence, as 
\begin{equation*}
\mathcal{K}\left( f_{\ast },\hat{f}_{m_{\ast }}\right) \geq C_{-}D_{m_{\ast
}}^{-\beta _{-}}\geq L_{\text{(\textbf{SA})}}n^{\frac{\beta _{+}}{\left(
1+\beta _{+}\right) \beta _{-}}}>0\text{ ,}
\end{equation*}%
this yields Inequality (\ref{borne_sup_K_plus_dev}) with a positive constant 
$L_{\text{(\textbf{SA})},\theta ,r}$ in the right-hand term sufficiently
large and then the result easily follows from using Inequality (\ref%
{oracle_gene_0}).

\begin{itemize}
\item \underline{Proof of Inequality (\ref{oracle_um_ap_opt}):}
\end{itemize}

If $D_{m_{\ast }}\geq L_{\text{(\textbf{SA})}}\left( \ln (n+1)\right) ^{2}$
with a constant $L_{\text{(\textbf{SA})}}$ chosen such that 
\begin{equation*}
A_{0}\varepsilon _{n}^{+}\left( m_{\ast }\right) \leq 1/2\text{ ,}
\end{equation*}%
where $A_{0}$ and $\varepsilon _{n}^{+}\left( m\right) $ are defined in
Theorem \ref{opt_bounds}, then by Theorem \ref{opt_bounds} it holds on $%
\Omega _{n}$,%
\begin{equation*}
\mathcal{K}\left( f_{\ast },\hat{f}_{m_{\ast }}\right) \geq C_{-}D_{m_{\ast
}}^{-\beta _{-}}+\frac{D_{m_{\ast }}}{4n}\text{ .}
\end{equation*}%
By Lemma \ref{lemma_oracle_model} we know that $L_{\text{(\textbf{SA})}%
}^{(1)}n^{\frac{\beta _{+}}{\left( 1+\beta _{+}\right) \beta _{-}}}\leq
D_{m_{\ast }}\leq L_{\text{(\textbf{SA})}}^{(2)}n^{\frac{1}{1+\beta _{+}}}$
on $\Omega _{n}$. This gives%
\begin{equation*}
\mathcal{K}\left( f_{\ast },\hat{f}_{m_{\ast }}\right) \geq L_{\text{(%
\textbf{SA})}}n^{-\frac{\beta _{-}}{1+\beta _{+}}}+L_{\text{(\textbf{SA})}%
}n^{-1+\frac{\beta _{+}}{\left( 1+\beta _{+}\right) \beta _{-}}}
\end{equation*}%
and we deduce by simple algebra that if $\beta _{-}<p\left( 1+\beta
_{+}\right) /(1+p+r)$ or $p/(1+r)>\beta _{+}/(\beta _{-}\left( 1+\beta
_{+}\right) )-1$, then%
\begin{equation*}
L_{\text{(\textbf{SA})}}\left( \ln (n+1)\right) ^{-1/2}\mathcal{K}\left(
f_{\ast },\hat{f}_{m_{\ast }}\right) \geq \frac{\left( \ln (n+1)\right) ^{%
\frac{3p-1-r}{2\left( p+1+r\right) }}}{n^{\frac{p}{p+1+r}}}\text{ .}
\end{equation*}%
On the other hand, if $D_{m_{\ast }}\leq L_{\text{(\textbf{SA})}}\left( \ln
(n+1)\right) ^{2}$, this implies in particular%
\begin{equation*}
L_{\text{(\textbf{SA})}}^{(1)}n^{\frac{\beta _{+}}{\left( 1+\beta
_{+}\right) \beta _{-}}}\leq L_{\text{(\textbf{SA})}}\left( \ln (n+1)\right)
^{2}\text{ .}
\end{equation*}%
Hence, there exists an integer $n_{0}($(\textbf{SA})$)$ such that $n\leq
n_{0}($(\textbf{SA})$)$. In this case, we can find a constant $L_{\text{(%
\textbf{SA})}}$ such that%
\begin{equation*}
L_{\text{(\textbf{SA})}}\left( \ln (n+1)\right) ^{-1/2}\mathcal{K}\left(
f_{\ast },\hat{f}_{m_{\ast }}\right) \geq \frac{\left( \ln (n+1)\right) ^{%
\frac{3p-1-r}{2\left( p+1+r\right) }}}{n^{\frac{p}{p+1+r}}}
\end{equation*}%
and through the use of inequality (\ref{oracle_um_ap}), this conclude the
proof of Inequality (\ref{oracle_um_ap_opt}).

\begin{lemma}[Control on the dimension of the selected model]
\label{lemma_selected_model}Assume that (\textbf{SA}$_{0}$) holds together
with (\textbf{Ap}$_{u}$). If $\beta _{+}\leq \frac{p}{r+1}$ then, on the
event $\Omega _{n}$ defined in the proof of Theorem \ref%
{theorem_opt_pen_MLE_gene}, we have%
\begin{equation}
D_{\widehat{m}}\leq L_{\Delta ,\theta ,r,\text{(\textbf{SA})}}n^{^{\frac{1}{%
2+\beta _{+}\left( 1-\frac{r+1}{p}\right) }}}\sqrt{\ln (n+1)}\text{ .}
\label{control_dim_hat_gene}
\end{equation}%
If moreover (\textbf{Ap}) holds, then we get on the event $\Omega _{n}$, 
\begin{equation}
L_{\Delta ,\theta ,\text{(\textbf{SA})}}^{(1)}\frac{n^{\frac{\beta _{+}}{%
\beta _{-}\left( 1+\beta _{+}\right) }}}{\left( \ln (n+1)\right) ^{\frac{1}{%
\beta _{-}}}}\leq D_{\widehat{m}}\leq L_{\Delta ,\theta ,\text{(\textbf{SA})}%
}^{(2)}n^{^{\frac{1}{2+\beta _{+}\left( 1-\frac{r+1}{p}\right) }}}\sqrt{\ln
(n+1)}\text{ .}  \label{control_dim_hat_Ap}
\end{equation}
\end{lemma}

\begin{lemma}[Control over the dimension of oracle models]
\label{lemma_oracle_model}Assume that (\textbf{SA}$_{0}$) holds together
with (\textbf{Ap}$_{u}$). We have on the event $\Omega _{n}$ defined in the
proof of Theorem \ref{theorem_opt_pen_MLE_gene},%
\begin{equation*}
D_{m_{\ast }}\leq L_{\text{(\textbf{SA})}}n^{\frac{1}{1+\beta _{+}}}\text{ .}
\end{equation*}%
If moreover (\textbf{Ap}) holds, then we get on the event $\Omega _{n}$,%
\begin{equation}
L_{\text{(\textbf{SA})}}^{(1)}n^{\frac{\beta _{+}}{\left( 1+\beta
_{+}\right) \beta _{-}}}\leq D_{m_{\ast }}\leq L_{\text{(\textbf{SA})}%
}^{(2)}n^{\frac{1}{1+\beta _{+}}}\text{ .}  \label{control_dim_star_Ap}
\end{equation}
\end{lemma}

\noindent \textbf{Proof of} \textbf{Lemma \ref{lemma_selected_model}}.
Recall that $\widehat{m}$ minimizes 
\begin{equation*}
\crit^{\prime }\left( m\right) =\crit\left( m\right) -P_{n}\gamma \left(
f_{\ast }\right) =\mathcal{K}_{m}-p_{2}\left( m\right) -\bar{\delta}\left(
m\right) +\pen\left( m\right)  
\end{equation*}%
over the models $m\in \mathcal{M}_{n}.$ Moreover, $\pen\left( m\right)
=\left( \theta +\Delta \varepsilon _{n}^{+}\left( m\right) \right) D_{m}/n$.
The analysis is restricted on $\Omega _{n}$.

\begin{enumerate}
\item Upper bound on $\crit^{\prime }\left( m\right) $: 
\begin{align*}
p_{2}\left( m\right) & \geq \left( \frac{1}{2}-L_{\text{(\textbf{SA})}%
}\varepsilon _{n}^{+}\left( m\right) \right) \frac{D_{m}}{n} \\
-\bar{\delta}\left( m\right) & \leq \sqrt{\frac{2w_{m}z_{n}}{n}}+\frac{2z_{n}%
}{nr}\text{ .}
\end{align*}%
Moreover, by Lemma \ref{lem:margin_like_unbounded}, we have $w_{m}\leq
A_{MR,-}\mathcal{K}_{m}^{1-\frac{r+1}{p}}$ and so,%
\begin{align*}
\crit^{\prime }\left( m\right) \leq & \mathcal{K}_{m}+\left( \theta -\frac{1%
}{2}+L_{\Delta ,\text{(\textbf{SA})}}\varepsilon _{n}^{+}\left( m\right)
\right) \frac{D_{m}}{n}+\sqrt{\frac{2A_{MR,-}\mathcal{K}_{m}^{1-\frac{r+1}{p}%
}z_{n}}{n}} \\
\leq & L_{\Delta ,\theta ,\text{(\textbf{SA})}}\left( D_{m}^{-\beta _{+}}+%
\frac{D_{m}}{n}+\frac{\ln (n+1)}{n}+\sqrt{\frac{D_{m}^{-\beta _{+}\left( 1-%
\frac{r+1}{p}\right) }\ln (n+1)}{n}}\right) \text{ .}
\end{align*}%
Now, if $\beta _{+}\leq \frac{p}{r+1}$, then for $m_{0}$ such that $%
D_{m_{0}}=\left\lceil n^{\frac{1}{1+\beta _{+}}}\right\rceil $ we have%
\begin{equation*}
\frac{D_{m_{0}}}{n}\leq 2n^{-\frac{\beta _{+}}{1+\beta _{+}}}\text{ };\text{ 
}D_{m_{0}}^{-\beta _{+}}\leq n^{-\frac{\beta _{+}}{1+\beta _{+}}}\text{ };%
\text{ }\sqrt{\frac{D_{m_{0}}^{-\beta _{+}\left( 1-\frac{r+1}{p}\right) }}{n}%
}\leq n^{-\frac{\beta _{+}\left( 2-\frac{r+1}{p}\right) +1}{2\left( 1+\beta
_{+}\right) }}\leq n^{-\frac{\beta _{+}}{1+\beta _{+}}}\text{ ,}
\end{equation*}%
so we get%
\begin{equation}
\crit^{\prime }\left( m_{0}\right) \leq L_{\Delta ,\theta ,\text{(\textbf{SA}%
)}}n^{-\frac{\beta _{+}}{1+\beta _{+}}}\sqrt{\ln (n+1)}\text{ .}
\label{upper_crit_M_0}
\end{equation}%
Otherwise, if $\beta _{+}>\frac{p}{r+1}$, then for $m_{1}$ such that $%
D_{m_{1}}=\left\lceil n^{\frac{1}{2+\beta _{+}\left( 1-\frac{r+1}{p}\right) }%
}\right\rceil $, we have%
\begin{align*}
\frac{D_{m_{1}}}{n}& \leq 2n^{-\frac{1+\beta _{+}\left( 1-\frac{r+1}{p}%
\right) }{2+\beta _{+}\left( 1-\frac{r+1}{p}\right) }}\text{ };\text{ }\sqrt{%
\frac{D_{m_{1}}^{-\beta _{+}\left( 1-\frac{r+1}{p}\right) }}{n}}\leq n^{-%
\frac{1+\beta _{+}\left( 1-\frac{r+1}{p}\right) }{2+\beta _{+}\left( 1-\frac{%
r+1}{p}\right) }}\text{ };\text{ } \\
D_{m_{1}}^{-\beta _{+}}& \leq n^{-\frac{\beta _{+}}{2+\beta _{+}\left( 1-%
\frac{r+1}{p}\right) }}\leq n^{-\frac{1+\beta _{+}\left( 1-\frac{r+1}{p}%
\right) }{2+\beta _{+}\left( 1-\frac{r+1}{p}\right) }}
\end{align*}%
and 
\begin{equation}
\crit^{\prime }\left( m_{1}\right) \leq L_{\Delta ,\theta ,\text{(\textbf{SA}%
)}}n^{-\frac{1+\beta _{+}\left( 1-\frac{r+1}{p}\right) }{2+\beta _{+}\left(
1-\frac{r+1}{p}\right) }}\sqrt{\ln (n+1)}\text{ .}  \label{upper_crit_M_1}
\end{equation}

\item Lower bound on $\crit^{\prime }\left( m\right) $: we have%
\begin{align*}
p_{2}\left( m\right) & \leq \left( \frac{1}{2}+L_{\text{(\textbf{SA})}%
}\varepsilon _{n}^{+}\left( m\right) \right) \frac{D_{m}}{n} \\
-\bar{\delta}\left( m\right) & \geq -\sqrt{\frac{2v_{m}z_{n}}{n}}-\frac{%
2z_{n}}{n}\text{ .}
\end{align*}%
Moreover, by Lemma \ref{lem:margin_like_unbounded}, we have for some
constant $A_{MR,-}>0$, $v_{m}\leq A_{MR,-}\mathcal{K}_{m}^{1-\frac{1}{p}}$.
For $\Delta $ large enough we thus get,%
\begin{equation}
\crit^{\prime }\left( m\right) \geq \mathcal{K}_{m}+\left( \theta -\frac{1}{2%
}\right) \frac{D_{m}}{n}-\sqrt{\frac{2A_{MR,-}\mathcal{K}_{m}^{1-\frac{1}{p}%
}z_{n}}{n}}  \label{lower_crit_1}
\end{equation}%
Assume that $\beta _{+}\leq \frac{p}{r+1}$. We take $D_{m}\leq L\left( n/\ln
(n+1)\right) ^{\frac{p}{\left( p+1\right) \beta _{-}}}$ for some constant $%
L>0$. If $L$ is small enough, we have by (\textbf{Ap}) $\mathcal{K}_{m}\geq
\left( 8A_{MR,-}z_{n}/n\right) ^{p/\left( p+1\right) }$ and by (\ref%
{lower_crit_1}), $\crit^{\prime }\left( m\right) \geq \mathcal{K}_{m}/2$.
Now if%
\begin{equation}
D_{m}\leq L\left( \left( \frac{n}{\ln (n+1)}\right) ^{\frac{p}{\left(
1+p\right) \beta _{-}}}\wedge \frac{n^{\frac{\beta _{+}}{\beta _{-}\left(
1+\beta _{+}\right) }}}{\left( \ln (n+1)\right) ^{\frac{1}{2\beta _{-}}}}%
\right)  \label{upper_D_M_hat}
\end{equation}%
with $L$ sufficiently small, only depending on $\Delta ,\theta $ and
constants in (\textbf{SA}), then by (\ref{upper_crit_M_0}) we obtain $\crit%
^{\prime }\left( m\right) >\crit^{\prime }\left( m_{0}\right) $. As $\beta
_{+}\leq \frac{p}{r+1}<p$, the upper bound in (\ref{upper_D_M_hat}) reduces
to $D_{m}\leq Ln^{\frac{\beta _{+}}{\beta _{-}\left( 1+\beta _{+}\right) }%
}/\left( \ln (n+1)\right) ^{\frac{1}{\beta _{-}}}$. This proves the
left-hand side of (\ref{control_dim_hat_Ap}).

\noindent Assume now that $\beta _{+}>\frac{p}{r+1}$. If%
\begin{equation*}
D_{m}\leq L\left( \left( \frac{n}{\ln (n+1)}\right) ^{\frac{p}{\left(
1+p\right) \beta _{-}}}\wedge \frac{n^{\frac{1+\beta _{+}\left( 1-\frac{r+1}{%
p}\right) }{\beta _{-}\left( 2+\beta _{+}\left( 1-\frac{r+1}{p}\right)
\right) }}}{\left( \ln (n+1)\right) ^{\frac{1}{2\beta _{-}}}}\right)
\end{equation*}%
with $L$ sufficiently small, only depending on $\Delta ,\theta $ and
constants in (\textbf{SA}), then by (\ref{upper_crit_M_1}) we obtain $\crit%
^{\prime }\left( m\right) >\crit^{\prime }\left( m_{1}\right) $. As $\beta
_{+}\leq \frac{p}{r+1}<p$, the upper bound in (\ref{upper_D_M_hat}) reduces
to $D_{m}\leq Ln^{\frac{\beta _{+}}{\beta _{-}\left( 1+\beta _{+}\right) }%
}/\left( \ln (n+1)\right) ^{\frac{1}{\beta _{-}}}$. This proves the
left-hand side of (\ref{control_dim_hat_Ap}).

\noindent Assume that $\beta _{+}\leq \frac{p}{r+1}$. We take $D_{m}\leq
L\left( n\ln (n+1)\right) ^{\frac{1}{2+\beta _{+}\left( 1-\frac{1}{p}\right) 
}}$ for some constant $L>0$. If $L$ is large enough, then we get by (\ref%
{lower_crit_1}) and simple calculations, $\crit^{\prime }\left( m\right)
\geq \left( \theta /2-1/4\right) D_{m}/n$. Furthermore, if%
\begin{equation}
D_{m}\geq L\left( \left( n\ln (n+1)\right) ^{\frac{1}{2+\beta _{+}\left( 1-%
\frac{1}{p}\right) }}\wedge n^{\frac{1}{1+\beta _{+}}}\sqrt{\ln (n+1)}\right)
\label{lower_D_M_hat_1}
\end{equation}%
with $L$ sufficiently small, only depending on $\Delta ,\theta $ and
constants in (\textbf{SA}), then by (\ref{upper_crit_M_0}) we obtain $\crit%
^{\prime }\left( m\right) >\crit^{\prime }\left( m_{0}\right) $. As $\beta
_{+}\leq \frac{p}{r+1}<p$, the lower bound in (\ref{lower_D_M_hat_1})
reduces to $D_{m}\leq L\left( n\ln (n+1)\right) ^{\frac{1}{2+\beta
_{+}\left( 1-\frac{1}{p}\right) }}$. This proves the left-hand side of (\ref%
{control_dim_hat_Ap}).

\noindent Assume now that $\beta _{+}>\frac{p}{r+1}$. If%
\begin{equation}
D_{m}\geq L\left( \left( n\ln (n+1)\right) ^{\frac{1}{2+\beta _{+}\left( 1-%
\frac{1}{p}\right) }}\vee n^{^{\frac{1}{2+\beta _{+}\left( 1-\frac{r+1}{p}%
\right) }}}\sqrt{\ln (n+1)}\right)  \label{lower_D_M_hat}
\end{equation}%
with $L$ sufficiently large, only depending on $\Delta ,\theta $ and
constants in (\textbf{SA}), then by (\ref{upper_crit_M_0}) we obtain $\crit%
^{\prime }\left( m\right) >\crit^{\prime }\left( m_{1}\right) $. As $\beta
_{+}\leq \frac{p}{r+1}<p$, the upper bound in (\ref{upper_D_M_hat}) reduces
to $D_{m}\leq Ln^{\frac{\beta _{+}}{\beta _{-}\left( 1+\beta _{+}\right) }%
}/\left( \ln (n+1)\right) ^{\frac{1}{\beta _{-}}}$. This proves the
left-hand side of (\ref{control_dim_hat_Ap}). As $r>0$, (\ref{lower_D_M_hat}%
) reduces to $D_{m}\geq Ln^{^{\frac{1}{2+\beta _{+}\left( 1-\frac{r+1}{p}%
\right) }}}\sqrt{\ln (n+1)}$, which proves (\ref{control_dim_hat_gene}) and
the right-hand side of (\ref{control_dim_hat_Ap}).
\end{enumerate}

\noindent \textbf{Proof of} \textbf{Lemma \ref{lemma_oracle_model}}. By
definition, $m_{\ast }$ minimizes 
\begin{equation*}
\mathcal{K}\left( f_{\ast },\hat{f}_{m}\right) =\mathcal{K}_{m}+p_{1}\left(
m\right)
\end{equation*}%
over the models $m\in \mathcal{M}_{n}.$ The analysis is restricted on $%
\Omega _{n}$.

\begin{enumerate}
\item Upper bound on $\mathcal{K}\left( f_{\ast },\hat{f}_{m}\right) $: we
have 
\begin{align*}
\mathcal{K}\left( f_{\ast },\hat{f}_{m}\right) \leq & C_{+}D_{m}^{-\beta
_{+}}+\left( \frac{1}{2}+L_{\text{(\textbf{SA})}}\varepsilon _{n}^{+}\left(
m\right) \right) \frac{D_{m}}{n}  \notag \\
\leq & L_{\text{(\textbf{SA})}}\left( D_{m}^{-\beta _{+}}+\frac{D_{m}}{n}+%
\frac{\ln (n+1)}{n}\right) \text{ .}  
\end{align*}%
Hence, if $m_{0}$ is such that $D_{m_{0}}=n^{\frac{1}{1+\beta _{+}}}$, then%
\begin{equation}
\mathcal{K}\left( f_{\ast },\hat{f}_{m_{0}}\right) \leq L_{\text{(\textbf{SA}%
)}}n^{-\frac{\beta _{+}}{1+\beta _{+}}}\text{ .}  \label{upper_crit_oracle_2}
\end{equation}

\item Lower bound on $\mathcal{K}\left( f_{\ast },\hat{f}_{m}\right) $:
there exists a constant $A_{0}$, only depending on constants in (\textbf{SA}%
), such that%
\begin{align}
\mathcal{K}\left( f_{\ast },\hat{f}_{m}\right) \geq & \mathcal{K}_{m}+\left( 
\frac{1}{2}-A_{0}\varepsilon _{n}^{-}\left( m\right) \right) \frac{D_{m}}{n}
\notag \\
\geq & \mathcal{K}_{m}+\frac{D_{m}}{2n}-A_{0}\max \left\{ \left( \frac{D_{m}%
}{n}\right) ^{3/2}\sqrt{\ln (n+1)};\frac{\sqrt{D_{m}\ln (n+1)}}{n}\right\} 
\text{ .}  \label{lower_crit_oracle}
\end{align}%
If (\textbf{Ap}) holds, then for 
\begin{equation*}
D_{m}\leq L_{\text{(\textbf{SA})}}\min \left\{ \frac{n^{\frac{3}{3+2\beta
_{-}}}}{\left( \ln (n+1)\right) ^{\frac{1}{3+2\beta _{-}}}};\frac{n^{\frac{2%
}{1+2\beta _{-}}}}{\left( \ln (n+1)\right) ^{\frac{1}{1+2\beta _{-}}}}%
\right\}
\end{equation*}%
with $L_{\text{(\textbf{SA})}}$ sufficiently small, we have 
\begin{equation*}
\mathcal{K}_{m}/2\geq C_{-}D_{m}^{-\beta _{-}}/2\geq A_{0}\max \left\{
\left( \frac{D_{m}}{n}\right) ^{3/2}\sqrt{\ln (n+1)};\frac{\sqrt{D_{m}\ln
(n+1)}}{n}\right\} \text{ .}
\end{equation*}%
In this case, we have by (\ref{lower_crit_oracle}), $\mathcal{K}\left(
f_{\ast },\hat{f}_{m}\right) \geq \mathcal{K}_{m}/2+\left( D_{m}\right)
/2n\geq C_{-}D_{m}^{-\beta _{-}}/2$. Moreover, if $m$ is such that $%
D_{m}\leq L_{\text{(\textbf{SA})}}n^{\frac{\beta _{+}}{\left( 1+\beta
_{+}\right) \beta _{-}}}$ with $L_{\text{(\textbf{SA})}}$ sufficiently
small, we also have%
\begin{equation*}
D_{m}\leq L_{\text{(\textbf{SA})}}n^{\frac{\beta _{+}}{\left( 1+\beta
_{+}\right) \beta _{-}}}\leq L_{\text{(\textbf{SA})}}\min \left\{ n^{\frac{%
\beta _{+}}{\left( 1+\beta _{+}\right) \beta _{-}}};\frac{n^{\frac{3}{%
3+2\beta _{-}}}}{\left( \ln (n+1)\right) ^{\frac{1}{3+2\beta _{-}}}};\frac{%
n^{\frac{2}{1+2\beta _{-}}}}{\left( \ln (n+1)\right) ^{\frac{1}{1+2\beta _{-}%
}}}\right\}
\end{equation*}%
and by (\ref{upper_crit_oracle_2}) we get $\mathcal{K}\left( f_{\ast },\hat{f%
}_{m_{0}}\right) <\mathcal{K}\left( f_{\ast },\hat{f}_{m}\right) $, which
gives the left-hand side of (\ref{control_dim_star_Ap}).

\noindent We turn now to the proof of the right-hand side of (\ref%
{control_dim_star_Ap}). Let $m\in \mathcal{m}_{n}$ be such that $D_{m}\geq
L_{1}n^{\frac{1}{1+\beta _{+}}}$. By (\ref{upper_crit_oracle_2}) we deduce
that if $L_{1}$ is large enough, depending only on constants in (\textbf{SA}%
), then we have 
\begin{equation*}
\frac{D_{m}}{4n}>\mathcal{K}\left( f_{\ast },\hat{f}_{m_{0}}\right) \text{ .}
\end{equation*}%
In addition, if $D_{m}\geq L_{2}\left( \ln (n+1)\right) ^{2}$ and $D_{m}\leq
L_{2}^{-1}n/\ln (n+1)$ for some constant $L_{2}$ sufficiently large, then 
\begin{equation*}
\frac{D_{m}}{4n}\geq A_{0}\max \left\{ \left( \frac{D_{m}}{n}\right) ^{3/2}%
\sqrt{\ln (n+1)};\frac{\sqrt{D_{m}\ln (n+1)}}{n}\right\}
\end{equation*}%
and by (\ref{lower_crit_oracle}), we deduce that $\mathcal{K}\left( f_{\ast
},\hat{f}_{m}\right) >\mathcal{K}\left( f_{\ast },\hat{f}_{m_{0}}\right) $.
The latter inequality implies that $D_{m_{\ast }}\leq L_{1}n^{\frac{1}{%
1+\beta _{+}}}$. Our reasoning is valid if $n$ is such that $L_{2}\left( \ln
(n+1)\right) ^{2}\leq L_{1}n^{\frac{1}{1+\beta _{+}}}\leq L_{2}^{-1}n/\ln
(n+1)$. At the price of enlarging $L_{1}$, we can always achieve $%
L_{2}\left( \ln (n+1)\right) ^{2}\leq L_{1}n^{\frac{1}{1+\beta _{+}}}$, with 
$L_{1}$ not depending on $n$. Then if $L_{2}^{-1}n/\ln (n+1)<L_{1}n^{\frac{1%
}{1+\beta _{+}}}$, we still have%
\begin{equation*}
D_{m_{\ast }}\leq \max_{m\in \mathcal{M}_{n}}D_{m}\leq A_{\mathcal{M},+}%
\frac{n}{\left( \ln (n+1)\right) ^{2}}\leq A_{\mathcal{M},+}L_{2}L_{1}n^{%
\frac{1}{1+\beta _{+}}}\text{ .}
\end{equation*}%
In every case, there exists $L>0$ only depending on constants in (\textbf{SA}%
) such that $D_{m_{\ast }}\leq Ln^{\frac{1}{1+\beta _{+}}}$.
\end{enumerate}

\begin{lemma}
\label{prop_ineq_tech}Let $\left( a,b\right) \in \mathbb{R}_{+}^{2}$ and $%
\left( u,v\right) \in \left[ 1,\infty \right] $ such that $1/u+1/v=1$. Then%
\begin{equation*}
ab\leq \max \left\{ a^{u};b^{v}\right\} \leq a^{u}+b^{v}\text{ .}
\end{equation*}
\end{lemma}

\begin{proof}
By symmetry, we can assume $a^{u}\geq b^{v}$. Then $b=\left( b^{v}\right)
^{1/v}\leq a^{u/v}=a^{u-1}$ and so, $ab\leq aa^{u-1}=a^{u}\leq a^{u}+b^{v}$.
\end{proof}

\end{document}